\documentclass{amsart}
\usepackage[utf8]{inputenc}
\usepackage{pstricks}
\usepackage{amsmath}
\usepackage{amssymb}
\usepackage{amsthm}
\usepackage{bm}
\usepackage{tikz-cd}
\usepackage{caption}
\usepackage{subcaption}
\usepackage{xcolor}
\usepackage{hyperref}
\usepackage{graphicx} 

\newcommand{\added}[1]{{{#1}}}

\newcommand{\C}{\mathbb{C}}

\newcommand{\R}{\mathbb{R}}
\newcommand{\Z}{\mathbb{Z}}

\newcommand{\tf}{\mathfrak{t}}
\newcommand{\sfk}{\mathfrak{s}}
\newcommand{\bou}{\partial}
\DeclareMathOperator{\Spin}{Spin}
\newcommand{\Gc}{\mathcal{G}}

\newcommand{\Spinc}{\(\Spin^c\,\)}
\newcommand{\SpincX}[1]{\Spin^c( #1 )}
\newcommand{\Char}{\mbox{Char}}
\newcommand{\PD}{\mbox{PD}}
\newcommand{\CFb}{\mathbb{CF}}
\newcommand{\CFKb}{\mathbb{CFK}}
\newcommand{\Q}{\mathbb{Q}}
\newcommand{\HFb}{\mathbb{HF}}
\newcommand{\Fc}{\mathcal{F}}
\newcommand{\Qc}{\mathcal{Q}}
\DeclareMathOperator{\gr}{gr}
\DeclareMathOperator{\id}{id}

\newcommand{\lrangle}[1]{\langle #1 \rangle}
\DeclareMathOperator{\Span}{Span}
\DeclareMathOperator{\colim}{colim}
\DeclareMathOperator{\hocolim}{hocolim}
\DeclareMathOperator{\sSet}{sSet}
\DeclareMathOperator{\coeq}{coeq}
\DeclareMathOperator{\eq}{eq}
\DeclareMathOperator{\Filt}{Filt}
\DeclareMathOperator{\Top}{Top}
\DeclareMathOperator{\Link}{Link}
\DeclareMathOperator{\GenAlg}{GenAlg}
\DeclareMathOperator{\GenAlgKnot}{GenAlgKnot}

\newcommand{\Ii}{J}
\newcommand{\Jj}{J_{v_0}}

\newcommand{\nat}{\natural}
\newcommand{\shrpt}[1]{\left( #1 \right)_{\sharp}}

\newcommand{\Orb}[1]{\{ [ #1 ] + iv \}_{i \in \Z}}
\newcommand{\Orbe}[1]{\{ [ #1 ] + ie \}_{i \in \Z}}
\newcommand{\OrbI}[1]{\{ [ #1 ] + \sum_{v\in I}i_vv\}_{i_v\in\Z}}
\newcommand{\OSS}{Osv\'ath, Stipsicz, and Szab\'o}

\newtheorem{theorem}{Theorem}[section]
\newtheorem{lemma}[theorem]{Lemma}

\newtheorem{proposition}[theorem]{Proposition}

\theoremstyle{definition}
\newtheorem{definition}[theorem]{Definition}

\theoremstyle{remark}
\newtheorem{remark}[theorem]{Remark}
\newtheorem{remarkn}[theorem]{Remark}
\newtheorem{example}[theorem]{Example}
\title{Invariance and Naturality of Knot Lattice Homology and Homotopy}
\author{Seppo Niemi-Colvin}
\date{April 2022}

\begin{document}

\begin{abstract}
Links of singularity and generalized algebraic links are ways of constructing three-manifolds and smooth links inside them from potentially singular complex algebraic surfaces and complex curves inside them.
We prove that knot lattice homology is an invariant of the smooth knot type of a generalized algebraic knot in a rational homology sphere.
In that case, knot lattice homology can be realized as the cellular homology of a doubly-filtered homotopy type, which is itself invariant.
Along the way, we show that the topological link type of a generalized algebraic link determines the topology of the minimal plumbing resolution for the nested singularity type used to create it.
Knot lattice homotopy is a natural invariant in that diffeomorphisms of the knot that play suitably well with the minimal good resolution will provide a contractible space of morphisms between the doubly-filtered knot lattice spaces associated to any presentation.
\end{abstract}
\maketitle
\tableofcontents

\section{Introduction}
One of the first questions a topologist usually asks of a calculation is if it is invariant, especially if it relies on a presentation of the underlying topological structure, such as a three-manifold or knot. 
While there are ways to create invariants out of non-invariant calculations, such as minimizing over all presentations \added{or using a minimal presentation}, knowing that a calculation is invariant provides hope that the calculations are getting at genuine topological information about the structure, instead of simply being artifacts of the presentation.
\added{Additionally, a minimal presentation may not express all relationships between the invariants, as for example a more complicated presentation may be needed for cobordism morphisms.}

We focus on the class of three-manifolds that are links of normal complex surface singularities and a class of knots that generalizes algebraic knots to this setting.
The calculations of concern for this article are knot lattice homology, which is built off of the foundations of lattice homology, and knot lattice homotopy which realizes knot lattice homology as the cellular homology of a doubly-filtered homotopy type.
We also include some maps as a part of the package we would like preserved, based on the context of Heegaard Floer homology and knot Floer homology.
The desired invariance is summarized in the following theorems:

\begin{theorem}\label{invarianceResult}
Let \((Y,K)\) be a generalized algebraic knot in a rational homology sphere with a \Spinc structure \(\tf\). The chain homotopy equivalence type of
\[(\CFKb(G_{v_0},\tf),\Jj,\Ii)\]
depends only \((Y,K)\) and \(\tf\) and not the particular graph with unweighted vertex \(G_{v_0}\) used to represent it.
\end{theorem}

\begin{theorem}\label{topInvarianceResult}
Let \((Y,K)\) be a generalized algebraic knot in a rational homology sphere with a \Spinc structure \(\tf\). The doubly filtered homotopy type of
\[(\CFKb^{\nat}(G_{v_0},\tf),\Jj^{\nat},\Ii^{\nat})\]
depends only on \((Y,K)\) and \(\tf\) and not the particular graph with unweighted vertex \(G_{v_0}\) used to represent it.
\end{theorem}

A paper claiming the first result without the maps \(\Jj\) and \(\Ii\) has appeared independently on the arXiv in a paper by Matthew Jackson \cite{otherInvariance}\added{, though there was an error in their proof} (see Remark \ref{knotNeededRem}).

\added{We go further by proving a naturality statement that not only are the invariants well defined, but the morphisms realizing those invariants are well defined up to a contractible choice.
Additionally, one would suspect that diffeomorphisms of \((Y,K)\) sufficiently compatible with the structure used to define knot lattice would provide weak equivalences of the knot lattice spaces, and indeed we get such maps, well defined up to a contractible choice.
This is described below, where \(C_{\GenAlgKnot}^{\Spin^c}\) is a category defined in Section \ref{subsubsecgenalglinks} which captures which diffeomorphisms are compatible with resolutions of the relevant singularities.
Note that while this paper does not get into the details on \(\infty\)-categorical interpretations, in that setting the functor from \(\tilde{C}_{\GenAlgKnot}^{\Spin^c}\) to \(C_{\GenAlgKnot}^{\Spin^c}\) is an equivalence, ensuring an \(\infty\)-categorical functor from \(C_{\GenAlgKnot}^{\Spin^c}\) itself.
}

\added{
\begin{theorem}\label{thm:natResult}
There exists a topologically enriched category \(\tilde{C}_{\GenAlgKnot}^{\Spin^c}\) which 
\begin{enumerate}
\item has objects \((\pi,\tf)\) where \(\pi\colon \tilde{S}\to S\) is a resolution of a weak nested singularities \((S,C,\{p_1,\ldots, p_n\})\) with link a generalized algebraic knot in a rational homology three-sphere \((Y,K)\), and \(\tf \in \SpincX{Y}\).
\item has morphism spaces \(\tilde{C}_{\GenAlgKnot}^{\Spin^c}((\pi_1,\tf_1),(\pi_2,\tf_2))\) have a contractible component for each morphism in \(C_{\GenAlgKnot}^{\Spin^c}((Y_1,K_1,\tf_1),(Y_2,K_2,\tf_2))\) between the generalized algebraic knots.
\end{enumerate}
and there is a functor
\[\CFKb^{\nat}\colon \tilde{C}_{\GenAlgKnot}^{\Spin^c}\to \Filt_{[\Q\times\Q\colon 2\Z\times2\Z]}\]
which takes a resolution \((\pi,\tf)\) to \(\CFKb^{\nat}(G_{v_0},\tf)\) where \(G_{v_0}\) is the dual graph for \(\pi\) and on minimal resolutions agrees with the functor \(F_{min}\).
\end{theorem}
}

\added{
This is both stronger and weaker than the current naturality statements for Heegaard Floer and knot Floer \cite{naturalityHF, naturalityHFProj}.
It is weaker in that there are extra requirements on the three-manifolds, knots, diffeomorphisms, and the isotopies between them, but given such requirements we produce higher order homotopies between filtered spaces.
In addition to the addition of a space level analysis, this also implies results when the chain complex is considered with \(\Z[U]\) coefficients in not just a projective manner. We have not yet verifie contractible choices of maps between the filtered spaces associated with analytic lattice homology \cite{anLattice}.ed that both naturality statements agree where defined.
}
\added{
Similar theorems should also hold for all of N\'emethi's suite of lattice style invariants, as the proofs of invariance for the invariants in the suite follow the same general strategy.
For example replacing \(C_{\GenAlgKnot}^{\Spin^c}\) with a similarly defined \(C_{\Link}^{\Spin^c}\) that no longer keeps track of the knot, one gets a similar functor for the filtered lattice space of the underlying three-manifold.
Isomorphisms of normal complex surface singularities should also induc
}

Despite being able to recover the algebraic invariant from the homotopy theoretic one, this paper tackles both the algebraic invariants and the homotopy invariants for several reasons.
First, while geometric intuition aids in the discussion, dealing with the algebraic invariants first without worry for continuity then using these maps as a guide for constructing the continuous homotopies is easier.
Second, this assuages potential worries that the homotopy theoretic argument only provides a quasi-isomorphism of chain complexes rather than a chain homotopy equivalence due to the knot lattice complex being associated to cellular rather than singular homology.
We will use the \(\nat\) symbol to signify that we are dealing with the topological realization of an algebraic invariant.

Andr\'as N\'emethi defined lattice homology for links of normal complex surface singularities \cite{latticeCohomNormSurf} based on attempts by N\'emethi, Ozsv\'ath, and Szab\'o to compute Heegaard Floer homology when the three-manifold bounds a negative definite tree-like plumbing of spheres \cite{NemAR,PlumbingsPape}.
Lattice homology was conjectured to be the same as Heegaard Floer homology for rational homology spheres \cite{latticeCohomNormSurf}, and after it was shown to coincide in some particular cases \cite{NemAR,knotsAndLattice}, Ian Zemke proved their equivalence for rational homology spheres \cite{ZemkeLattice}.
The lattice homology of rational homology spheres also shares some formal properties with Heegaard Floer homology such as: an exact triangle (\cite{holoDisksTopological} for Heegaard Floer homology, \cite{surgeryTriangle,twoExactSequences} for lattice homology) and the existence of an involutive structure (\cite{InvolSetup} for Heegaard Floer homology, \cite{involCheck} for lattice homology).
Lattice homology was particularly useful for the creation of an algorithm for computing the Heegaard Floer homology of a Brieskorn sphere \(\Sigma(p,q,r)\) purely in terms of $p,q,$ and $r$ \cite{brieskornAlgorithm}.

Ozsv\'{a}th, Stipsicz, and Szab\'{o} developed knot lattice homology as a way to make use of the surgery formula from knot Floer homology in the lattice homology setting \cite{knotsAndLattice}.
This construction parallels a lot of the structures in knot Floer homology, particularly in that it can be viewed as the addition of a filtration to the lattice complex for the ambient three-manifold.
Ozsv\'{a}th, Stipsicz, and Szab\'{o} showed that when the ambient three-manifold is a lattice homology L-space, knot lattice homology and knot Floer homolgoy agree \cite{knotLatticeLspace}.
Alfieri used the Alexander filtration from knot lattice homology to deform lattice homology, and showed that this deformation corresponded to the analagous deformation of Heegaard Floer homology for graphs with at most one bad vertex (i.e. a vertex whose degree is greater than the negative of its weight) \cite{upsilonGenAlg}.
Alfieri used this to provide a method for computing the upsilon invariant \cite{upsilonGenAlg}.
\added{Borodzik, Liu, and Zemke have then extended Zemke's argument for lattice homology to show that a link lattice complex generalizing the construction of the knot lattice complex is equivalent with link Floer homology \cite{linkLattice}, though they did not work on the level of filtered spaces.}

Furthermore, finding homotopy-theoretic versions of Floer-theoretic invariants is an active area of research, including work by Kronheimer, Manolescu, and Sarkar for Seiberg-Witten Floer theory \cite{SWFhomotopy, PeriodicFloerSpectra} and knot Floer homology \cite{gridHomotopy}. 
Theorem \ref{topInvarianceResult} fits in with this wider research program.
The work of Manolescu and Sarkar \cite{gridHomotopy} in on creating a homotopy theoretic version of knot Floer homology expresses the action of \(UV\) in knot Floer homology differently than we do for knot lattice homology, so the two theories are not at this point immediately comparable \added{though one would hope that after application of a functor on the knot lattice space they become equivalent.
Invariance is not yet known for the homotopy theoretic version of knot Floer homology.}

Section \ref{SingBackSec} covers the relevant singularity theory with an eye towards low-dimensional topology.
In particular Subsection \ref{linksSubsec} goes over links of singularity in various contexts, culminating in introducing weak generalized algebraic links, which combine aspects of links of normal complex surface singularities and algebraic links in \(S^3\).
In each of these contexts, we define what the link of a singularity means, as well as how to produce a plumbing diagram by resolving the singularity in question.
We introduce the moves of blow-ups and blow-downs needed to get between plumbing resolutions of the same singularity, and \added{in each setting we discuss the category of resolutions, which gives rise to a minimal resolution as its terminal object.}
This builds up to Subsection \ref{BlowUpsSufficeSec}, where we show that we only need to check invariance under blow-ups and blow-downs of the graph representing the weak generalized algebraic knot.
This leads to the following theorem of potentially independent interest, a variation of which was independently proved by Jackson in \cite{otherInvariance}.

\begin{theorem}\label{BlowUpsSuffice}
The link type of a \added{weak} generalized algebraic link \((Y,L)\) determines the topology of the minimal good resolution for the singularity \((S,C,\{p_1,\ldots, p_n\})\) used to create it.
\added{In particular, if there exists homeomorphism \(f\colon (Y_1,L_2)\to (Y_2,L_2)\) between weak generalized algebraic links, then there exists a diffeomorphism which extends over the minimal good resolutions of the nested singularities and respects the plumbing structures with fibers.} 
\end{theorem}

Before introducing either lattice homology or knot lattice homology, in Section \ref{FloerSec} \added{we discuss the target categories we are aiming for and our conditions for equivalence, including the relevant perspective on doubly-filtered spaces along with the conditions on the maps \(\Jj\) and \(\Ii\).}
The most important pieces of information from this section are the definitions of a complete knot package for a knot \((Y,K)\) and a complete homotopy package for a knot \((Y,K)\), which respectively define the equivalence types for knot lattice homology and knot lattice homotopy in this setting.
\added{We  also establish several functors on doubly-filtered spaces.
These functors and the conditions on \(\Jj\) and \(\Ii\) mirror constructions and conditions in involutive knot Floer homology.}

We then cover the constructions for lattice homology (in Section \ref{introToLatticeHom}) and knot lattice homology  (in Section \ref{knotLatticeBackSec}).
Lattice homology provides a cube decomposition of Euclidean space with a height function on it, and we use this cube decomposition to construct our chain complex over \(\Z[U]\) as well as a filtered topological space.
Knot lattice homology adds a second height function to the chain complex and a second filtration to the topological space.
Both lattice homology and knot lattice homology come equipped with involutive maps, respectively denoted as \(\Ii\) and \(\Jj\), which can be realized as continuous maps.
\added{Note that this presentation of the knot lattice chain complex differs from that of \OSS's packaging in that they track the Alexander grading and our second height function is most analagous to the second maslov grading of knot Floer homology, and this formulation does not require passing to a larger graph in order to define.
We provide a proof that these two presentations agree, and we also extend formulas of \OSS{} for chain complexes to filtred spaces.
The height functions here also play a similar role to N\'emethi's weight functions, but in a formulation more natural to the Heegaard Floer context.}

In Section \ref{ContractionSec} we develop the tools to construct the desired spaces of morphisms.
First, in Section \ref{subsec:hocolim} we review the construction of homotopy colimits and how they relate to homotopy coherent natural transformations, including a condition on a diagram where one knows the space of homotopy coherent natural transformations from any other diagram is either empty or contractible.
Then, in Section \ref{subsec:reducVertex} we provide a description of the knot lattice space as a homotopy colimit based on a subset of vertices of the describing graph, and in the case where the subset in question is a singleton, show that the resulting diagram satisfies the condition of the previous section.
Finally, in Section \ref{subsec:Compare}, we compare this approach to the approaches of Jackson, N\'emethi, and \OSS.
In particular, we show that this approach is compatible with the quasi-fibrations and contractions used by N\'emethi and \OSS, while allowing for stronger control over the resulting space of morphisms.
We also provide some tools for understanding what the resulting morphisms can look like and the role the knot itself plays in that process.


In section \ref{sec:InvarNat} we finally give the proofs for invaraince and naturality.
Section \ref{subsec:BlowUps} focuses on invariance , due to Theorem \ref{BlowUpsSuffice} specifically looks at the case where one resolution is a blow up of other resolution, allowing for us to conclude Theorem \ref{topInvarianceResult}.
Section \ref{subsec:Naturality} then shows how the morphisms from Section \ref{subsec:BlowUps} can be arranged into the broader functor of Theorem \ref{thm:natResult}.

\subsection{Avenues for Future Work}

While a lack of an invariance proof has not prevented researchers such as Ozsv\'ath, Szab\'o, Stipsicz, and Alfieri from working with knot lattice homology \cite{knotLatticeLspace,upsilonGenAlg}, and Borodzik Liu, and Zemke have proved invariance for link lattice by showing it equivalent to link Floer, the proof of invariance provided here provides justification for further work on the knot lattice homotopy type.

Since lattice homology provided a useful tool in finding algorithms for computing the Heegaard Floer homology of Brieskorn spheres \cite{brieskornAlgorithm}, one may suspect that knot lattice homology could provide a useful tool for the computation of knot Floer homology and related invariants for the regular and singular fibers of Brieskorn spheres.
There may be benefits to performing these calculations in homotopy theoretic terms.
More generally, computation of the underlying homotopy type is needed in more cases to see if it contains more information than the chain complex.

Finally, N\'emethi defined lattice cohomology in the context of general links of normal complex surface singularities not just rational homology spheres and framed the behavior of lattice homology in the terms of singularity theory \cite{latticeCohomNormSurf}.
Extending the proof of knot lattice homology's invariance to this larger setting and connecting it with the properties of nested singularities would be interesting.

\section*{Acknowledgements}
I would like to acknowledge Andr\'{a}s Stipsicz for his correspondence including verifying that knot lattice homology was not yet known to be invariant when I started.
Furthermore, I would like to thank Barbara Fantechi and Twitter user \verb|@d_m_d_m_d_d| for helping me understand what it means for a complex variety to be normal.
I would like to thank Irving Dai and Antonio Alfieri for useful conversations at the Nearly Carbon Neutral Geometric Topology Conference, and I thank Sean Fiscus for useful conversations while in Budapest.
Andr\'{a}s N\'emethi deserves thanks for pointing me towards his updated proof of invariance for the latice spaces. 
Finally, I would like to thank my advisor Adam Levine for lots of advice during the editing process.
The author was partially supported by NSF grant DMS-1806437.

\section{Links of Singularity and Generalized Algebraic Links}\label{SingBackSec}

Singularity theory provides important context for the construction of knot lattice homology, in particular, by providing the necessary scope where it is defined, i.e. generalized algebraic knots in rational homology spheres, and the needed presentation of said knots in order to calculate knot lattice homology, i.e. resolutions of nested singularities.
Furthermore Subsection \ref{BlowUpsSufficeSec} provides a proof of Theorem \ref{BlowUpsSuffice}, which is needed to show invariance.
Since a number of people approaching knot lattice homology may be doing so from a low-dimensional topology rather than singularity theory background, we will not be assuming much familiarity with algebraic geometry.

\subsection{Types of Links of Singularity and Their Resolutions}\label{linksSubsec}
Before defining a link of a singularity in any of the relevant contexts, we will need the following definition.

\begin{definition}
Given a topological space \(X\) and a subspace \(Y\), we have that a \emph{regular neighborhood \(N\) of \(Y\)} is a closed neighborhood of \(Y\) in \(X\) so that there exists a map \(\pi: \bou N \to X\) where \(N\) is homeomorphic to the mapping cylinder of \(\pi\).
\end{definition}
Regular neighborhoods generalize the idea of tubular neighborhoods to the singular setting, and in the case where \(X\) is a point, we have that regular neighborhoods are cone-like.

Generalized algebraic knots combine aspects of both links of normal complex surface singularities and algebraic links, and \added{we will first cover the relevant material in the context of links of normal complex surface singularities before adding the additional structure of a complex curve to give a knot.}
In each context, we will cover what the relevant link of singularity is and what is required of a resolution of a singularity.
Relevant to the context of how a generalized algebraic knot is presented for knot lattice homology, we cover how to represent said resolutions using plumbing graphs.
In preparation for the proof of invariance, we discuss what moves are allowed on said plumbing graphs and \added{the role minimal resolutions play in each context.}

\subsubsection{Links of Normal Complex Surface Singularities}\label{subsubseclinkssurf}


\begin{definition}
Given a complex variety \(V\), viewed as an algebraic subset of \(\C^n\) equipped with the euclidean topology, which has an isolated singularity at \(p\), we define its \emph{link} as the boundary of a regular neighborhood around \(p\) in \(V\).
A \emph{link of multiple singularities} \(p_1,\ldots,p_n\) is the boundary of a regular neighborhood of a path connecting \(p_1,\ldots, p_n\) and passing through no other singular points.
\end{definition}

\added{Note that the term link here refers to the three-manifold (as a link of the singularity) and not to an embedded disjoint union of circles, and as such there are concerns when taking the connect sum of those links taht do not appear here.}
A common way to form the link is using a particular embedding of \(V\) into affine space and taking the intersection of \(V\) with a small enough sphere \(S_{\epsilon}(p)\) around \(p\).
A complex affine space is \emph{normal} at its singular point \(p\) if the ring of functions localized at \(p\) is integrally clsoed in its ring of fractions.
Here we care about links of normal complex surface singularities, and the links of these singularities are oriented 3-manifolds.
Figure \ref{subfig:linkofsing} shows a dimensionally reduced picture of a link of a normal complex surface singularity.
If \(V\) is normal, then the link around any of its points must be connected.

The motivation for the definition of a link of multiple singularities is that if a sphere instead contains multiple isolated singularities \(p_1,\ldots, p_n\), and no other topology, then the result is the connected sum of the links of individual \(p_i\).
Viewing the sphere as the boundary of a cone-like neighborhood, expanding those neighborhoods until they merge gives the cobordism connecting the disjoint union and the connected sum.
Extending the cores of the 1-handles connecting the individual neighborhoods to the singularities provides the path used in constructing the regular neighborhood.
This argument also gives why the precise path between the points does not matter; a different choice of path would result in handle slides for the 1-handles and connected sum being applied in a potentially different order.
If one would prefer, one can have the link of multiple singularities be the disjoint union of the links of the individual singularities, but it is traditional in low-dimensional topology to restrict to connected manifolds with the operation of choice being connect sum.
Furthermore, lattice homology is defined in the context of connect sums of links of singularity, so we chose a definition that includes those in its scope.

\begin{figure}
    \centering
    \begin{subfigure}[t]{.3\textwidth}
        \includegraphics[width=\textwidth]{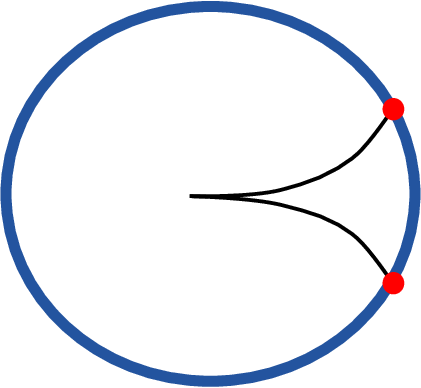}
        \caption{An algebraic knot (assumed in \(S^3\).}
				\label{subfig:algknot}
    \end{subfigure}
    \hfill
    \begin{subfigure}[t]{.3\textwidth}
        \includegraphics[width=\textwidth]{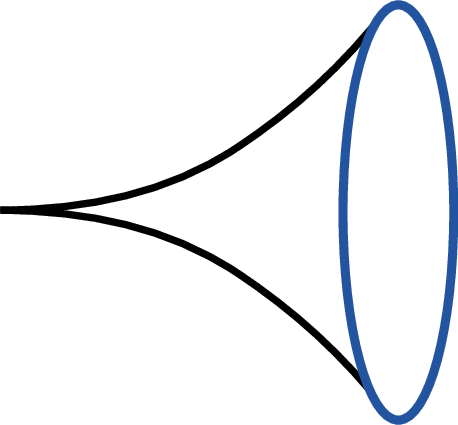}
        \caption{A link of a normal complex surface singularity}
        \label{subfig:linkofsing}
    \end{subfigure}
    \hfill
    \begin{subfigure}[t]{.3\textwidth}
    \includegraphics[width=\textwidth]{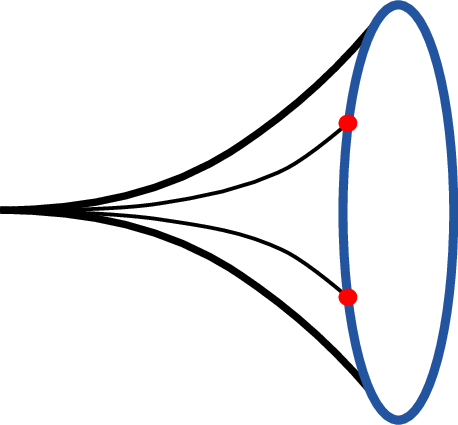}
    \caption{A generalized algebraic knot}
    \label{subfig:genalgknot}
    \end{subfigure}
    \caption{Dimensionally reduced figures for an algebraic knot, a link of a normal complex surface singularity, and a generalized algebraic knot.}
    \label{fig:algDefOfLinks}
\end{figure}

\begin{figure}
    \centering
    \begin{subfigure}[t]{.4\textwidth}
        \includegraphics[width=\textwidth]{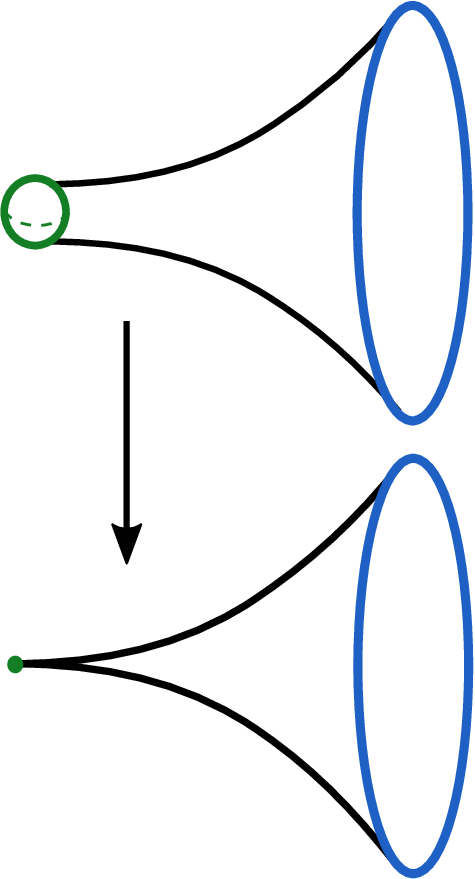}
        \caption{Resolution of a normal complex surface singularity}
        \label{subfig:surfacesingres}
    \end{subfigure}
    \hfill
    \begin{subfigure}[t]{.4\textwidth}
        \includegraphics[width=\textwidth]{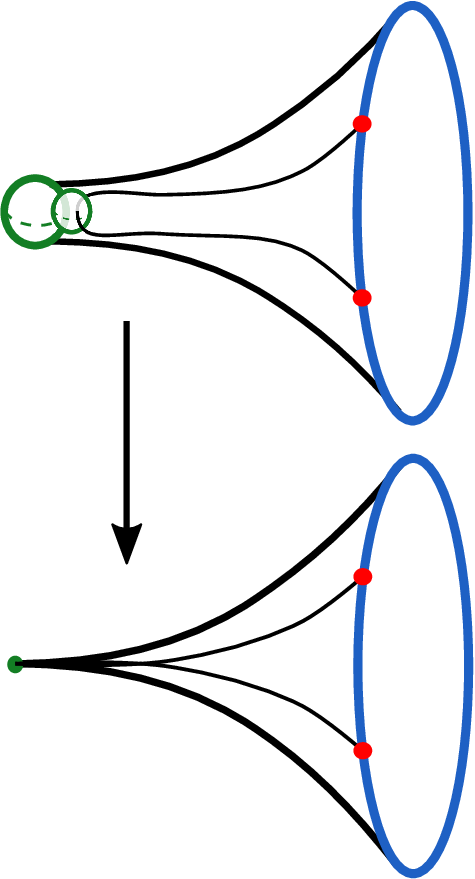}
        \caption{Resolution of a normal complex surface singularity plus a curve}
        \label{subfig:knotsingres}
    \end{subfigure}
    \caption{A resolution in the three manifold context and for a curve inside it. The extra sphere in Subfigure \ref{subfig:knotsingres} signifies that further blow-ups may be needed to resolve the curve. Furthermore, the curve pierces the sphere to emphasize that it should be transverse to the sphere, and in fact locally represented as a fiber of the plumbing.}
    \label{fig:res}
\end{figure}

\begin{definition}
A \emph{resolution of a singularity \((V,p)\)} is a map \(\pi:\tilde{V}\to V\) which is an isomorphism from \(\tilde{V}- \pi^{-1}\{p\}\) to \(V-\{p\}\) and where \(\tilde{V}\) is a smooth complex affine variety.
A \emph{resolution of multiple singularities \((V,\{p_1,\ldots, p_n\})\)} is defined similarly except the map \(\pi\) must only be an isomorphism from \(\tilde{V}- \pi^{-1}(\{p_1,\ldots p_n\})\)  to \(V- \{p_1,\ldots, p_n\}\).
For each \(p_i\) we have that \emph{\(\pi^{-1}(\{p_i\})\)} is an \emph{exceptional fiber} of the resolution.
\added{A resolution is \emph{good} if all the embedded singularities in the exceptional fibers are \emph{normal crossings}, i.e. can be analytically modeled on \(\{xy=0\}\) in \(\C^2\).
}
\end{definition}

Figure \ref{subfig:surfacesingres} depicts the resolution of a singularity.
Any normal surface singularity can be resolved \added{to a good resolution}, and in particular an algorithm is given in \cite{normalsurfacesing}.
Note that the process of resolving a singularity does not change its link.

\added{Associated to a good resolution is a \emph{dual graph \(\Delta\)}, which has vertices corresponding to the irreducible components of the exceptional fibers and edges corresponding to intersections between them (counting self intersections as loops). Vertices have integer weights encoding the euler number of their normal bundle.
One may also choose to keep track of the genus of the irreducible components using another number, often denote in brackets to distinguish it from the euler number or dropped entirely if it is zero.
This encodes the topological information contained in a neighborhood of the exceptional fibers, specifically the pullback of a regular neighborhood of the singularity to the resolution.
This neighborhood will be a smooth four-manifold and is an example of a more general construction called a plumbing and realizes the link of singularity as a graph three-manifold.
}
%

\begin{definition}
A \emph{plumbing of two disk bundles over surfaces} is the manifold achieved by gluing trivialized neighborhoods of points in such a way so that the base of one becomes identified with the fiber of the other.
In particular, the neighborhoods over which one is trivializing the disk bundles must be homeomorphic to disks themselves.
A general \emph{plumbing of (disk bundles over) surfaces} is a four-manifold constructed from set of disk bundles over surfaces with the plumbing operation done some number of times between them.
The \emph{base surfaces of the plumbing} are the images of the zero sections of the disk bundles used to make \added{the plumbing. These are included from the individual disk bundles into the plumbings.}
\added{A \emph{graph three-manifold} is a three manifold that is the fiberwise boundary of a plumbing, i.e. only the circle boundaries of the disk bundles are counted.}
\end{definition}

Plumbings can also be represented by undirected graphs called \emph{plumbing graphs}, which generalizes the notion of the dual graph to this setting.
It is possible for general plumbings to allow nonorientable bases (so long as the total space of the disk bundle is oriented) \added{denoted inside the brackets with negative the crosscap number, or even base surfaces with boundary components, written as a second number inside the brackets and no euler number.
In general, plumbing operations can be done two ways: preserving the orientations on fiber and base or reversing both.
This can be recorded individually as a \(+\) or \(-\) on the edges, according to the sign of the induced intersection between the base surfaces plumbed together; however, up to change of orientation on the base surfaces, the signs depend only on a choice of which cycles in the graph will have an odd number of \(-\) edges.}

The intersection form for the plumbing can be read directly off of the graph.
 \added{The orientable closed} base surfaces for the plumbing provide a basis for the second homology of the plumbing.
A base surface \(S_i\) has self intersection equal to the Euler number \(e_i\) of the corresponding disk bundle, and two different base surfaces \(S_i\) and \(S_j\) have intersection equal to the number of edges between them \added{counted with sign}. 

A plumbing graph that comes from a dual graph will have all of the base surfaces be closed and orientable and all edges will have a positive sign.
Furthermore, such a plumbing graph will represent a dual graph for the resolution of some singularity if and only if the intersection form is negative definite \cite{NegiDefi}.
The graphs representing resolutions of links of normal complex surface singularities can potentially have cycles or higher genus bases; however, the link is a rational homology sphere if and only if the plumbing graph is a forest of spheres.
We will be primarily dealing with dual graphs, using them to describe the link of singularity, though more general plumbings will come up again in Subsection \ref{BlowUpsSufficeSec}, where we prove Theorem \ref{BlowUpsSuffice}.


%

\begin{definition}
Given a point \(q\) of a smooth complex surface \(V\) so that in local coordinates \((x,y)\) we have that \(q\) is the origin, then \emph{the blow-up of \(V\) at \(q\)} is the complex surface \(V_q\) where the chart around \(q\) is replaced with
\[\{((x,y),[s:t]) \in \C^2\times \mathbb{CP}^1\, :\, xt=ys\}.\]
The projection \(\pi:V_q\to V\). induced by locally projecting onto \(\C^2\) is called the \emph{canonical projection}.
The curve \(\pi^{-1}(q)\) is called \emph{the exceptional sphere}.
The complex surface \(V\) is \emph{the blow-down} of \(V_q\) at the exceptional sphere.
\end{definition}
One can define a blow-up when \(V\) is singular at \(q\) as well, and this is a part of the process of constructing resolutions, though we will not be needing it in this paper.
Algebraically, blowing up a point \(q\) replaces \(q\) with the set of lines through that point, and assuming that we were already at a resolution, the space of lines is represented by a copy of \(\mathbb{CP}^1\), which will have a normal bundle of Euler number \(-1\).
Smoothly this operation produces a connect sum with \( \overline{\mathbb{CP}^2}\).
Given the presence of an embedded sphere with normal bundle of Euler number \(-1\) in \(V\), we can blow down that sphere, i.e. there exists a smooth complex surface \(\tilde{V}\) which is the blow-down of \(V\), a process which can be achieved holomorphically if the original sphere was holomorphic \cite{blowDown}.

Blowing up a plumbing resolution will give another plumbing resolution with a new vertex \(e\) that has weight \(-1\), representing the exceptional sphere.
To understand how blow-ups affect the graph, we will need the idea of the proper and total transform of a curve under a resolution.
\begin{definition}
Given a resolution \(\pi:\tilde{V}\to V\) of a complex surface \(V\) at a point \(q\) and a complex curve \(C\) inside \(V\), the \emph{total transform of \(C\)} inside the resolution is \(\pi^{-1}(C)\), while the \emph{proper transform of \(C\)} inside the resolution is the closure of \(\pi^{-1}(C-\{q\})\).
\end{definition}
In particular, blowing up provides a resolution of our plumbing, and the base surfaces for this new plumbing resolution will be precisely the proper transforms of the base surfaces of our original plumbing.

The effect of a blow-up on a plumbing graph depends on which point \(q\) was blown up.
If \(q\notin \pi^{-1}\left(\{p_1,p_2,\ldots,p_n\}\right)\), then the blow-up would be represented by the disjoint union of the previous graph with the vertex \(e\).
This is called a \emph{generic blow-up}.
If \(q\) is on a single one of the base surfaces \(v\) in \(\pi^{-1}\left(\{p_1,p_2,\ldots,p_n\}\right)\), we say that we are \emph{blowing up a vertex}.
In the resulting graph, \(e\) has an edge with \(v\) and the weight of \(v\) is lowered by 1.
Finally \(q\) could potentially be at the intersection of two base surfaces \(v_1\) and \(v_2\), in which case we say that we are \emph{blowing up an edge}.
In that case, we connect \(e\) to both \(v_1\) and \(v_2\) then delete the old edge connecting them.
The weights on both \(v_1\) and \(v_2\) lower by 1.
Figure \ref{fig:trefblowup} covers the case of blowing up a resolution of a nested singularity, but Figures \ref{subfig:genblowup}, \ref{subfig:vertexblowup}, \ref{subfig:edgeblowup} provide examples for the blow-ups discussed here.


\begin{figure}
    \centering
    \begin{subfigure}[t]{.4\textwidth}
        \includegraphics[width=\textwidth]{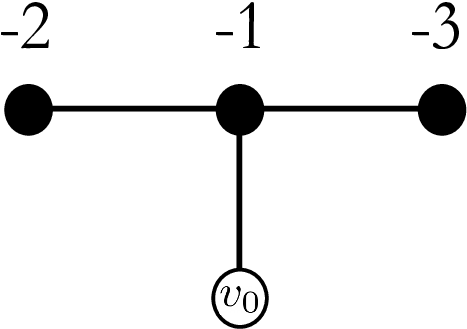}
        \caption{Original graph}
        \label{subfig:tref}
    \end{subfigure}
    \hfill
    \begin{subfigure}[t]{.4\textwidth}
        \includegraphics[width=\textwidth]{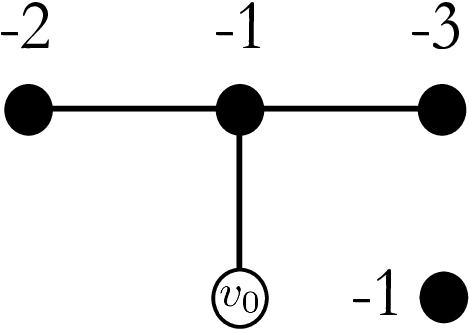}
        \caption{Generic blow-up}
        \label{subfig:genblowup}
    \end{subfigure}
    
    \begin{subfigure}[t]{.4\textwidth}
        \includegraphics[width=\textwidth]{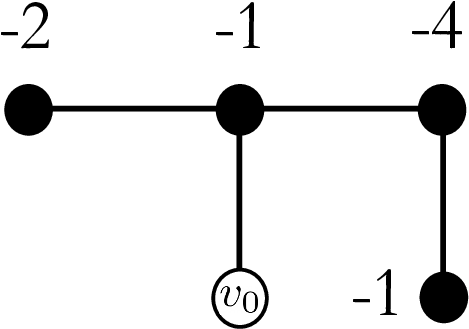}
        \caption{Blow-up of a (weighted) vertex}
        \label{subfig:vertexblowup}
    \end{subfigure}
		\hfill
		    \begin{subfigure}[t]{.4\textwidth}
        \includegraphics[width=\textwidth]{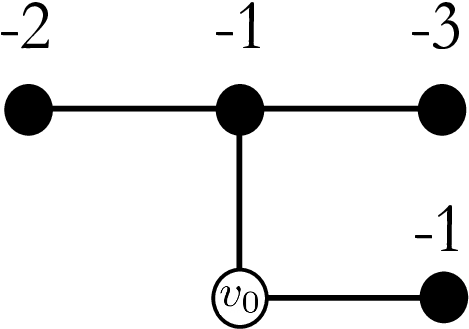}
        \caption{Blow-up of the unweighted vertex}
        \label{subfig:vertexblowupspec}
    \end{subfigure}
    
    \begin{subfigure}[t]{.4\textwidth}
        \includegraphics[width=\textwidth]{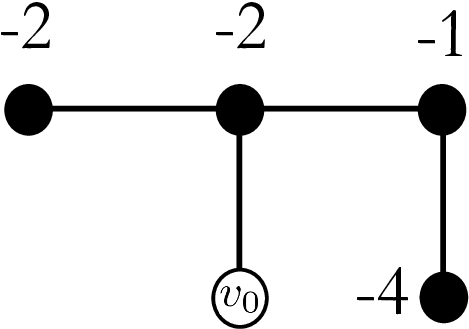}
        \caption{Blow-up of an edge (connecting two weighted vertices)}
        \label{subfig:edgeblowup}
    \end{subfigure}
    \hfill
    \begin{subfigure}[t]{.4\textwidth}
        \includegraphics[width=\textwidth]{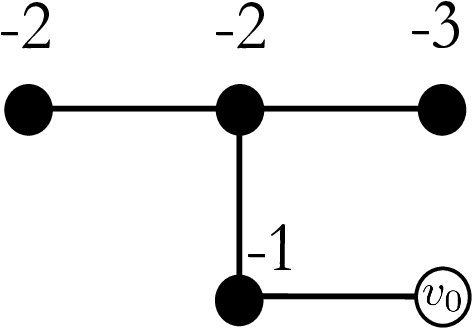}
        \caption{Blow-up of the edge connected to \(v_0\)}
        \label{subfig:edgeblowupspec}
    \end{subfigure}
    
    \caption{An example of a graph depicting a generalized algebraic knot \ref{subfig:tref}, along with the various blow-up moves that can be done to said graph that occur away from the curve defining the knot (\ref{subfig:genblowup},\ref{subfig:vertexblowup},\ref{subfig:edgeblowup}) and those that affect the curve defining the knot (\ref{subfig:vertexblowupspec} and \ref{subfig:edgeblowupspec}).}
    \label{fig:trefblowup}
\end{figure}

Not only will there exist a resolution of a given singularity, there will be an unique minimal good resolution.
\added{Before moving on, we will clarify what we mean by minimal.}
\begin{definition}
\added{
Let \((V,\{p_1,\ldots, p_n\})\) be a collection of singularities. We will define \emph{the category of good resolutions} for \((V,\{p_1,\ldots,p_n\})\) has objects good resolutions \(\pi\colon \tilde{V}\to V\) over the points \(\{p_1,\ldots, p_n\}\) and morphisms analytic maps \(\phi\colon \tilde{V_1}\to \tilde{V_2}\) that commute with the resolution maps to \(V\).
We will say a resolution is \emph{minimal} if any morphism from it is an isomorphism.
}
\end{definition}
\added{Note the category above is a poset.
In particular, away from \(\{p_1,\ldots, p_n\}\), \(\pi_1\) and \(\pi_2\) must be isomorphisms and in particular \(\phi= \pi_2^{-1}\circ \pi_1\) and \(\pi_i^{-1}(\{p_1,\ldots, p_n\})\) will be dense in \(\tilde{V_i}\).
When this map does exist it will present \(\tilde{V_1}\) as the resolution of \(\tilde{V_2}\) over some collection of points.
However note that, minimal resolutions are minimal poset elements with respect to the opposite of this poset, i.e. a map from \(\tilde{V_1}\) to \(\tilde{V_2}\) implies \(\tilde{V_1}\geq \tilde{V_2}\).
Furthermore, the proof of existence and uniqueness of a minimal resolution further states that being minimal is equivalent to having no sphere that can be blown down, and for any resolution the proof provides a minimal resolution that the original resolution maps to through successive blow downs.
As such these minimal resolutions are in fact terminal in the category of good resolutions and thus unique up to unique isomorphism.
See \cite{normalsurfacesing} for discussion of why minimal good resolutions exist.}


This argument highlights that is possible to get between any two resolutions of the same singularity through blow-ups and blow-downs.
Furthermore, \cite{calcForPlumbing} provides that the homeomorphism type of the link of singularity determines the topology of the minimal good resolution.
\added{The following definition will allow us to discuss what we mean by determines in a categorical form. Note that we discuss \Spinc structures further in Section \ref{FloerSec}.}
\begin{definition}
Let \(X_{G_1}\) and \(X_{G_2}\) be plumbings and \(f\colon X_{G_1}\to X_{G_2}\) be an orientation preserving diffeomorphism.
We say that \(f\) \emph{respects the plumbings} if when restricted to each disk bundle inside \(X_{G_1}\), \(f\) becomes an isomorphism with some disk bundle in \(X_{G_2}\).
If \(Y_{1}\) and \(Y_{2}\) are links of the normal complex surface singularities \(V_1\) and \(V_2\), we say a diffeomorphism \(f\colon Y_1\to Y_2\) is \emph{graph-induced} if it extends over the minimal resolutions to a diffeomorphism \(\tilde{f}\colon\tilde{V}_1\to\tilde{V}_2 \) that respects plumbings, and we define a graph induced isotopy of diffeomorphisms similarly.
The \emph{groupoid of links of normal complex surface singularities}  \(C_{\Link}^{\Spin^c}\) is the category with objects links of normal complex surface singularities eqiupped with \Spinc structures and morphisms graph-induced diffeomorphisms between them up to graph-induced isotopy.
\end{definition}
\added{The result of \cite{calcForPlumbing} then becomes a hom space of \(C_{\Link}^{\Spin^c}\) is non-empty if and only if there exists a diffeomorphism between \(Y_1\) and \(Y_2\) (which exists if and only if there is a homeomorphism between \(Y_1\) and \(Y_2\)).
Further note that an extension of a diffeomorphism of the boundary will extend to a diffeomorphism that respects plumbings uniquely away from the regions where the plumbing operation has been done.
On those regions, due to needing to respect both bundle structures we are guaranteed a contractible choice of extensions.}

\subsubsection{Generalized Algebraic Links}\label{subsubsecgenalglinks}
Links of normal complex surface singularities depend on a normal complex surface singularity, whereas algebraic links depend on a singular plane curve but have a smooth ambient space.
For generalized algebraic links we will be combining these notions and will thus need the following concept
\begin{definition}
Let \(S\) be a normal two-dimensional complex variety \(S\) and \(C \subseteq S\) a complex curve inside \(S\).
Let \(\{p_1,p_2,\ldots ,p_n\}\) be a collection of potential singularities in \(S\), which may or may not also be on \(C\).
We say that \((S,C,\{p_1,p_2,\ldots ,p_n\})\) is a \emph{collection of nested singularities} or if \(n=1\) a \emph{nested singularity}.
\end{definition}
Here we have both a complex surface and a curve, and both are allowed to be singular.
Figure \ref{subfig:genalgknot} provides a visualization of a nested singularity and its link next to figures depicting algebraic links and a link of a normal complex surface singularity.
We are now in a position to formally define a generalized algebraic link as follows:

\begin{definition}
Let \((S,C,\{p_1,p_2,\ldots, p_n\})\) be a collection of nested singularities, and let \(\gamma\) be a path connecting  the \(p_i\)  that does not pass through any singularity not in the list.
Take a regular neighborhood \(N\) of \(\gamma\) so that \(C\cap N\) is a regular neighborhood of \(C\cap \gamma\).
Then \(L=\bou N \cap C\)  is a \emph{strong generalized algebraic link} inside the link of singularity \(Y = \bou N\).
A \emph{strong generalized algebraic knot} is a generalized algebraic link so that \(\bou N\cap C\) has only one component.
\end{definition}

The path \(\gamma\) exiting \(C\) would indicate taking a disjoint union of generalized algebric links while \(\gamma\) remaining in \(C\) creates a connected sum of \added{generalized} algebraic links \added{along the components that \(\gamma\) connects}.
If one cares only about algebraic knots one may assume \(\gamma\) lies entirely in \(C\) and if one cares only about prime algebraic knots one may assume that one is taking the link of singularity of a single point.
The primary motivation for including multiple potential singular points and the path \(\gamma\) was to include connect sums in the definition as that is the appropriate scope for knot lattice homology.
\added{Even if one cares only about prime knots, one can often realize these as the dual knots of surgeries on connect sums, and thus it makes sense to include connect sums regardless.}

Eisenbud and Neumann use \added{a single point of singularity} as a part of a general definition of an algebraic link towards the end of the book \cite{planecurvesingularities}.
We are choosing the term generalized algebraic link to avoid confusion because algebraic links are often assumed to be in \(S^3\).
Furthermore, while Eisenbud and Neumann  do not allow for connect sums, their definition is set up to account for links with multiplicity, where the corresponding algebraic curve \(C\) is not presumed to be reduced \cite{planecurvesingularities}.

\begin{definition}
A \emph{resolution of a nested singularity} \((S,C,\{p_1,p_2,\ldots,p_n\})\) is a map \(\pi:\tilde{S}\to S\) that is a resolution of \(S\) at \(\{p_1,\ldots, p_n\}\) and it is a resolution of \(C\), i.e. the total transform of \(C\) in \(\tilde{S}\) has all smooth components with intersections locally modeled in algebraic charts on \(xy=0\).
\added{The \emph{category of resolutions} associated to a nested singularity has objects resolutions of that singularity and morphisms analytic maps \(\phi\colon \tilde{S_1}\to \tilde{S_2}\) that commute with the resolution maps and that take the proper transform of \(C\) in \(\tilde{S_1}\) to the proper transform of \(C\) in \(\tilde{S_2}\).}
\end{definition}

See Figure \ref{subfig:knotsingres} for a representation of a resolution of a normal surface singularity.
This subfigure has an extra sphere in contrast to Figure \ref{subfig:surfacesingres} to highlight that additional blow-ups may be needed after resolving the surface in order to completely resolve the curve.
One can see the proper transform of the curve pierce the extra sphere transversely to highlight that the proper transform would be a fiber of the plumbing resolution.
As with normal complex surface singularities, minimal good resolutions \added{can be defined similarly, } exist, and are unique.
The same argument for existence and uniqueness carries through from the link of a normal complex surface singularity case, except one considers all irreducible components of the total transform of \(C\) and not just the irreducible components of the exceptional fibers in identifying problem points. 

Theorem 24.1 of Eisenbud and Neumann and \cite{linkLattice} establishes the equivalence of resolutions of nested singularities with collections of fibers in negative definite plumbings whose bases have positive intersections \added{where Equation (\ref{Grauert}) has an integral vector solution \(\vec{l}\)} \cite{planecurvesingularities}.
In Equation (\ref{Grauert}), \(\Gc\) represents the adjacency matrix for our plumbing graph as an endomorphism on the vector space freely generated by the weighted vertices of the plumbing graph ,and \(\vec{s}\) is the vector with coefficient on a vertex \(v_i\) equal to the number of unweighted vertices adjacent to \(v_i\).
\begin{equation}
\Gc\vec{l}+\vec{s}=0 \label{Grauert}
\end{equation}
Note that Equation (\ref{Grauert}) represents a relative `resingularization' criterion as to whether the disk fibers can still be carved out by holomorphic equations in the resulting singular complex surface.
This condition essentially requires that our link be integrally null homologous, but it also means that if one is only attempting to reduce to the minimal resolution of the underlying normal complex singularity then one is guaranteed the resulting potentially singular curve is holomorphic.
In particular, restricting the graph to the exceptional fibers of that resolution yields a graph whose link of singularity is \(S^3\) inside which our original link will be integrally null homologous and satisfy the criterion.

\begin{definition}
\added{A \emph{weak nested singularity} \((S,C,\{p_1,\ldots, p_n\})\) is a normal complex surface \(S\) along with a curve \(C\) inside \(S\) which is not holomorphic but whose total transform becomes holomorphic in any resolution of \(S\) over \(\{p_1,\ldots, p_n\}\).
 A \emph{weak generalized algebraic link} is the link of singularity associated to a weak nested singularity.}
\end{definition}

\added{The discussion of the category of resolutions and minimal resolutions in particular associated to strong nested singularities applies just as well to a weak nested singularity, and thus we can focus on weak generalized algebraic links.}

We can then represent \added{weak generalized} algebraic knots using unweighted vertices attached to a plumbing graph, \added{where each unweighted vertex represents a disk in the disk bundle it is adjacent to, or in the case of non-prime knots a boundary connect sum of such disks.}
All the blow-ups and blow-downs that can be done to the plumbings graphs representing resolutions of normal complex surface singularities can be done to the plumbing resolutions for algebraic links (See Figure \ref{fig:trefblowup}); however, we also have moves that come from blowing up a point on the proper transform of \(C\).
One option is if this point is away from the exceptional fiber, which is called \emph{blowing up the unweighted vertex} (see Figure \ref{subfig:vertexblowupspec}).
Another option is if the blown up point is an intersection between the proper transform of C and an exceptional fiber, which results in \emph{blowing up an edge incident to the unweighted vertex} (see Figure \ref{subfig:edgeblowupspec}).

\added{
As with links of normal complex surface singularities, we can form a category to capture which diffeomorphisms actually respect the structure of the links.
Section \ref{BlowUpsSufficeSec} will focus on proving that the hom set for this category is non-empty if and only if there exists a any relative diffeomorphism between the links.
}
\begin{definition}
Let \((X_{G_1},F_1)\) and \((X_{G_2},F_2)\) be plumbings with collections of oriented fibers singled out.
An orientation preserving diffeomorphism \(f\colon X_{G_1}\to X_{G_2}\) \emph{respects the plumbings with fibers} if it respects the plumbing is an orientation preserving diffeomorphism from \(F_1\) to \(F_2\).
Given two weak generalized algebraic links \((Y_1,L_1)\) and \((Y_2,L_2)\) for the weak nested singularities \((S_1,C_1,P_1)\) and \((S_2,C_2,P_2)\) an orientation preserving diffeomorphism \(f\colon (Y_1,L_1)\to (Y_2,L_2)\) is \emph{graph-induced} if it extends to a diffeomorphism that respects the plumbings with fibers over the minimal resolutions
\emph{The groupoid of weak generalized algebraic knots \(C_{\GenAlg}^{\Spin^c}\)} has objects weak generalized algebraic links equipped with \Spinc structures and morphisms the isotopy classes of graph induced diffeomorphisms between them that take one \Spinc structure to the other.
Let \(C_{\GenAlgKnot}^{\Spin^c}\) be the full subcategory of \(C_{\GenAlg}^{\Spin^c}\) whose objects have as ambient three-manifold a rational homology sphere and whose links have a single component.
\end{definition}

\subsection{Checking Blow-Ups Suffice}\label{BlowUpsSufficeSec}

This section will focus on showing that the link type \((Y,L)\) of a \added{weak} generalized algebraic link determines the topology of the minimal good resolution of the nested singularity type used to create it.
In particular, we show that any two good representations of the same generalized algebraic link are equivalent up to a series of blow-ups and blow-downs, which can be done in the algebraic/analytic setting.
This fact may potentially be known to experts, but we could not find a proof in the literature prior to the independent proof by Jackson \cite{otherInvariance}.
The equivalence of link types we are considering is up to homeomorphism of the three-manifold which takes one link to the other.
However, while we are showing that knowledge of the algebraic/analytic setting suffices to get between any two represenations of the same link type, we will have to use knowledge regarding a broader class of three-manifolds and links to do so.

\begin{definition}
A \emph{graph manifold} is a three-manifold which is the boundary of a plumbing of \added{disk bundles of} surfaces (\added{base surfaces may be} orientable or not), with positive or negative intersections between the base surfaces, and any intersection form. The \added{base} surfaces may even have boundary, where the three-manifold is the fiber boundary instead of the boundary of the full four-manifold.

A \emph{graph link} is a link encoded as a graph manifold with boundary, where boundary components are equipped with isotopy classes of \added{simple} closed curves.
The corresponding link is achieved by gluing solid tori to the manifold so that the meridians of the solid tori are sent to the isotopy classes in question, and then using the cores of the solid tori for the link.
\end{definition}

Graph links can be represented as unweighted vertices in plumbing graphs, where the meridinal data is assumed to come from the boundary of the \added{fiber of the unweighted vertex}.
They are often given a different decoration to distinguish them from other vertices with boundary components that do not have meridinal data. Here unweighted vertices that are meant to represent links do not have genus or a number of boundary components listed.
They also are given names with a subscript of 0.

\added{As with our discussion of dual graphs for weak generalized algebraic links, the corresponding link to a generalized algebriac link can be viewed as the boundary of a fiber represented by a vertex \(v_0\).
However, for the purposes of this proof it is easier to consider such vertices \(v_0\) as separate base surfaces with meridinal data on the boundary.
Here, we are setting up what information we will temporarily forget in reducing to normal form from which we can reconstruct our original graph.
Forgetting the meridinal data would be equivalent to considering the link complement, which contains far more information about the link than the ambient three-manifold, which is what remains if one forgets fibers.}

In \cite[p.~304-306]{calcForPlumbing}, Neumann defined a series of moves R1-R8 that can be done in the class of graph manifolds.
In these diagrams, genus information is encoded in brackets with negative numbers representing non-orientable surfaces with those numbers of crosscaps, and if the surface has boundary, the number of boundary components is included as a second number inside the brackets.
Furthermore \cite[p.~311-312]{calcForPlumbing} defines a \emph{normal form} of a graph, satisfying 6 conditions, N1-N6, such that the homeomorphism type of a graph three-manifold determines a graph in normal form.
Since the definitions are long and not central to this paper outside of this subsection, we will not reproduce them here.

In order to show a set of moves go between any two graph representations of the same three-manifold, Neumann describes a strategy of determining a graph representation for that type to which one can reduce any representation given using only those moves \cite{calcForPlumbing}.
Then one shows that \added{the representative} form is determined by the corresponding normal form for the three-manifold.
This strategy was implemented to show that the homeomorphism type of a link of singularity determined homeomorphism type of the minimal resolution.
Neumann also stated (with proof omitted) that this strategy worked to show that any two representations of the same graph link could be related via moves R1-R7 \cite{calcForPlumbing}

Here we will be considering a minimal plumbing representation \(\Delta\) for our generalized algebraic link, which as mentioned in the previous section exists and is unique for the singularity type.
Note that a good resolution will be minimal if and only if it has no \(-1\) weighted vertices of degree 2 or less, counting edges to the unweighted vertices in this count.
Let \(\Gamma\) represent the normal form for the graph with boundary (and no meridinal data).

\begin{proposition}\label{prop:deltaGamma}
One can get from \(\Delta\) to \(\Gamma\) by
\begin{enumerate}
    \item Collapsing the path from \added{ a }unweighted vertex \(v_0\) \added{of degree one} to the first weighted vertex \(v\) (depicted here with weighte \(e\)), which either has degree greater than 2 or is not represented by a sphere.
    This leaves \(v\) as an unweighted vertex representing a surface with boundary and no meridinal data.
		
		\begin{center}
\begingroup%
  \makeatletter%
  \providecommand\color[2][]{%
    \errmessage{(Inkscape) Color is used for the text in Inkscape, but the package 'color.sty' is not loaded}%
    \renewcommand\color[2][]{}%
  }%
  \providecommand\transparent[1]{%
    \errmessage{(Inkscape) Transparency is used (non-zero) for the text in Inkscape, but the package 'transparent.sty' is not loaded}%
    \renewcommand\transparent[1]{}%
  }%
  \providecommand\rotatebox[2]{#2}%
  \newcommand*\fsize{\dimexpr\f@size pt\relax}%
  \newcommand*\lineheight[1]{\fontsize{\fsize}{#1\fsize}\selectfont}%
  \ifx\svgwidth\undefined%
    \setlength{\unitlength}{270.99122429bp}%
    \ifx\svgscale\undefined%
      \relax%
    \else%
      \setlength{\unitlength}{\unitlength * \real{\svgscale}}%
    \fi%
  \else%
    \setlength{\unitlength}{\svgwidth}%
  \fi%
  \global\let\svgwidth\undefined%
  \global\let\svgscale\undefined%
  \makeatother%
  \begin{picture}(1,0.17823305)%
    \lineheight{1}%
    \setlength\tabcolsep{0pt}%
    \put(0,0){\includegraphics[width=\unitlength]{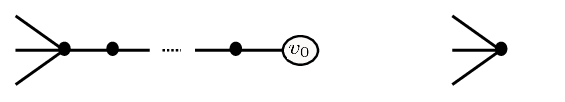}}%
    \put(0.11110474,0.12615564){\makebox(0,0)[lt]{\lineheight{1.25}\smash{\begin{tabular}[t]{l}$e$\end{tabular}}}}%
    \put(0.10859441,0.04416684){\makebox(0,0)[lt]{\lineheight{1.25}\smash{\begin{tabular}[t]{l}$[g,r]$\end{tabular}}}}%
    \put(0.88125846,0.04258515){\makebox(0,0)[lt]{\lineheight{1.25}\smash{\begin{tabular}[t]{l}$[g,r+1]$\end{tabular}}}}%
    \put(0.19575488,0.12044781){\makebox(0,0)[lt]{\lineheight{1.25}\smash{\begin{tabular}[t]{l}$e_1$\end{tabular}}}}%
    \put(0.4106438,0.11954492){\makebox(0,0)[lt]{\lineheight{1.25}\smash{\begin{tabular}[t]{l}$e_n$\end{tabular}}}}%
  \end{picture}%
\endgroup%

    \end{center}

    \item If a component of \(\Delta\) looks like
    
    \begin{center}
\begingroup%
  \makeatletter%
  \providecommand\color[2][]{%
    \errmessage{(Inkscape) Color is used for the text in Inkscape, but the package 'color.sty' is not loaded}%
    \renewcommand\color[2][]{}%
  }%
  \providecommand\transparent[1]{%
    \errmessage{(Inkscape) Transparency is used (non-zero) for the text in Inkscape, but the package 'transparent.sty' is not loaded}%
    \renewcommand\transparent[1]{}%
  }%
  \providecommand\rotatebox[2]{#2}%
  \newcommand*\fsize{\dimexpr\f@size pt\relax}%
  \newcommand*\lineheight[1]{\fontsize{\fsize}{#1\fsize}\selectfont}%
  \ifx\svgwidth\undefined%
    \setlength{\unitlength}{269.61523475bp}%
    \ifx\svgscale\undefined%
      \relax%
    \else%
      \setlength{\unitlength}{\unitlength * \real{\svgscale}}%
    \fi%
  \else%
    \setlength{\unitlength}{\svgwidth}%
  \fi%
  \global\let\svgwidth\undefined%
  \global\let\svgscale\undefined%
  \makeatother%
  \begin{picture}(1,0.2675991)%
    \lineheight{1}%
    \setlength\tabcolsep{0pt}%
    \put(0,0){\includegraphics[width=\unitlength]{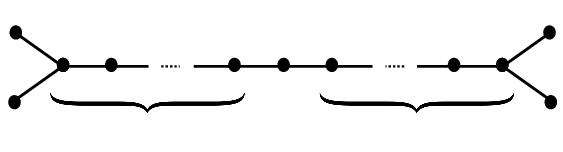}}%
    \put(0.09078803,0.17846343){\makebox(0,0)[lt]{\lineheight{1.25}\smash{\begin{tabular}[t]{l}$-2$\end{tabular}}}}%
    \put(0.00926911,0.11311977){\makebox(0,0)[lt]{\lineheight{1.25}\smash{\begin{tabular}[t]{l}$-2$\end{tabular}}}}%
    \put(0.00567678,0.22987043){\makebox(0,0)[lt]{\lineheight{1.25}\smash{\begin{tabular}[t]{l}$-2$\end{tabular}}}}%
    \put(0.17700393,0.17756535){\makebox(0,0)[lt]{\lineheight{1.25}\smash{\begin{tabular}[t]{l}$-2$\end{tabular}}}}%
    \put(0.39613597,0.17756535){\makebox(0,0)[lt]{\lineheight{1.25}\smash{\begin{tabular}[t]{l}$-2$\end{tabular}}}}%
    \put(0.48414798,0.17936151){\makebox(0,0)[lt]{\lineheight{1.25}\smash{\begin{tabular}[t]{l}$-3$\end{tabular}}}}%
    \put(0.56587343,0.17756535){\makebox(0,0)[lt]{\lineheight{1.25}\smash{\begin{tabular}[t]{l}$-2$\end{tabular}}}}%
    \put(0.87143627,0.17846343){\makebox(0,0)[lt]{\lineheight{1.25}\smash{\begin{tabular}[t]{l}$-2$\end{tabular}}}}%
    \put(0.96200672,0.10862936){\makebox(0,0)[lt]{\lineheight{1.25}\smash{\begin{tabular}[t]{l}$-2$\end{tabular}}}}%
    \put(0.95931253,0.22897234){\makebox(0,0)[lt]{\lineheight{1.25}\smash{\begin{tabular}[t]{l}$-2$\end{tabular}}}}%
    \put(0.78522038,0.17756535){\makebox(0,0)[lt]{\lineheight{1.25}\smash{\begin{tabular}[t]{l}$-2$\end{tabular}}}}%
    \put(0.25362187,0.02415214){\makebox(0,0)[lt]{\lineheight{1.25}\smash{\begin{tabular}[t]{l}$b$\end{tabular}}}}%
    \put(0.73285384,0.02415214){\makebox(0,0)[lt]{\lineheight{1.25}\smash{\begin{tabular}[t]{l}$c$\end{tabular}}}}%
  \end{picture}%
\endgroup%

		\end{center}
    
    where \(b\geq 0\) and \(c\geq 1\), then the corresponding component of \(\Gamma\) looks like 
    
    \begin{center}
\begingroup%
  \makeatletter%
  \providecommand\color[2][]{%
    \errmessage{(Inkscape) Color is used for the text in Inkscape, but the package 'color.sty' is not loaded}%
    \renewcommand\color[2][]{}%
  }%
  \providecommand\transparent[1]{%
    \errmessage{(Inkscape) Transparency is used (non-zero) for the text in Inkscape, but the package 'transparent.sty' is not loaded}%
    \renewcommand\transparent[1]{}%
  }%
  \providecommand\rotatebox[2]{#2}%
  \newcommand*\fsize{\dimexpr\f@size pt\relax}%
  \newcommand*\lineheight[1]{\fontsize{\fsize}{#1\fsize}\selectfont}%
  \ifx\svgwidth\undefined%
    \setlength{\unitlength}{55.478863bp}%
    \ifx\svgscale\undefined%
      \relax%
    \else%
      \setlength{\unitlength}{\unitlength * \real{\svgscale}}%
    \fi%
  \else%
    \setlength{\unitlength}{\svgwidth}%
  \fi%
  \global\let\svgwidth\undefined%
  \global\let\svgscale\undefined%
  \makeatother%
  \begin{picture}(1,0.63869476)%
    \lineheight{1}%
    \setlength\tabcolsep{0pt}%
    \put(0,0){\includegraphics[width=\unitlength]{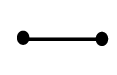}}%
    \put(0.06426818,0.43678358){\makebox(0,0)[lt]{\lineheight{1.25}\smash{\begin{tabular}[t]{l}$b+1$\end{tabular}}}}%
    \put(0.77317688,0.43616932){\makebox(0,0)[lt]{\lineheight{1.25}\smash{\begin{tabular}[t]{l}$c+1$\end{tabular}}}}%
    \put(0.0792886,0.03912417){\makebox(0,0)[lt]{\lineheight{1.25}\smash{\begin{tabular}[t]{l}$[-1]$\end{tabular}}}}%
    \put(0.75257452,0.03980797){\makebox(0,0)[lt]{\lineheight{1.25}\smash{\begin{tabular}[t]{l}$[-1]$\end{tabular}}}}%
  \end{picture}%
\endgroup%

		\end{center}
    
    \item Any components of \(\Delta\) that look like
    
    \begin{center}
\begingroup%
  \makeatletter%
  \providecommand\color[2][]{%
    \errmessage{(Inkscape) Color is used for the text in Inkscape, but the package 'color.sty' is not loaded}%
    \renewcommand\color[2][]{}%
  }%
  \providecommand\transparent[1]{%
    \errmessage{(Inkscape) Transparency is used (non-zero) for the text in Inkscape, but the package 'transparent.sty' is not loaded}%
    \renewcommand\transparent[1]{}%
  }%
  \providecommand\rotatebox[2]{#2}%
  \newcommand*\fsize{\dimexpr\f@size pt\relax}%
  \newcommand*\lineheight[1]{\fontsize{\fsize}{#1\fsize}\selectfont}%
  \ifx\svgwidth\undefined%
    \setlength{\unitlength}{139.6093245bp}%
    \ifx\svgscale\undefined%
      \relax%
    \else%
      \setlength{\unitlength}{\unitlength * \real{\svgscale}}%
    \fi%
  \else%
    \setlength{\unitlength}{\svgwidth}%
  \fi%
  \global\let\svgwidth\undefined%
  \global\let\svgscale\undefined%
  \makeatother%
  \begin{picture}(1,0.40163643)%
    \lineheight{1}%
    \setlength\tabcolsep{0pt}%
    \put(0,0){\includegraphics[width=\unitlength]{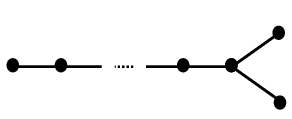}}%
    \put(0.00898349,0.23469969){\makebox(0,0)[lt]{\lineheight{1.25}\smash{\begin{tabular}[t]{l}$e_k$\end{tabular}}}}%
    \put(0.77252872,0.23643408){\makebox(0,0)[lt]{\lineheight{1.25}\smash{\begin{tabular}[t]{l}$e$\end{tabular}}}}%
    \put(0.92662692,0.09463212){\makebox(0,0)[lt]{\lineheight{1.25}\smash{\begin{tabular}[t]{l}$-2$\end{tabular}}}}%
    \put(0.92142387,0.32703996){\makebox(0,0)[lt]{\lineheight{1.25}\smash{\begin{tabular}[t]{l}$-2$\end{tabular}}}}%
    \put(0.60602753,0.23643408){\makebox(0,0)[lt]{\lineheight{1.25}\smash{\begin{tabular}[t]{l}$e_1$\end{tabular}}}}%
    \put(0.17589952,0.2329653){\makebox(0,0)[lt]{\lineheight{1.25}\smash{\begin{tabular}[t]{l}$e_{k-1}$\end{tabular}}}}%
  \end{picture}%
\endgroup%

		\end{center}
    
    with \(k>1\) and \(e_i\leq 2\) for all \(i\), appear as is in \(\Gamma\).
    
    \item For all remaining components, substitute portions of the graph that look like the figures on the left with the portions on the right for \(e\leq -3\) and \(b\geq 1\).
    
    \begin{center}
\begingroup%
  \makeatletter%
  \providecommand\color[2][]{%
    \errmessage{(Inkscape) Color is used for the text in Inkscape, but the package 'color.sty' is not loaded}%
    \renewcommand\color[2][]{}%
  }%
  \providecommand\transparent[1]{%
    \errmessage{(Inkscape) Transparency is used (non-zero) for the text in Inkscape, but the package 'transparent.sty' is not loaded}%
    \renewcommand\transparent[1]{}%
  }%
  \providecommand\rotatebox[2]{#2}%
  \newcommand*\fsize{\dimexpr\f@size pt\relax}%
  \newcommand*\lineheight[1]{\fontsize{\fsize}{#1\fsize}\selectfont}%
  \ifx\svgwidth\undefined%
    \setlength{\unitlength}{329.72538757bp}%
    \ifx\svgscale\undefined%
      \relax%
    \else%
      \setlength{\unitlength}{\unitlength * \real{\svgscale}}%
    \fi%
  \else%
    \setlength{\unitlength}{\svgwidth}%
  \fi%
  \global\let\svgwidth\undefined%
  \global\let\svgscale\undefined%
  \makeatother%
  \begin{picture}(1,0.40633956)%
    \lineheight{1}%
    \setlength\tabcolsep{0pt}%
    \put(0,0){\includegraphics[width=\unitlength]{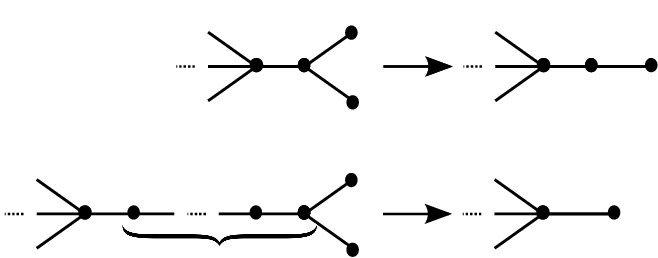}}%
    \put(0.45416183,0.31455472){\makebox(0,0)[lt]{\lineheight{1.25}\smash{\begin{tabular}[t]{l}$e$\end{tabular}}}}%
    \put(0.52714805,0.25257939){\makebox(0,0)[lt]{\lineheight{1.25}\smash{\begin{tabular}[t]{l}$-2$\end{tabular}}}}%
    \put(0.5172057,0.35291824){\makebox(0,0)[lt]{\lineheight{1.25}\smash{\begin{tabular}[t]{l}$-2$\end{tabular}}}}%
    \put(0.38366342,0.31455472){\makebox(0,0)[lt]{\lineheight{1.25}\smash{\begin{tabular}[t]{l}$e_i$\end{tabular}}}}%
    \put(0.37351927,0.24373799){\makebox(0,0)[lt]{\lineheight{1.25}\smash{\begin{tabular}[t]{l}$[g_i,r_i]$\end{tabular}}}}%
    \put(0.87155249,0.31165249){\makebox(0,0)[lt]{\lineheight{1.25}\smash{\begin{tabular}[t]{l}$e+1$\end{tabular}}}}%
    \put(0.81330797,0.31068507){\makebox(0,0)[lt]{\lineheight{1.25}\smash{\begin{tabular}[t]{l}$e_i$\end{tabular}}}}%
    \put(0.79735937,0.23890092){\makebox(0,0)[lt]{\lineheight{1.25}\smash{\begin{tabular}[t]{l}$[g_i,r_i]$\end{tabular}}}}%
    \put(0.96530422,0.2534121){\makebox(0,0)[lt]{\lineheight{1.25}\smash{\begin{tabular}[t]{l}$[-1]$\end{tabular}}}}%
    \put(0.97950579,0.31397031){\makebox(0,0)[lt]{\lineheight{1.25}\smash{\begin{tabular}[t]{l}$0$\end{tabular}}}}%
    \put(0.52224649,0.03220301){\makebox(0,0)[lt]{\lineheight{1.25}\smash{\begin{tabular}[t]{l}$-2$\end{tabular}}}}%
    \put(0.52101087,0.13641149){\makebox(0,0)[lt]{\lineheight{1.25}\smash{\begin{tabular}[t]{l}$-2$\end{tabular}}}}%
    \put(0.13505667,0.09514573){\makebox(0,0)[lt]{\lineheight{1.25}\smash{\begin{tabular}[t]{l}$e_i$\end{tabular}}}}%
    \put(0.10936404,0.01907201){\makebox(0,0)[lt]{\lineheight{1.25}\smash{\begin{tabular}[t]{l}$[g_i,r_i]$\end{tabular}}}}%
    \put(0.19542742,0.096311){\makebox(0,0)[lt]{\lineheight{1.25}\smash{\begin{tabular}[t]{l}$-2$\end{tabular}}}}%
    \put(0.44818728,0.09317573){\makebox(0,0)[lt]{\lineheight{1.25}\smash{\begin{tabular}[t]{l}$-2$\end{tabular}}}}%
    \put(0.37285181,0.08760431){\makebox(0,0)[lt]{\lineheight{1.25}\smash{\begin{tabular}[t]{l}$-2$\end{tabular}}}}%
    \put(0.33347455,-0.00762026){\makebox(0,0)[lt]{\lineheight{1.25}\smash{\begin{tabular}[t]{l}$b$\end{tabular}}}}%
    \put(0.91739225,0.09498324){\makebox(0,0)[lt]{\lineheight{1.25}\smash{\begin{tabular}[t]{l}$b$\end{tabular}}}}%
    \put(0.81548466,0.09042953){\makebox(0,0)[lt]{\lineheight{1.25}\smash{\begin{tabular}[t]{l}$e_i$\end{tabular}}}}%
    \put(0.79735929,0.02626383){\makebox(0,0)[lt]{\lineheight{1.25}\smash{\begin{tabular}[t]{l}$[g_i,r_i]$\end{tabular}}}}%
    \put(0.90892197,0.03316469){\makebox(0,0)[lt]{\lineheight{1.25}\smash{\begin{tabular}[t]{l}$[-1]$\end{tabular}}}}%
  \end{picture}%
\endgroup%

		\end{center}
Substitutions 2-4 are the substitutions in Theorem 8.2 of \cite{calcForPlumbing}.
\end{enumerate}
Furthermore graph \(\Delta\) can be recovered from \(\Gamma\) with a choice of meridinal slopes on the boundary components.
\end{proposition}
\begin{proof}
Collapsing the path from an unweighted vertex \(v_0\) to the first vertex \(v\) either not represented by a sphere or of degree greater than or equal to 3 can be achieved by a sequence of R8 moves, and Theorem 8.2 of \cite{calcForPlumbing} demonstrates the validity of moves (2)-(4).
What needs to be shown is that after these moves, the resulting graph, which we will call \(\Gamma'\) is in normal form.

\added{Now, we} check N1\added{, which states that the moves R1-R8 cannot be applied except on components of the form in substitution 3.}
R8 moves require degree 1 vertices representing disks which we have removed in Substitution 1 and do not reappear in Substitutions 2-4. 
The starting positions for substitutions 2-4 include any starting case for an R2 moves on \(\Gamma'\) that could violate N1.
Furthermore, we started with no zero framed spheres, and the moves do not produce any 0 framed spheres.
Thus there are no R3-R6 moves available on \(\Gamma'\).
Finally none of the starting components of R7 moves could be produced by the moves described above, nor are any of them negative definite.
Thus there will be no R7 moves on \(\Gamma'\).
We assumed that there were no R1 moves available, and the moves provided do not produce any -1 or +1 framed \added{spheres}, and thus here will be no R1 moves available on \(\Gamma'\).
Hence, \(\Gamma'\) satisfies N1.

The chains, paths of vertices representing spheres of degree 1 or 2, produced in this process have weight less than or equal to \(-2\).
\added{All other chains of \(\Gamma'\) came from chains of \(\Delta\) and thus also have weights less than or equal to \(-2\)}.
\added{This is precisely what is needed for \(\Gamma'\) to satisfy} N2.

The Substitutions 2 and 4 above precisely prevent any portion of \(\Gamma'\) from violating N3.
\added{N4 concerns 0 framed bundles over an \(\mathbb{RP}^2\), which do not exist in \(\Delta\)}.
As such in \(\Gamma'\) these would be all produced by Substitution 4, which do not violate N4.
No components violating N5 are produced because the declaration in Substitution 3 specifically prevents Substitution 4 from producing such components.
Lemma 8.1 of \cite{calcForPlumbing} prevents the forbidden components of N6 from appearing.
Thus, having established that \(\Gamma'\) is in normal form, we can conclude that \(\Gamma'=\Gamma\).

To show that \(\Delta\) can be recovered from \(\Gamma\), we need to show these steps are reversible.
For Substitutions 2 and 4 this is clear: substitute the figure on the left for the figure on the right instead of right for left, and for Substitution 3, we do nothing.
What remains to be checked is if we can recover a single negative definite minimal \(\Delta\) from the information of meridians on the boundary components of \(\Gamma\).
\added{A given vertex \(v\) that has boundary components in \(\Gamma\) corresponds to a star-shaped subgraph \(\Delta'\), one ray (potentially of length zero) for every boundary component of \(v\), of a \(\Delta\) which could produce \(\Gamma\).
In particular rays of length greater than zero come from applying the first operation in this proposition, and a ray of length zero would correspond to the central vertex being unweighted in \(\Delta\) and the first operation did not apply (due to having degree greater than one).}

We will show that the meridinal data determines a unique minimal negative definite \(\Delta'\), except perhaps allowing \(v\) to be a negative \(-1\) sphere of degree 1 or 2 in the case where \(v\) is adjacent to other vertices in \(\Gamma\) and thus in any \(\Delta\).
In that case, the degree of \(v\) in \(\Delta\) would need to be at least 3, else it would have not been the vertex chosen in operation 1, and thus it could not be blown down in \(\Delta\).
Thus \(\Delta'\) must then appear as the corresponding subgraph in a minimal negative definite \(\Delta\) that produces \(\Gamma\).

In the case where there is a single ray of length zero, nothing more needs to be done. If there is a ray of length zero along with rays of greater length, the star decomposes as a  connect sum of links each of which is modeled off of the two core knots in a lens space \(L(p,q)\).
The lens space is determined based off of the meridinal data.
The minimal model for the lens space is a linear graph and adjoining unweighted vertices at either end. This can be viewed as coming from a specific seifert fibering with two singular fibers, referring us to the case below.

If there is no ray of length zero then \(\Delta'\) corresponds to a Seifert fibered space with  Seifert fibration determined by the meridinal data.
Furthermore, by the fact that the larger \(\Delta\) exists and is negative definite this Seifert fibered space supports a negative definite \(\Delta'\).
By the argument in Theorem 5.2 of \cite{seifertStuff}, we get a unique negative definite graph which is minimal way from the central vertex, in particular because the Seifert invariants being in normalized form does not place a restriction on the central vertex.

\end{proof}

The proposition above with Neumann's result that the homeomorphism type determines the normal form \(\Gamma\) \cite[Theorem 3.1]{calcForPlumbing} gives us a proof of Theorem \ref{BlowUpsSuffice}.
In particular, we are use Neumann's result to establish that the normal form for graphs describing the knot complements must be the same and Proposition \ref{prop:deltaGamma} to then recover the minimal graph for the link from the meridinal data on those knot complements.
\added{Having isomorphic graphs ensures that some diffeomorphism respecting fibers exists, but also given a diffeomorphism that extends  over the \(X_\Gamma\) as a plumbing respecting diffeomorphism up to isotopy that will extend  over the \(X_{\Delta}\) as as the moves for modifying the graphs are based on local models that the diffeomorphism respects.
In particular \(C_{\GenAlgKnot}^{\Spin^c}\) will have a morphism between two knots in three-manifolds equipped with \Spinc structures if there is some diffeomorphism respecting those \Spinc structures and knots.}


\section{Motivations for the Equivalence Type} \label{FloerSec}

Here we provide some more discussion of our target category of filtered spaces, and in particularly modeling within it some behavior we would expect from knot lattice homotopy based on its connection to Heegaard Floer homology and knot Floer homology.
While we will not go into detail on the Floer theory, a key aspect we would like to emulate are certain additional maps such as those from involutive Heegaard Floer homology and involutive knot Floer homology, and thus we need to state how we will encode such maps here.
However, these maps interact with the \Spinc structures on the underlying three-manifold, so we first discuss those more.

The exact definition of a \Spinc structure will not be needed here, but for a manifold \(M\), \(\SpincX{M}\) is affinely modeled on \(H^2(M;\Z)\) and there exists a natural map \(c_1\colon\SpincX{M} \to H^2(M;\Z)\).
In particular for \(\sfk \in \SpincX{X}\), if \(\sfk|_Y=\tf\) then \(c_1(\sfk)\) maps to \(c_1(\tf)\)
There also exists an involution \(\sfk\to \overline{\sfk}\) called conjugation, and \(c_1(\overline{\sfk}) = - c_1(\sfk) \).

In the context of lattice homology, it can be useful to think of \Spinc structures on \(Y\) in terms orbits of characteristic cohomology classes of \(X\) \added{under the action of \(H_2(X;\Z)\)}.
In the case where \(X\) is a plumbing of \added{disk bundles of} surfaces \added{with plumbing graph \(G\)} the first homology of \(X\) is isomorphic to the free part of \added{the first homology of} \(Y\), and as such, \(H^2(X;\Z) \cong H_2(X;\Z)^{\ast}\).
An element \(K \in H^2(X;\Z)\) is characteristic if \(K([S]) \equiv [S]\cdot [S] \pmod 2\) for all homology classes \([S]\).
It suffices to check that \(K(v_i) \equiv e_i \pmod 2\) for all vertices \(v_i\) of \(G\).
This is equivalent to \(K\) reducing mod 2 to the second Steifel-Whitney class  \(w_2(X)\).
Denote the set of characteristic cohomology classes on \(X\) by \(\Char(G)\).
For many four-manifolds including plumbings, the map \(c_1:\SpincX{X} \to \Char(G)\) is a bijection.
In this context,  \([S] \in H_2(X;\Z)\) can act on \(K \in \Char(G)\) by  \(K + [S]:= K+ 2\PD[S]\).
The orbits of this action on  \(\Char(G)\) correspond to the torsion elements of \(\SpincX{Y}\), and we will denote the orbit of \(\Char(G)\) corresponding to an element \(\tf \in \SpincX{Y}\) by \(\Char(G,\tf)\).



We will use the following definition to provide the level of structure we will be looking for in knot lattice homology as an invariant, including what it would mean for two such structures to be equivalent.
\begin{definition}\label{knotLatPackdef}
Given a knot \(K\) in a \added{rational homology} three-\added{sphere} \(Y\), a \emph{complete package for \((Y,K)\)} is a triple \((C,J_{K},J)\), where
\begin{enumerate}
    \item \(C\) is a bigraded \(\Z[U,V]\) complex that splits as \(\bigoplus_{\tf \in \SpincX{Y}}(C,\tf)\) with \(U\) having grading shift \((-2,0)\) and \(V\) having grading shift \((0,-2)\). \added{The bigrading is absolutely \(\Q\times \Q\) graded and each individual \((C,\tf)\) is relatively \(\Z\times \Z\) graded.}
    \item A map \(J_K:C \to C\), which is skew-\(\Z[U,V]\) linear, i.e. swaps the actions of \(U\) and \(V\). It should send \((C,\tf)\) to \((C,\overline{\tf+\PD[K]})\), and \(J_{K}^2\) is chain homotopic to the identity.
    \item A \(\Z[U,V]\)-linear map \(\Ii:(C,\tf)\otimes{\Z[U,V]/(V-1)}\to (C,\overline{\tf})\otimes{\Z[U,V]/(V-1)}\). Further, \(\Ii^2\) is chain homotopic to the identity.
\end{enumerate}
We will consider two complete packages \((C_1,J_{K,1},J_1)\) and \((C_2,J_{K,2},J_2)\) equivalent if there exists a chain homotopy equivalence \(f:C_1\to C_2\) that respects the gradings and the \Spinc structures, commutes with \(J_{v_0}\) up to homotopy and, once included into the tensor with \(\Z[U,V]/(V-1)\), commutes with \(J\) up to homotopy. \added{Complete packages for a knot \((Y,K)\) form a simplicial category.}
\end{definition}

\added{In the context of a graph knot where our knot is denoted by an unweighted vertex \(v_0\), we will refer to \(J_K\) as \(\Jj\).}

\begin{remark}
\added{This parallels the formal properties of knot Floer in relation to Heegaard Floer.
The maps \(J_K\) and \(\Ii\) are tracked as potentially related to the involutive maps \(\iota_K\) and \(\iota\), and together \(J_K\) and \(J\) recover the map \(\Gamma\) tracked by \OSS{} in their master complex \cite{knotsAndLattice}.}
However, the map \(\iota_K\), while conventionally called the knot involution, is not necessarily an involution and squares to the Sarkar map \cite{InvolSetup}.
Furthermore, even if for two knots \(K_1\) and \(K_2\) it is known to be an honest involution, this may become false when taking the connect sum, as the knot involution for \(K_1\# K_2\) is not merely \(\iota_{K_1}\otimes \iota_{K_2}\) but has an extra term coming from formal derivative maps on the component chain complexes \cite{ConnectSumInvol}.
However, the map \(J_{v_0}\) constructed for knot lattice which most closely appears to be \(\iota_K\)'s counterpart is always involution, informing our definition of a complete package here.
The exact relationship between these maps is not known, and for example, they may potentially agree for prime knots.
\end{remark}

In order to express a complete package in homotopy theoretic terms, we will be replacing our \(\Z[U,V]\) chain complex with a doubly-filtered topological space, where the action by \(U\) and \(V\) become inclusions on the filtration.
We will refer to a doubly filtered space with \(X_{\ast\ast}\) where the subspace at level \((m,n)\) is given by \(X_{m,n}\), and inclusions filter downward, so \(X_{m_1,n_1}\subseteq X_{m_2,n_2}\) if \(m_1\geq m_2\) and \(n_1\geq n_2\).
\begin{definition}
\added{The category \(\Filt_{[\Q\times\Q:2\Z\times 2\Z]}\) has objects doubly filtered spaces whose filtrations are indexed by a coset of \(2\Z \times 2\Z\) in \(\Q\times \Q\).
The nonempty morphism spaces are between doubly filtered spaces indexed by the same coset, in which case the morphisms spaces are given by the filtered maps between them.
Similarly \(\Filt_{[\Q:2\Z]}\) will have objects with filtrations indexed by cosets of \(2\Z\) in \(\Q\) and morphism spaces defined similarly.}
\end{definition}
\added{Note that \(\Filt_{[\Q\times\Q:2\Z\times 2\Z]}\) naturally forms a topologically and thus simplicially enriched category.
The following provide some useful functors on \(\Filt_{[\Q\times \Q:2\Z\times 2\Z]}\).}
\begin{definition}
The \emph{first and second single filtrations} are respectively the continuous functors \(p_{1,!}\) and \(p_{2,!}\) from \(\Filt_{[\Q\times \Q:2\Z\times 2\Z]}\) to \(\Filt_{[\Q:2\Z]}\) defined on a double filtered space \(X_{\ast\ast}\) as
\begin{align*}
p_{1,!}(X_{\ast\ast})_m &= \bigcup_{n}X_{m,n} \\
p_{2,!}(X_{\ast\ast})_n&=\bigcup_mX_{m,n}.
\end{align*}
Additionally
\[\sigma\colon \Filt_{[\Q\times \Q:2\Z\times 2\Z]} \to \Filt_{[\Q\times \Q:2\Z\times 2\Z]},\]
is the continuous functor so that \(\sigma(X_{\ast\ast})_{m,n}=X_{n,m}\).
A continuous function between two doubly filtered spaces \(X_{\ast\ast}\) to \(Y_{\ast\ast}\) is \emph{skew-filtered} if it is a doubly filtered map from \(X_{\ast\ast}\) to \(\sigma(Y_{\ast\ast})\).
Finally there is a continuous bifunctor
\[ \otimes\colon \Filt_{[\Q\times\Q:2\Z\times 2\Z]} \times \Filt_{[\Q\times\Q:2\Z\times 2\Z]}\to \Filt_{[\Q\times\Q:2\Z\times 2\Z]}\]
so that
\[(X_{\ast\ast}\otimes Y_{\ast\ast})_{m,n} = \bigcup_{\substack{m_1+m_2=m\\n_1+n_2=n}} X_{m_1,n_1}\times Y_{m_2,n_2}\]
\end{definition}
\added{If \(X_{\ast\ast}\) is expressed as the superlevel sets of a function \((h_U,h_V)\colon X\to \Q\times \Q\) then \(p_{1,!}\) remembers only \(h_U\), \(p_{2,!}\) only \(h_V\), \(\sigma\) swaps \(h_U\) and \(h_V\), and \(\otimes\) adds the height functions on \(X_{\ast\ast}\) and \(Y_{\ast\ast}\) from the individual coordinates on \(X\times Y\).
A continuous map \(f\colon X_{\ast\ast}\to Y_{\ast\ast}\) is filtered if and only if \(h_U(x)\leq h_U(f(x))\) for all \(x\in X\).}
\added{Similarly these work if instead of doubly filtered spaces one considers doubly filtered CW-complexes with cellular maps or simplicial/cubical sets and one gets simplicially enriched categories.
Note that filtratons on CW-complexes should respect the CW-strucutre, i.e. the CW-structure on an individual layer of the filtraiton comes from the CW-structure on the ambient space.
}

\begin{definition}
Given a knot \(K\) in a rational homology three-sphere \(Y\) a \emph{complete homotopy package for \((Y,K)\)} is a tuple \((X^\tf_{\ast\ast},J_K^{\tf,\nat},\Ii^{\tf,\nat})\), where \(\tf\) itself ranges across \(\SpincX{Y}\) where 
\begin{enumerate}
\item Each \(X_{\ast\ast}^{\tf}\) is a doubly filtered CW-complex, with the indexing filtration for each \(\tf\) potentially coming from a different coset of \(2\Z\times 2\Z\) in \(\Q\times \Q\).
\item Each \(J_K^{\tf,\nat}\) is a skew filtered cellular morphism in \(\Filt_{[\Q\times\Q:2\Z\times2\Z]}\)
\[\Jj^{\tf,\nat}:X^{\tf}_{\ast\ast}\to X^{\overline{\tf+\PD[K]}}_{\ast\ast}\]
with \(\left(\Jj^{\nat}\right)^2\) doubly-filtered homotopic to the identity.
\item the morphism in \(\Filt_{[\Q:2\Z]}\)
\[\Ii^{\tf,\nat}:p_{1,!}\left(X_{\ast\ast}^{\tf}\right)\to p_{1,!}\left(X^{\overline{\tf}}_{\ast\ast}\right)\]
is cellular, and it takes \(X_{\tf}\) to \(X_{\overline{\tf}}\) with \(\left(\Ii^{\nat}\right)^2\) singly-filtered homotopic to the identity.
\end{enumerate}
The complete packages associated to a knot \((Y,K)\) form a topological category with morphisms given by filtered maps that respect the decomposition by \Spinc structure and commute up to homotopy with both the \(J_K^{\tf,\nat}\) and the \(J^{\tf,\nat}\).
\end{definition}
We will drop the \(\tf\) subscript when discussing these constructions as a whole, treating it similarly to a coproduct, though technically the different cosets used on the different components mean it is not a true coproduct.



Given a complete homotopy package \((X_{\ast\ast}^{\tf},J_K^{\tf,\nat},\Ii^{\tf,\nat})\), we can recover a complete package as 
\begin{equation}
\left(\bigoplus_{m,n,\tf}C_{\bullet}\left(X_{m,n}^{\tf}\right),\shrpt{J_K^{\nat}},\shrpt{\Ii^{\nat}} \right)\label{shrptPackage},
\end{equation}
where  \(C_{\bullet}\) is the cellular chain complex, given a \(\Z[U,V]\) structure through persistence.
In particular, while \(\bigoplus_{m,n}C_{\bullet}\left(X_{m,n}^{\tf}\right)\) is naturally an abelian group, it gains the structure of a \(\Z[U,V]\) module by letting \(U\) act by the inclusion \(X_{m,n}^{\tf}\subseteq X_{m-2,n}^{\tf}\) and letting \(V\) act by the inclusion \(X_{m,n}\subseteq X_{m,n-2}\).
Elements of \(C_d\left(X_{m,n}\right)\) have the gradings \(\gr_U = d + m\) and \(\gr_V= x+n\).
\added{Furthermore, \(J_K^{\nat}\) being skew-filtered means that \(\shrpt{J_K^{\nat}}\) is skew-\(\Z[U,V]\) linear due to \(C_{\bullet}\) taking \(\sigma\) to the functor swapping the gradings \(\gr_U\) and \(\gr_V\) and the actions of \(U\) and \(V\).
Additionally \(C_{\bullet}\) takes the functor \(p_{1,!}\) to tensoring with \(\Z[U,V]/(V-1)\), allowing \(\Ii^{\nat}\) to be taken to \(\Ii\).}

We will talk about a doubly-filtered cellular construction \(Y^{\nat}\) \emph{realizing} a \(\Z[U,V]\)-linear construction \(Y\) if, following the format above, \(\shrpt{Y^{\nat}}\)  yields \(Y\).
We use the word construction loosely here as there are a number of different things we would like to pass back and forth between these two settings, and in each case we will use the idea of realization to signal our intent.
For example, the doubly filtered cellular space \(X_{m,n}\)  realizes the chain complex \(C\), the skew-filtered map \(\Jj^{\nat}\) realizes the map \(\Jj\), and the singly-filtered map \(\Ii^{\nat}\) realizes the map \(\Ii\) on \(C\) with coefficients in \(\Z[U,V,V^{-1}]\).
The complete homotopy package \((X,\{\Fc_{t_1,t_2}\},\Jj^{\nat},\Ii^{\nat})\) realizes the complete package in Equation (\ref{shrptPackage}).
More generally, we can talk about doubly-filtered continuous cellular maps \(f^{\nat}\) \emph{realizing} chain maps \(f\) if \(\shrpt{f^{\nat}}=f\) and doubly filtered cellular homotopies \(H^{\nat}\) \emph{realizing} chain homotopies \(H\) if \(\shrpt{H^{\nat}}=H\).
We will have that all the components of the proof of Theorem \ref{topInvarianceResult} realize the components of the proof of Theorem \ref{invarianceResult}.
\added{However, the cellular approximation theorem, means that unless we are trying to compare with a given map of chain complexes, we will not be forcing maps to be cellular as we go.}

\section{A Cubical Construction of Lattice Homology}\label{introToLatticeHom}

\subsection{The Lattice Chain Complex}

Lattice homology \(\HFb(Y,\tf)\) is an invariant for a rational homology link of singularity \(Y\) equipped with a \Spinc structure and computed using a plumbing resolution for the corresponding collection of singularities \((S,\{p_1,\ldots,p_n\})\).
We will generally reference this plumbing resolution as \(X\) and assume it has graph \(G\) with vertices \(v_1,\ldots, v_n\)  and that \(v_i\) has weight \(v_i^2\).
Since the information of \(X\) is often presented in the form of \(G\), we will often reference \(G\) directly and reference a vertex \(v_i\) as if it were also the corresponding homology class of a sphere in \(X\).
Because we are assuming \(Y\) is a rational homology sphere, all vertices represent spheres and \(G\) has no cycles.
N\'emethi introduced lattice homology in \cite{latticeCohomNormSurf} and expanded on it in \cite{twoExactSequences}.
We will primarily be drawing notation from \cite{knotsAndLattice} and the framework from \cite{latticeCohomNormSurf}.

Letting \(\tf\) be a \Spinc structure of \(Y\) represented as an orbit of characteristc cohomology classes \(\Char(G,\tf)\) under the action by \(H_2(X;\Z)\), \added{we will now define some auxilary information on the way to defining \(\CFb(G,\tf)\) as a chain complex.
This will start with the generators, then the gradings, before finally covering the differential.}

\subsubsection{Generators of the Lattice Chain Complex}

\added{
Let \(\Qc(G,\tf)\) denote elements ordered pairs \([K,E]\) where \(K\in \Char(G,\tf)\) and \(E\) is a subset of the vertices of \(G\), and let  \(\Qc_{k}(G,\tf)\) be the subset of \([K,E] \in \Qc(G,\tf)\) such that \(|E|=k\).
We will often denote \([K,\emptyset]\) simply by \(K\).
As a bit of foreshadowing, the elements of \(\Qc(G,\tf)\) will be called \emph{the cubes of} \(\Char(G,\tf)\) and \(\Qc_k(G,\tf)\) the \emph{\(k\)-dimensional cubes}.
As a \(\Z[U]\) module our complex \(\CFb(G,\tf)\) is generated by \(\Qc(G,\tf)\) with the \(\delta\) grading \(k\) submodule \(\CFb_k(G,\tf)\) generated by \(\Qc_k(G,\tf)\). 
The \(\delta\) grading will be one of two homological gradings.
}

\subsubsection{The Maslov Grading on the Lattice Chain Complex}\label{hUsubsec}
\added{
We will now define a height function \(h_U\colon \Qc \to \Q\) so that the image lies entirely in a single coset of \(2\Z\).
To start, define \(h_U\) on \(\Char(G,\tf)\) to be
\[ h_U(K):= \frac{K^2+s}{4}\]
where \(s\) is the rank of \(H^2(X_G;\Z)\).
This is chosen so that \(h(c_1(\sfk))\) calculates the grading shift of the Heegaard Floer map \(F_{W,\sfk}^-\colon HF^-(S^3,0)\to HF^-(Y,\tf)\) where \(W\) is the cobordism achieved by removing a ball from \(X\).
We will then choose the height function on the remaining cubes as
\begin{align*}
h_U([K,E]) &= \min\{h_U\left(K'\right)\, :\, I \subset E\}.
\end{align*}
On an orbit \(\Char(G,\tf)\), the values of \(h_U(K)\) will all be in the same coset of \(2\Z\) in \(\Q\), guaranteeing that \(a_v\) and \(b_v\) are always integers.
Furthermore, at this point we can verify that if \(c_1(\tf)=0\) then this coset is precisely \(2\Z\).
}

\added{
The role of the height function \(h_U\) will play a similar role to the weight system \(w_K\) in \cite{latticeCohomNormSurf} which is defined relative to a choice of \(K\in \Char(G,\tf)\) and with cubes defined in \(H_2(X_G;\Z)\).
In particular, we have the relation:
\[ w_K([l,E]) = \frac{h_U(K)- h_U([K+l,E])}{2}.\]
Much of the discussion of weight systems in \cite{latticeCohomNormSurf} can be carried over here with the changes that heights should go to a particular coset of \(2\Z\) in \(\Q\) instead of to \(\Z\) and we work with the opposite poset structure  due to the subtraction that occurs in the conversion.
}

\added{
One benefit of working with \(h_U\) instead of \(w_K\) is that \(h_U\) is defined absolutely rather than relatively and thus lets one recover the Maslov grading \(\mu\) in absolute rather than relative terms.
Specifically,
\[\mu(U^j[K,E]):= h_U([K,E]) + |E| + j\]
Here \(\mu\) is designed to match up with the Maslov grading on \(HF^-(Y,\tf)\) and thus is also called the Maslov grading, whereas there is no analogue in Heegaard Floer homology for \(\delta\).
}


\subsubsection{Differential on the Lattice  Chain Complex} \label{latticeBouSec}

Now that we have the generators of \(\CFb(G,\tf)\) we need to define the differential.
The decomposition of \(\R^s\) into cubes presents it as a cubical complex, and up to powers of \(U\) the differential on \(\CFb(G,\tf)\) is the differential induced on this chain complex.
Up to some \(\varepsilon_{E,v}\in\{1,-1\} \) to account for orientation conventions and some non-negative integers \(a_v[K,E]\)  and \(b_v[K,E]\) to account for \added{how } \(U\) \added{changes the maslov grading}, we have

\[ \bou[K,E] = \sum_{v\in E}\varepsilon_{E,v}\left(U^{a_v[K,E]}[K,E-\{v\}] - U^{b_v[K,E]}[K+v,E-\{v\}]\right). \]

One way to establish \(\varepsilon_{E,v}\) is that if
\[E=\{v_{i_1},v_{i_2},\ldots,v_{i_n}\}\]
with \(i_j<i_{j+1}\) then \(\varepsilon_{E,v_{i_j}}=(-1)^j\).
This choice of signs has the property that if \(u\neq v\) then  \(\varepsilon_{E,v}\varepsilon_{E-\{v\},u}=-\varepsilon_{E,u}\varepsilon_{E-\{u\},v}\), which is what is needed for \(\bou^2=0\).

\added{The following choice of \(a_v[K,E]\) and \(b_v[K,E]\) are equivalent to stipulating that our differential lowers the maslov grading by one.}
\begin{equation}
a_v[K,E] :=\frac{h_U([K,E-\{v\}])-h_U([K,E])}{2} \label{avForm}
\end{equation}
and
\begin{equation}
b_v[K,E] := \frac{h_U([K+v,E-\{v\}])-h_U([K,E])}{2}.\label{bvForm}
\end{equation}
Note that our construction of \(h_U\) ensures that both \(a_v\) and \(b_v\) are in \(2\Z_{\geq 0}\) as needed for our differential to be defined.

\added{
We will now go through a collection of formulas from \cite{knotsAndLattice} which are useful in computation. 
The focus of these calculations is towards understanding \(a_v[K,E]\) and \(b_v[K,E]\), and they make use of the fact that these values only depend on \(h_U\) up to a constant that may depend on \([K,E]\).
As such analyzing the behavior of \(\frac{h_U(\cdot) - h_U(K)}{2}\) over the vertices of \([K,E]\), \([K,E]\) itself, plus the front and back faces of \([K,E]\) in a direction \(v\) yields lots of relevant information.
We can then define the following values, with the first equality giving the definition as in \cite{knotsAndLattice} and the second equality translating the statement in terms of \(h_U\):
\begin{align*}
    f([K,I]) &:= \frac{\left(\sum_{v\in I} K(v)\right) + \left(\sum_{v\in I}v\right)^2 }{2} &&=\frac{ h_U\left(K + \sum_{v\in I}v\right) - h_U(K)}{2} \\
    g([K,E]) &:= \min\{f([K,I]) \, : \, I \subseteq E \} &&= \frac{h_U([K,E]) - h_U(K)}{2}\\
    A_v[K,E] &:= \min\{f([K,I]) \, : \, I \subseteq E - \{v\} \} &&= \frac{h_U([K, E - \{v\}]) - h_U(K)}{2}\\
    B_v[K,E] &:= \min\{f([K,I]) \, : \, v \in I \subseteq E \} &&= \frac{h_U([K+v,E-\{v\}]) - h_U(K)}{2}.
\end{align*}
The first two of these are called the \emph{\(G\)-weight} of \([K,I]\) and \emph{the minimal \(G\)-weight} of \([K,E]\) in \cite{knotsAndLattice}.
The first of these follows directly from the definition of \(h_U\) and from how we defined the action of the vertices on \(\Char(G,\tf)\) by \(K+v := K+ 2 \PD[v]\), since
\begin{align*}
\frac{ h_U\left(K + \sum_{v\in I}v\right) - h_U(K)}{2} &= \frac{\left(K+ \sum_{v\in I}v\right)^2 +s - K^2 -s}{8}\\
&= \frac{\left(K+ \sum_{v\in I}2\PD[v]\right)^2-K^2}{8} \\
&= \frac{\left(\sum_{v\in I} K(v)\right) + \left(\sum_{v\in I}v\right)^2 }{2} := f([K,I]).
\end{align*}
The rest of the formulas relating these values to \(h_U\) follow from this first obvservation.
Note that \cite{knotsAndLattice} then define \(a_v= A_v[K,E]-g([K,E])\) and \(b_v=B_v[K,E]-g([K,E])\) which by the above calculations reduces to our formulas for \(a_v\) and \(b_v\) given earlier.
We will define a new value \(\Delta_v\) by
\[\Delta_v[K,E] := \frac{h_U([K+v,E-\{v\}])-h_U([K,E-\{v\}])}{2},\]
which can be computed using \(\Delta_v[K,E] = B_v[K,E] - A_v[K,E]\).
}

\added{
The formula for the \(G\)-weight in terms of \(h_U\) holds using a direct computation of the right hand side using the intersection form and the definition of the action of \(H_2(X_G;\Z)\) on \(\Char(G)\).
Each of \(g\), \(A_v\), an \(B_v\) each then take minima over respectively the whole cube, the front face in direction \(v\), and the back face in direction \(v\), which by definition then allows the term \(h_U\left(K+\sum_{v\in I}v\right)\) to be replaced with \(h_U([K,E])\), \(h_U([K,E-\{v\}])\), and \(h_U([K+v,E-\{v\}])\).
From these, we can see that \(a_v[K,E] = A_v[K,E] - g([K,E])\) and \(b_v([K,E]) = B_v[K,E] - g([K,E])\), and in fact \cite{knotsAndLattice} uses this as the definition of \(a_v\) and \(b_v\).
}

The following proposition will be useful in calculations, since we may want to understand how \(h\) acts on a cube by looking at how it acts on opposing faces, and it follows relatively directly from the definition.
\begin{proposition}\label{spherehdif}
Let \(v\) be a vertex and \(E\) is a subset of the vertices with \(v\in E\), then for cubes \([K,E]\), we have that
\[ h_U([K,E]) = \min(h_U([K,E-\{v\}]),h_U([K+v,E-\{v\}])).\]
In particular, at least one of \(a_v[K,E]\) and \(b_v[K,E]\) must be zero.
\end{proposition}

\begin{figure}
\centering

\begin{subfigure}[t]{.45\textwidth}

\includegraphics[width=\textwidth]{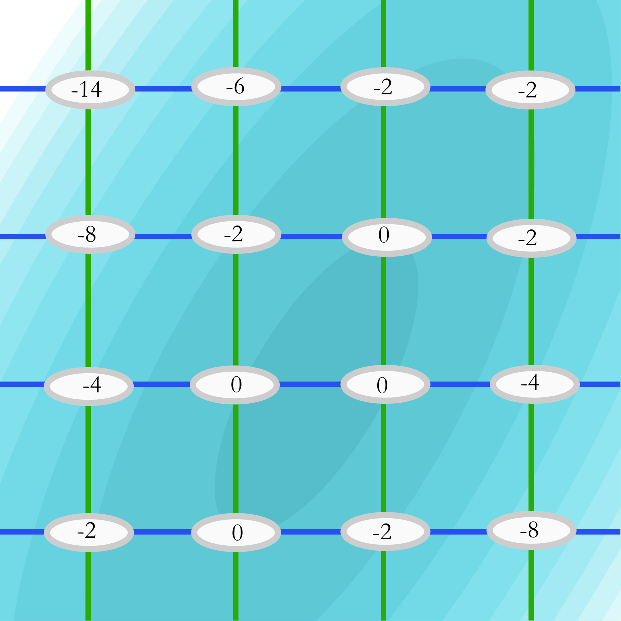}

\caption{The ellipses that generate the height function.}

\label{subfig:LatExample}
\end{subfigure}
\hfill
\begin{subfigure}[t]{.45\textwidth}
\includegraphics[width=\textwidth]{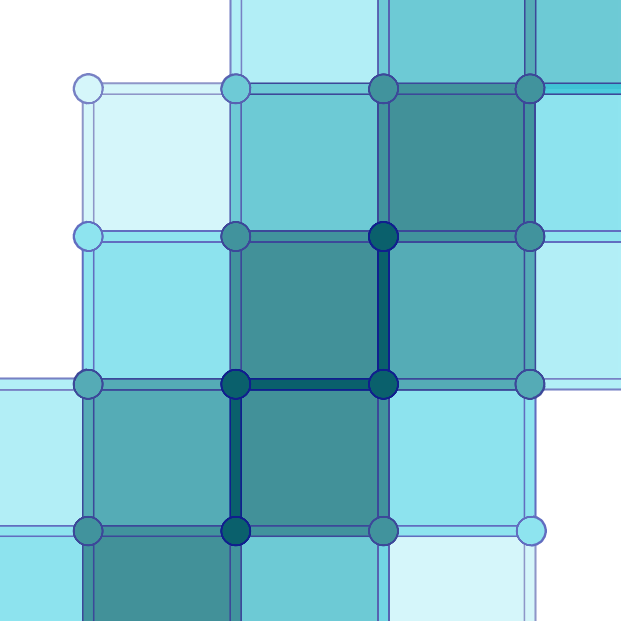}

\caption{The height function on the cubes themselves.}

\label{subfig:LatExampleCubed}
\end{subfigure}
\vspace{.3in}

\begin{subfigure}[t]{.5\textwidth}
\begin{center}
\includegraphics{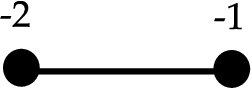}
\caption{The plumbing graph.}
\label{subfig:ExamplePlumbing}
\end{center}
\end{subfigure}

\caption{An example of the lattice homology filtration. Subfigure \ref{subfig:ExamplePlumbing} gives the plumbing for this example. Subfigure \ref{subfig:LatExample} gives the ellipses when using the quadtratic form \(\frac{K^2+2}{4}\) with the value of the height function labeled on the vertices, i.e. the characteristic cohomology classes. Subfigure \ref{subfig:LatExampleCubed} shows this height function discretized so that the height of a cube is the same height as its lowest vertex; this gives the filtration for lattice homology.
In order to make each figure readable on their own, the colors of the ellipses in Figure \ref{subfig:LatExample} are lighter than those for the filtration in figure \ref{subfig:LatExampleCubed}. However, the relative shadings to agree, and for example the darkest ellipse in Figure \ref{subfig:LatExample} and the darkest cubes in Figure \ref{subfig:LatExampleCubed} are both at height 0.
}
\label{fig:latticeExample}
\end{figure}

\begin{example}\label{ex:latticeEx}
Figure \ref{fig:latticeExample} provides an example of this calculation for the plumbing of \(S^3\) given in Figure \ref{subfig:ExamplePlumbing}.
Let \(v_1\) be the vertex with weight \(-1\) and \(v_2\) the vertex with weight \(-2\).
In the figure \(v_1\) is associated with the vertical direction and \(v_2\) with the horizontal direction.
The height function is shown on the characteristic cohomology classes, i.e. the vertices, where the upper right vertex with height 0 represents the characteristic class \(K_0\) such that \(K_0(v_1)= -1\) and \(K_0(v_2)=0\).
The background of Figure \ref{subfig:LatExample} shows the ellipses from the quadratic form \(\frac{K^2+2}{4}\).
Because we work with characteristic cohomology classes, squaring \(K\) based on how \(K\) evaluates on \(v_1\) and \(v_2\) requires looking at the inverse of the adjacency matrix for the plumbing graph.
\added{We have included the adjacency matrix \(\Gc\) for the plumbing graph and its inverse below from which one can use the upper left corner to calculate that \(K_0^2=-2\) and thus \(h_U(K_0)=0\).}
\[\Gc=\left[ \begin{array}{cc}
-1  & 1 \\
1 & -2
\end{array}\right]\quad
\Gc^{-1}= \left[ \begin{array}{cc}
-2 & -1 \\
-1 & -1
\end{array}\right]
\]

The background of Figure \ref{subfig:LatExampleCubed} colors the squares by where they sit in the filtration.
In order to make the figures individually readable, the colors in Figure \ref{subfig:LatExampleCubed} do not precisely match the ellipses in Subfigure \ref{subfig:LatExample}, but in both cases darker means the cube is higher.
One can see that the darkest squares in Figure \ref{subfig:LatExampleCubed} are precisely those that fit in the second smallest ellipse in Figure \ref{subfig:LatExample}.
There are no squares that fit inside the smallest ellipse from Figure \ref{subfig:LatExample}, and the filtration at level 0 gives a path of four vertices and three edges, which receive the darkest coloring in Figure \ref{subfig:LatExampleCubed}.
\end{example}


\subsection{The Filtered Homotopy Type}

One key observation of N\'{e}methi in \cite{latticeCohomNormSurf} is that a \(\Z[U]\)-complex \(\CFb(G,\tf)\) can be \added{interpreted as coming from a filtered cube decomposition for \(\R^s\), where \(s\) is the order of the graph \(G\), i.e. the rank of \(H_2(X_G;\Z)\).
The \(k\)-dimensional cubes are precisely the elements \([K,E]\) of \(\Qc_k(G,\tf)\).
The zero-dimensional cubes are the elements of \(\Char(G,\tf)\) themselves, and the one-dimensional cubes \([K,\{v_i\}]\) are edges connecting \(K\) and \(K+v_i\) where \(v_i\) is a base surface.
More generally for \(K \in \Char(G,\tf)\) and \(E\) a subset of the vertices of \(G\), we have that the cube \([K,E]\) is the convex hull of the points \(\{ K+\sum_{I \subset E}v_i\, : \, I \subseteq E\}\).
Here \(|E|\) is the dimension of the cube, so \([K,E] \in \Qc_{|E|}\).
We will say that if \(v \in E\) then the cube \([K,E]\) is \emph{supported in direction \(v\)}.
For each direction \(v\) that it is supported in, the cube \([K,E]\) has two faces: \([K,E-\{v\}]\), which we will call the \emph{front face of \([K,E]\) in the direction \(v\)}, and \([K +v,E-\{v\}]\), which we will call \emph{the back face of \([K,E]\) in the direction \(v\)}.}
\added{As mentioned in Section \ref{FloerSec}, we  can then define a filtered space from \(h_U\) using superlevel sets
\[X_{r} :=  \bigcup_{h_U([K,E])\geq r} [K,E] ,\]
which recovers the chain complex using cellular persistent homology.}
This approach suggests the following definition:

\begin{definition}
The \emph{lattice space \(\CFb^{\nat}(G,\tf)\)} associated to a plumbing and \Spinc structure \((G,\tf)\) is the data of the cube decomposition of Euclidean space equipped with the filtration induced by \(h_U\).
The \emph{lattice homotopy type} \(\CFb^{\nat}(Y,\tf)\) is the filtered homotopy type of a  given \(\CFb^{\nat}(G,\tf)\) where \((G,\tf)\) is a particular plumbing representation of \((Y,\tf)\).
\end{definition}


\subsection{Other Properties and Structure}

Ozsv\'{a}th, Stipsicz, and Szab\'{o} remarked upon the following equivalence in the chain complex setting \cite{knotsAndLattice}, and we now show that it can be used in the filtered setting as well.
\begin{proposition} \label{connectSums}
Let \(G_1\) and \(G_2\) be graphs and with \Spinc structures \(\tf_1\) and \(\tf_2\) respectively.
Then there exist canonical isomoprhisms 
\begin{align*}
\varphi\colon \CFb(G_1,\tf_1)\otimes \CFb(G_2,\tf_2)& \to \CFb(G_1\sqcup G_2,\tf_1\#\tf_2)\\
\varphi^{\nat}\colon \CFb^{\nat}(G_1,\tf_1)\otimes \CFb^{\nat}(G_2,\tf_2) &\to \CFb^{\nat}(G_1\sqcup G_2,\tf_1\#\tf_2),
\end{align*}
where on filtered spaces \(\otimes\) is constructed analogously to the functor described in Section \ref{FloerSec}.
\end{proposition}
\begin{proof}
Define
\[\varphi^{\infty}\colon \CFb^{\infty}(G_1,\tf_1)\otimes \CFb^{\infty}(G_2,\tf_2)\to \CFb^{\infty}(G_1\sqcup G_2,\tf_1\#\tf_2)\]
by 
\[
\varphi^{\infty}(\left([K_1,E_1]\otimes [K_2,E_2]\right):=U^{\zeta}[(K_1,K_2),E_1\sqcup E_2].
\]
This map is a bijection on the generating sets, and an inspection of the differential reveals it to be a chain map, so we will get desired map \(\varphi\) so long as the filtrations agree. 
This suggests a cellular homeomorphism \(\varphi^{\nat}\) which takes the points in \([K_1,E_1]\times [K_2,E_2]\) linearly to the corresponding point in \([(K_1,K_2),E_1\sqcup E_2]\).
\added{Now, we would like to check that this map preserves height functions, and in particular if \(h_{U,1}\) and \(h_{U,2}\) are the height functions for \((G_1,\tf_1)\) and \((G_2,\tf_2)\) respectively then,
\begin{equation}
h_U\left(\varphi\left([K_1,E_1]\otimes [K_2,E_2]\right)\right) = h_{U,1}([K_1,E_1]) +h_{U,2}([K_2,E_2]) \label{heightSum}
\end{equation}
}

We have \(h_U\) satisfies equation (\ref{heightSum}) on the vertices, since the intersection form on \(X_{G_1\sqcup G_2}\) is the direct sum of the intersection forms on \(X_{G_1}\) and \(X_{G_2}\).
Now for \([(K_1,K_2),E_1\sqcup E_2]\) we have that its vertices are \((K_1',K_2')\) where \(K_1'\) is a vertex of \([K_1,E_1]\) and \(K_2'\) is a vertex of \([K_2,E_2]\). 
So,
\begin{align*}
h_U([(K_1,K_2),E_1\sqcup E_2]) &= \min\{h_U((K_1',K_2'))\, |\, K_1'\in [K_1,E_1],\, K_2'\in [K_2,E_2]\}\\
&=\min\{h_{U,1}(K_1')+h_{U,2}(K_2')\, |\, K_1'\in [K_1,E_1],\, K_2'\in [K_2,E_2]\}\\
&= \min\{h_{U,1}(K_1')\,|\, K_1'\in [K_1,E_1]\} \\
&\quad\quad\quad\quad + \min\{h_{U,2}(K_2') \, |\, K_2 \in [K_2,E_2]\} \\
&= h_{U,1}([K_1,E_1]) + h_{U,2}([K_2,E_2])
\end{align*}
\end{proof}

\begin{remark}
This definition of \(\CFb(G,\tf)\) predisposes lattice homology \(\HFb(G,\tf)\) to many of the formal properties of Heegaard Floer homology.
It has an absolute \(\Q\) grading for torsion \Spinc structures with a relative \(\Z\) grading.
Furthermore taking coefficients in \(\Z[U,U^{-1}]\) is equivalent to letting every cube appear in every level of the filtration, and thus because \(\R^s\) is contractible we would have that \(\HFb^{\infty}_0(G,\tf) \cong \Z[U,U^{-1}]\) and \(\HFb_k^{\infty}(G,\tf)\cong 0\) for \(k>0\).
The even and odd parts of \(\HFb(G,\tf)\) are respectively \(\bigoplus_{k}\HFb_{2k}(G,\tf)\) and \(\bigoplus_{k}\HFb_{2k+1}(G,\tf)\).
Proposition \ref{connectSums} mirrors how taking connect sums of three-manifolds yields tensor products on the chain level for Heegaard Floer Homology.
\end{remark}

Finally, define a filtered chain map
\[\Ii\colon \bigoplus_{\tf\in \SpincX{Y}} \CFb(G,\tf)\to \bigoplus_{\tf\in \SpincX{Y}}\CFb(G,\tf)\]
by
\[\Ii([K,E]) := (-1)^{|E|}\left[-K- \sum_{u \in E}u,E\right].\]
Observe that on the vertices \(K \in \Char(G)\), we have \(\Ii(K) = -K\).
Because our height function is based on a quadratic form,  \(h_U(\Ii(K)) = h_U(K)\).
Since, \(h_U([K,E])\) is calculated based on the value of \(h_U\) on the vertices the calculation is thus unchanged when considering \(\Ii([K,E])\).
Thus, \(h_U([K,E]) = h_U(J([K,E]))\), meaning their are no \(U\) powers needed in our definition.
One can model \(\Ii\)'s action on \(\CFb^{\nat}(G,\tf)\) by looking at an an affine translation of negation acting on \(\R^s\).
This enables us to define a 
\[\Ii^{\nat}: \CFb^{\nat}(G,\tf)\to \CFb^{\nat}(G,\overline{\tf})\]
such that \(\shrpt{\Ii^{\nat}}= \Ii\).
This construction gives the same map as \(\iota\) on Heegaard Floer homology for almost rational graphs, which are graphs which can be made into \(L\)-spaces by lowering the weight on a single vertex \cite{involCheck}.

\section{Knot Lattice Homology: A Second Filtration} \label{knotLatticeBackSec}


Let \(G_{v_0}\) be a plumbing graph representing a generalized algebraic knot in a rational homology sphere with an unweighted vertex \(v_0\), and we will define \(G:=G_{v_0}-v_0\) to talk about the graph representing the ambient three-manifold.
Let \(\tf\) be a \Spinc structure on the plumbed three-manifold defined by \(G\).
\added{We will first construct the knot lattice chain complex associated to \((G_{v_0},\tf)\), then the doubly filtered knot lattice space.}

\subsection{Knot Lattice Chain Complex}

\subsubsection{Generators and First Maslov Grading}
Let \(\CFKb(G_{v_0},\tf)\) be freely generated over \(\Z[U,V]\) by all the cubes \added{ \(\Qc(G,\tf)\).}
As with lattice homology, we have a homological grading \(\delta\), where \(\delta\) the grading \(k\) portion \(\CFKb_{k}(G_{v_0},\tf)\) is freely generated by the elements of \(\Qc_{k}\).
\added{Additionally, the height function \(h_U\) from Section \ref{hUsubsec} carries over as well, and the first Maslov grading \(\gr_U\) is defined using the same formula as in Section \ref{hUsubsec} with the action of \(V\) preserving \(\gr_U\).}

\subsubsection{The Second Maslov Grading and Alexander Grading}

Recall from Section \ref{subsubsecgenalglinks} that the vertex \(v_0\) represents a disk in the plumbing resolution \(X\), in particular the boundary connect sum of the fibers of vertices adjacent to \(v_0\).
As a homology class in \(X\) relative its boundary \(Y\), the disk \(v_0\) can act on elements of \(\Char(G)\) as \(K+v_0 := K+ 2\PD(v_0)\), though it will only \added{preserve} orbits \(\Char(G,\tf)\) if the weak generalized algebraic knot was null-homologous in the ambient three-manifold, \added{i.e. a strong generalized algebraic knot.}
Note that this formula does not require a framing or weight on \(v_0\) as we are using the relative homology class of the disk rather than moving to a larger plumbing, and we view not having to use knowlege of a larger plumbing as a benefit of using this approach.
This action of \(v_0\) on \(\Char(G)\) can extend to an action on the cubes of \(\CFb(G)\) that translates the cubes in the direction given by \(v_0\), i.e. sending \([K,E]\) to \([K+v_0,E]\).
\added{Define a second height function \(h_V\) on the cubes by
\[h_V([K,E]) := h_U([K+v_0,E]),\]
and the second Maslov grading by
\[\gr_V(U^iV^j[K,E]) := h_V([K,E])+|E|-2j.\]}

\added{Together the two Maslov gradings induce an Alexander grading
\[A := \frac{\gr_U - \gr_V}{2} = \frac{h_U - h_V}{2}.\]}

We will now confirm that the formulation above and that given in \cite{knotsAndLattice} give the same result.
To define the Alexander grading of an element \([K,E]\), the authors of \cite{knotsAndLattice} first consider a graph \(G_{v_0}(m_0)\), where the vertex \(v_0\) is given the weight \(m_0\).
The value of \(m_0\) will wash out in our calculations, but we can assume it is chosen so \(G_{v_0}(m_0)\) is negative definite. 
This process caps off the disk \(v_0\) to a sphere.
Furthermore, \(G\) being negative definite means that \(\PD(v_0) = \PD(\Sigma_0)\), where \(\Sigma_0\) is a second homology class in the plumbing \(X\) associated to \(G\) potentially with rational coefficients.
Consider then \(\Sigma := v_0 - \Sigma_0\) in the plumbing associated to \(G_{v_0}(m_0)\).
We will also consider an extension \(L\) of \(K\) in \(G_{v_0}(m_0)\) such that
\[L(v_0) = -m_0 +2g[K,E] - 2g[K+ v_0,E]. \]
The Alexander grading is then defined to be
\[ A([K,E]) = \frac{L(\Sigma) + \Sigma^2}{2}. \]

To show that this gives the same result as our own presentation of knot lattice homology, we only need to show that this definition of the Alexander grading yields the same result as looking at half the difference of the two Maslov gradings.
In particular, letting \added{\(\Gc\) represent the adjacency matrix of the graph \(G\)} and \(v_0^{\ast}\) represent \(\PD[v_0]\) in \(H^2(X;\Z)\), we have that \(\Gc^{-1}v_0=\Sigma_0\), and
\[\Sigma^2=m_0-2\lrangle{v_0,\Gc^{-1}v_0^{\ast}} + (\Gc^{-1}v_0^{\ast})^2 = m_0 -(\Gc^{-1}v_0^{\ast})^2.\]
Therefore,
\begin{align*}
    A[K,E] &= \frac{L_{[K,E]}(\Sigma) + \Sigma^2}{2}\\
    &=\frac{L_{[K,E]}(v_0) - L_{[K,E]}(\Gc^{-1}v_0^{\ast}) + m_0 -(\Gc^{-1}v_0^{\ast})^2 }{2}\\
    &= \frac{g[K,E]-g[K+v_0,E] -K(\Gc^{-1}v_0^{\ast}) - (\Gc^{-1}v_0^{\ast})^2 }{2} \\
    &=\frac{g[K,E]-g[K+v_0,E] -\lrangle{v_0^{\ast},K}_{\Gc^{-1}} - \lrangle{v_0^{\ast},v_0^{\ast}}_{\Gc^{-1}}^2}{2}\\
    &=\frac{g[K,E] - g[K+v_0,E]}{2} - \frac{K^2 - K^2 -4\lrangle{v_0^{\ast},K}_{\Gc^{-1}} -4 \lrangle{v_0^{\ast},v_0^{\ast}}_{\Gc^{-1}}}{8}\\
    &=\frac{1}{2}\left(\frac{K^2+s}{4} +g[K,E] -\frac{(K+v_0)^2 +s}{4} - g[K+v_0,E] \right)\\
    &= \frac{1}{2}(h_U[K,E]) - h_V([K,E])).
\end{align*}

\subsubsection{Differential on the Knot Lattice Chain Complex}

\added{The differential on the knot lattice chain complex will parallel that of the lattice chain complex discussed in Section \ref{latticeBouSec} but with extra \(V\) powers inserted to keep track of \(\gr_V\).}
We define
\begin{align*}
    \bou [K,E] := &\sum_{v\in E} \varepsilon_{E,v}U^{a_v[K,E]}V^{c_v[K,E]}[K,E-\{v\}] \\
    -&\sum_{v\in E}\varepsilon_{E,v}U^{b_v[K,E]}V^{d_v[K,E]}[K+v,E-\{v\}],
\end{align*}
where \(a_v\) and \(b_v\) are exactly as in Equations (\ref{avForm}) and (\ref{bvForm}), while \(c_v\) and \(d_v\) are defined similarly but with \(h_V\) instead of \(h_U\).
In particular, \(c_v[K,E]=a_v[K+v_0,E]\) and \(d_v[K,E]\) = \(b_v[K+v_0,E]\).
The same arguments from Section \ref{introToLatticeHom} ensure that \(a_v,b_v,c_v\) and \(d_v\) are all non-negative integers.
 
\subsection{The Doubly Filtered Space}


Further paralleling the discussion in Subsection \ref{introToLatticeHom}, each height function independently gives us a filtration of superlevel sets, which are cubical subcomplexes.
Considered together they form a double-filtration using superlevel sets
\[X_{m,n} = \Span\{ [K,E]\,:\, h_U([K,E])\geq m \mbox{ and } h_V([K,E])\geq n \}.\]
As such, the \(\Z[U,V]\) module can be recovered using the persistence chain complex.

This construction suggests the following definition for a corresponding doubly-filtered topological space.
\begin{definition}
The \emph{knot lattice space} \(\CFKb^{\nat}(G_{v_0},\tf)\) associated to a plumbing with unweighted vertex \(G_{v_0}\) and \Spinc structure \(\tf\) is Euclidean space with the cube decomposition above and the double filtration induced by \(h_U\) and \(h_V\).
\end{definition}
Many of the conventions mentioned in Section \ref{introToLatticeHom} carry over here as \(\CFKb^{\nat}(G_{v_0},\tf)\) is \(\CFb^{\nat}(G,\tf)\) with another filtration added.
\added{In particular \(p_{1,!}(\CFKb^{\nat}(G_{v_0},\tf)) =\CFb(G,\tf)\) and the \(J\) defined in section \ref{introToLatticeHom} can be the map \(J\) in the definition of a complete homotopy package for our knot.}

\begin{example}
Figures \ref{fig:knotLatticeExample} and \ref{fig:knotLatticeExample2} provide examples of this construction.
Both are presented similarly to Figure \ref{fig:latticeExample} with one subfigure showing the corresponding ellipses and two others showing the discretized version of the height functions on the cubes themselves.
However, this time there are two sets of ellipses, one in blue representing \(h_U\) and one in magenta representing \(h_V\).
\added{The filtration induced by \(h_U\) is reproduced for convenience, and a new subfigures added including he filtration induced by \(h_V\) again with higher heights depicted with darker colors.}
Note that because our vertices are cohomology classes, but the arrangement of the integer lattice is based off of an action by the second homology, to calculate the translation associated to \(v_0\) one looks at the dual basis to the basis given by the base surfaces.
This can be calculated using the inverse of the matrix for the intersection form.
Having multiple examples is useful in and of itself, but it will come up again in Section \ref{ContractionSec}, where we will see how the choice of knot can affect our calculations.
\end{example}

\begin{figure}
\centering
\begin{subfigure}[b]{.45\textwidth}

\includegraphics[width=\textwidth]{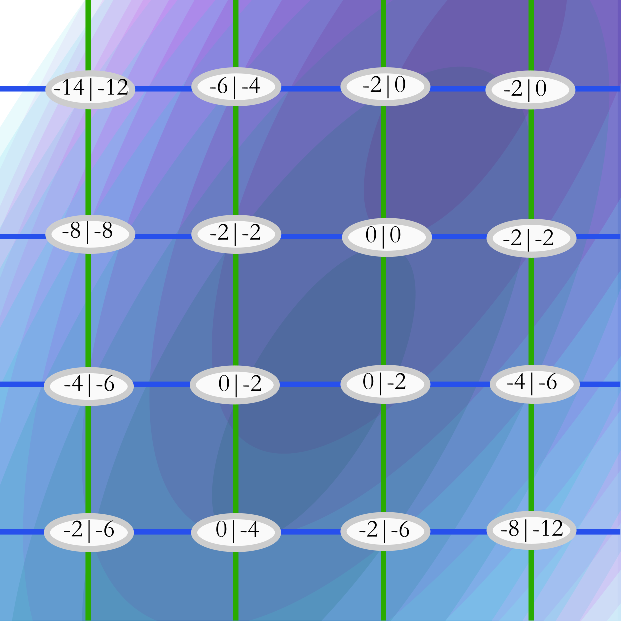}

\caption{The height function as ellipses}
\label{subfig:knotLatExample}
\end{subfigure}
\hfill
\begin{subfigure}[b]{.45\textwidth}
\includegraphics[width=\textwidth]{updatedLatticeExamplesLat.eps}
\caption{The first height function on the cubes}
\label{subfig:knotLatticeExampleCubed}
\end{subfigure}

\begin{subfigure}[b]{.45\textwidth}
\includegraphics[width=\textwidth]{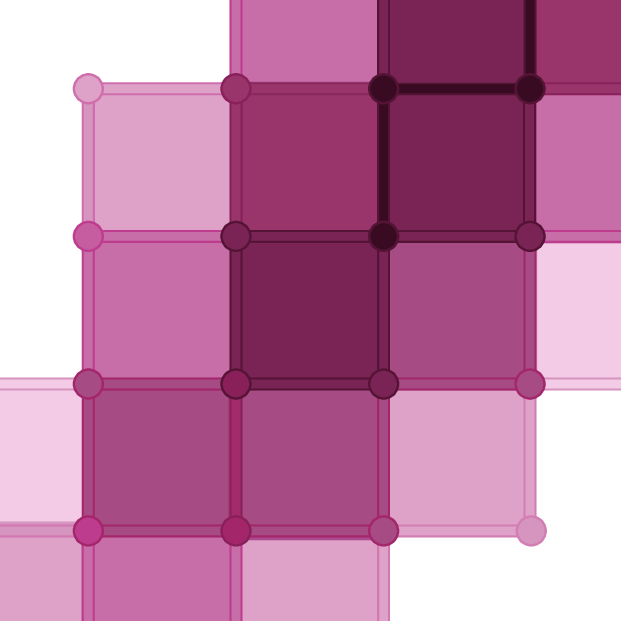}
\caption{The second height function on the cubes}
\label{subfig:knotLatticeExampleCubedRed}
\end{subfigure}
\vspace{.3in}

\begin{subfigure}[b]{.6\textwidth}
\begin{center}
\includegraphics{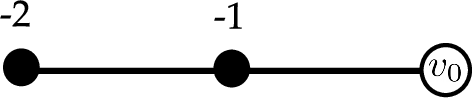}
\caption{The plumbing graph used in this examples}
\label{subfig:ExamplePlumbingKnot}
\end{center}

\end{subfigure}
\caption{A figure illustrating an example of the computation of knot lattice homology. Figure \ref{subfig:ExamplePlumbingKnot} gives the plumbing and knot used for this example. Figure \ref{subfig:knotLatExample} shows both  the original ellipses used to create this example with \(h_U\) in blue and the translated ellipses created using \(h_V\) in magenta.  These are done with transparency so that darker ellipses have higher values on the height functions and more magenta means \(h_V\) is higher and more blue means \(h_U\) is higher. Inside the vertices of Figure \ref{subfig:knotLatExample} are the values \((h_U,h_V)\). Figure \ref{subfig:knotLatticeExampleCubed} reproduces the first height function, while \ref{subfig:knotLatticeExampleCubedRed} provides the discritized second height function. Heights of less than -14 were rendered as white.}
\label{fig:knotLatticeExample}
\end{figure}

\begin{figure}
\centering
\begin{subfigure}[b]{.45\textwidth}
\includegraphics[width=\textwidth]{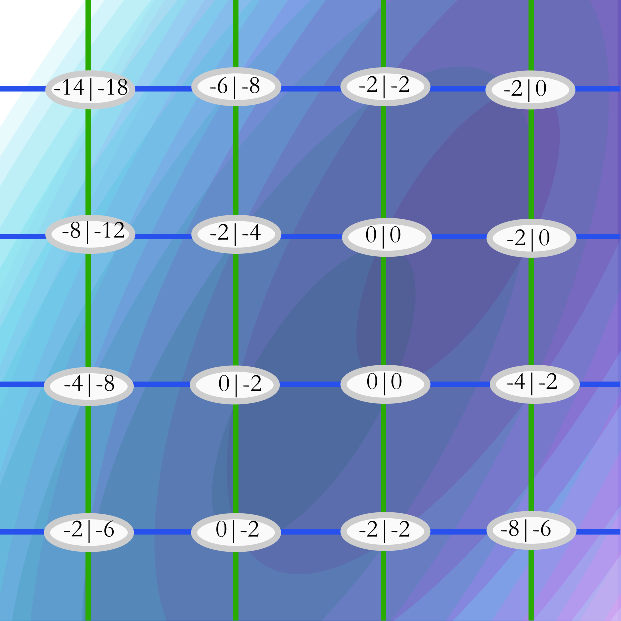}
\caption{The height function as ellipses}
\label{subfig:knotLatExample2}
\end{subfigure}
\hfill
\begin{subfigure}[b]{.45\textwidth}
\includegraphics[width=\textwidth]{updatedLatticeExamplesLat.eps}
\caption{The first height function on the cubes}
\label{subfig:knotLatExampleCubed2}
\end{subfigure}

\begin{subfigure}[b]{.45\textwidth}
\includegraphics[width=\textwidth]{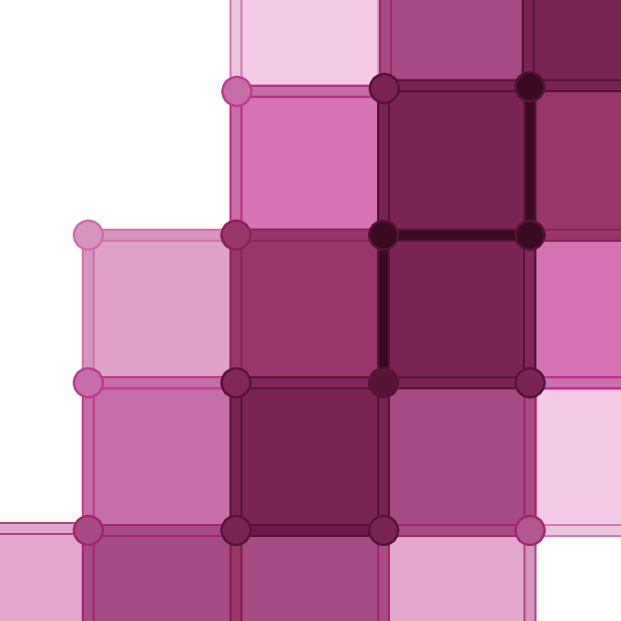}
\caption{The first height function on the cubes}
\label{subfig:knotLatExampleCubed2Red}
\end{subfigure}

\vspace{.3in}

\begin{subfigure}{.6\textwidth}
\includegraphics{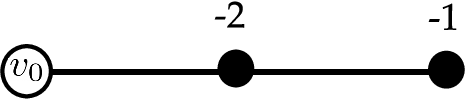}
\caption{The plumbing graph used in this examples}
\label{subfig:ExamplePlumbingKnot2}
\end{subfigure}
\caption{Another example of the double filtration for the knot lattice complex. This example has the same three-manifold as Figure \ref{fig:knotLatticeExample} but use a different unweighted vertex, as shown in \ref{subfig:ExamplePlumbingKnot2}. See Figure \ref{fig:knotLatticeExample} for a discussion of how these figures were constructed.}
\label{fig:knotLatticeExample2}
\end{figure}

\subsection{Other Properties and Structure}

The following proposition mirrors Proposition \ref{connectSums} in Subsection \ref{introToLatticeHom} and was also noticed by Ozsv\'ath, Stipsicz, and Szab\'o in the chain complex setting \cite{knotsAndLattice}.
\begin{proposition}\label{ConnectSumsKnot}
Let \(G_{1,v_0}\) and \(G_{2,v_0}\) be negative definite graphs with an unweighted vertex, with \Spinc structures \(\tf_1\) and \(\tf_2\) respectively.
Let \(G_{1,v_0}\# G_{2,v_0}\) be the graph which is the disjoint union of the graphs merged at \(v_0\).  
Then there exist canonical isomoprhisms, \added{where \(\otimes\) on doubly filtered spaces is the functor described in Section \ref{FloerSec}.}
\begin{align*}
\varphi\colon \CFb(G_{1,v_0}\tf_1)\otimes \CFb(G_{2,v_0},\tf_2)& \to \CFb(G_{1,v_0}\# G_{2,v_0},\tf_1\#\tf_2)\\
\varphi^{\nat}\colon \CFb^{\nat}(G_{1,v_0},\tf_1)\otimes \CFb^{\nat}(G_{2,v_0},\tf_2) &\to \CFb^{\nat}(G_{1,v_0}\# G_2,\tf_1\#\tf_2).
\end{align*}
\end{proposition}
\begin{proof}
Assume\(h_{U,1}\) and \(h_{V,1}\) are the height functions for \((G_1,\tf_1)\), while  \(h_{U,2},\) and \(h_{V,2} \) are the height functions on \((G_2,\tf_2)\), then we would like to show that
\begin{align}
h_U\left(\varphi\left([K_1,E_1]\otimes [K_2,E_2]\right)\right) &= h_{U,1}([K_1,E_1]) +h_{U,2}([K_2,E_2]) \label{heightSumU}\\
h_V\left(\varphi\left([K_1,E_1]\otimes [K_2,E_2]\right)\right) &= h_{V,1}([K_1,E_1]) +h_{V,2}([K_2,E_2]) \label{heightSumV}
\end{align}
The maps \(\varphi\) and \(\varphi^{\nat}\) follow the same construction as in Proposition \ref{connectSums}, and Equation (\ref{heightSumU}) becomes Equation (\ref{heightSum}) proved in that context.
What remains to be shown is Equation (\ref{heightSumV}). 
If \((K_1,K_2)\) is in \(\Char(G_{1,v_0}\# G_{2,v_0})\), then \( (K_1,K_2)+v_0\) = \((K_1+v_0,K_2+v_0)\) where each \(v_0\) is interpreted to be the unweighted vertex in the appropriate setting.
Then for a given \([(K_1,K_2),E_1\sqcup E_2]\) with \(E_1\) a subset of the vertices of \(G_{1,v_0}\) and \(E_2\) a subset of \(G_{2,v_0}\),
\begin{align*}
h_V([(K_1,K_2),E_1\sqcup E_2]) &= h_U([(K_1,K_2)+v_0,E_1\sqcup E_2]) \\
&= h_U([(K_1+v_0,K_2+v_0),E_1\sqcup E_2]) \\
&=h_{U,1}([K_1+v_0,E_1]) + h_{U,2}([K_2+v_0,E_2])\\
&= h_{V,1}([K_1,E_1]) + h_{V,2}([K_2,E_2]),
\end{align*}
which confirms equation (\ref{heightSumV}).
\end{proof}

Finally define
\[\Jj\colon \CFKb(G_{v_0},\tf) \to \CFKb(G_{v_0},\overline{\tf+\PD(v_0)})\]
on the generators by
\[\Jj([K,E]) = (-1)^{|E|}\left[-K -v_0 - \sum_{v \in E}v,E\right].\]
The map \(\Jj\) can be thought of geometrically via negation and translation (potentially between orbits \(\Char(G,\tf)\)).
Note that this can be realized in a continuous manner as
\[\Jj^{\nat}\colon \CFKb^{\nat}(G_{v_0},\tf) \to \CFKb^{\nat}(G_{v_0},\overline{\tf+\PD(v_0)}),\]
similar to \(\Ii\).
We can check that \(\Jj\) is skew \(\Z[U,V]\)-linear by checking that it is skew filtered. 
Thus observe that
\begin{align*}
    h_V\left(\Jj([K,E])\right) &= h_V\left(\left[-K-v_0 -\sum_{v \in E}v,E\right] \right)\\
    &= h_U\left(\left[-K - \sum_{v \in E}v,E\right]\right) \\
    &= h_U\left(\Ii[K,E]\right) \\
    &= h_U([K,E]) \\
    h_U\left(\Jj([K,E]) \right) &= h_U\left(\left[-K-v_0 - \sum_{v\in E}v,E\right]\right)\\
    &=h_U\left(\Ii([K+v_0,E]\right)\\
    &=h_U\left([K+v_0,E\right)\\
    &=h_V([K,E]).
\end{align*}


The definitions below summarize the constructions above.
Note that \(\CFKb^{\nat}(G_{v_0},\tf)\) contains both the information of the topological space and  the double filtration.
\begin{definition}
The \emph{complete knot lattice homology package} for the generalized algebraic knot \((Y,K)\) is the complete package 
\[\left(\bigoplus_{\tf \in \SpincX{Y}}\CFKb(G_{v_0},\tf),\Jj,\Ii\right),\]
while the \emph{complete knot lattice homotopy package} for \((Y,K)\) is the complete homotopy package
\[\left(\bigsqcup_{\tf \in \SpincX{Y}}\CFKb^{\nat}(G_{v_0},\tf),\Jj^{\nat},\Ii^{\nat}\right).\]
\end{definition}

\section{Contractable Spaces of Maps} \label{ContractionSec}

Before proving invariance and naturality, we will discuss the idea of homotopy colimits, which will subsume the approaches of quasi-fibrations used by N\'emethi \cite{anLattice, normalSurfSingBook} and contractions used by \OSS{} \cite{knotsNLatticeUnpub}.
Furthermore, while this framework subsumes the previous to approaches, it provides a more powerful way to discuss the ways in which we are identifying the different presentations of knot lattice, so that we can not only guarantee that a homotopy equivalence exists but to identify a contractible space of such maps.
This approach will also be compatible with the reduction theorems for lattice homology such as that in \cite{reductionTheorem, anLattice}.
It will also be reminiscent of the surgery formula from \cite{knotsAndLattice}.

The role homotopy colimits play here is as a way describing our space as constructed in terms of simpler parts so that we can identify the specific maps that play well with respect to that decomposition.
In particular, our goal in this section will be to provide such a deconstruction for every weighted vertex of our graph so that each piece of the deconstruction is homotopy equivalent to a filtered point, a property which we will call being \emph{subcontractible}.
The name subcontractible comes from doubly filtered points \(p\) being subterminal in doubly filtrered spaces; given a doubly filtered space \(Y\), there is at most one doubly filtered map from  \(Y\) to \(p\), but such a map might not always exist.
Similarly, the space of doubly filtered maps to a subcontractible space will be either be contractible or empty.

Subsection \ref{subsec:hocolim} provides an review of homotopy colimits and how they can be used to identify a particular nice space of maps between two homotopy colimits.
More can be found on the construction and use of homotopy colimits in \cite{catHomoTheory,primerHocolim}.
We mention conditions that guarantee which said space is contractible.
These descriptions will produce a diagram of doubly-filtered spaces and identifying the desired space of maps will be the homotopically analogous to when working with two colimits identifying specifically the maps which come from natural transformations of the underlying diagrams.
Subsection \ref{subsec:reducVertex} will show that \(\CFKb(G_{v_0},\tf)\) has such a description as a homotopy colimit so that each object in the diagram describes an orbit of \(v\) and how this orbit can be collapsed in a way guaranteeing the contractible space of maps from Subsection \ref{subsec:hocolim}.
The end of this section compares its methods to those in \cite{anLattice, normalSurfSingBook, knotsNLatticeUnpub, otherInvariance}.
In particular, this includes highlighting how homotopy colimits relate to quasi-fibrations of \cite{anLattice, normalSurfSingBook} and the contraction maps of \cite{knotsNLatticeUnpub} but also exploring generalizations of the contraction maps highlighted by the homotopy colimit approach and that such generalizations are needed.

\subsection{Homotopy colimits} \label{subsec:hocolim}

\begin{definition}
Given a small category \(D\) and functor \(F:D\to M\) for some simplicially enriched and tensored category \(M\), the \emph{homotopy colimit} of \(F\) is 
\[\hocolim F:=\coeq\left(\coprod_{\substack{\vec{d}\colon [n]\to D \\ i\colon[m]\to [n]}} \Delta^m \otimes F(\vec{d}(0))\rightrightarrows \coprod_{\vec{d}\colon [n]\to D} \Delta^n\otimes F(\vec{d}(0))\right),  \]
where \([n]\) is a categorical \(n\) simplex \(0\to 1\to ... \to n\) and \(\Delta^n\) is the \(n\)-simplex.
The two maps of the coequalizer come from considering the effect of \(i\) on \(\Delta^m\) versus considering the effect of \(i^{\ast}\) on \(F(\vec{d}(0))\)
The \emph{left deformation} of \(F\) is the diagram \(B(D,D,F)\) so that for \(d\in D\), letting \(\eta_d:D/d\to D\) be the canonical map from the overcategory at \(d\) to \(D\), 
\[B(D,D,F)(d) = \hocolim_{D/d} \eta_d^{\ast}F.\]
\end{definition}

Homotopy colimits are specific example of a wider construction known as the two-sided bar construction \(B(G,D,F)\) which takes as input a covariant functor \(F\colon\to C\) and a contravariant functor \(G\colon D\to \sSet\), where \(B(\ast,D,F)\) recovers \(\hocolim F\) and \(B(D,D,F)(d):=B(D(-,d),D,F)\) the left deformation.
Note that the left deformation comes with a canonical natural transformation \(\eta\colon B(D,D,F)\to F\), which will in fact be a point-wise homotopy equivalence (see Definition 2.2.1 and Theorem 5.1.1 of \cite{catHomoTheory}). 
Furthermore, the left deformation produces the homotopy colimit via the colimit \(\colim_D B(D,D,F) \cong \hocolim_D F\).

Morally and most importantly to our work, homotopy colimits operate as colimits do except the condition on cocones and natural transformations commuting gets replaced by homotopy coherence, i.e. commuting up to specified homotopies.
For every string of \(n\) composable morphisms starting at \(d\) that one might need a homotopy for one has thickened \(F(d)\) by \(\Delta^n\) to provide that homotopy and the identifications from the coequalizer ensure that all of these homotopies are compatible with composition and the insertion of identities.
The following definition will not only provide a way of specifying such a homotoopy coherent natural transformation but collect them in to a space.

\begin{definition}
Let \(D\) be a small category and let \(F,G\colon D\to M\) be two functors from \(D\) to some simplicially enriched, tensored, and cotensored category \(M\).
Then, \emph{the space of homotopy coherent natural transformations} from \(F\) to \(G\) is
\[C(F,D,G) := \eq \left( \prod_{[n]\to D} G(d(n))^{F(d(0))\otimes \Delta^n} \rightrightarrows \prod_{\substack{[n]\to D\\ \eta\colon [m]\to [n]}} G(d(n))^{F(d(0))\Delta^m} \right),\] 
where the equalizer is taken over the maps induced by applying \(\eta\) to \(\Delta^m\) and by mapping \(G(d(\eta(m))^F(d\eta(0))\) to \(G(d(n))^{F(d(0))}\) via precomposition by \(F(d(0))\to F(d(\eta(0)))\) and post-composition from \(G(d(\eta(m)))\to G(d(n))\).
\end{definition}

As suggested by the preceding paragraph, we do indeed have
\[C(F,D,G) \cong M^D(B(D,D,F),D,G),\]
as noted in \cite{catHomoTheory} equation (7.7.5).
Additionally, \(C(F,D,G)\) is functorial, in particular using \(\left(M^D\right)^{op}\times M^D \to M\).
Just as the homotopy colimit is specification of a two-sided bar construction \(B(-,D-)\), the mapping space is an application of the two-sided co-bar construction \(C(-,D,-)\).
Furthermore, \cite{catHomoTheory} it is noted that \(B(-,D,-)\) respects weak equivalences for a simplicaly enriched model category \(M\) if the functors are objectwise cofibrant, and the dual statement is that \(C(-,D,-)\) preserves weak equivalences between functors with fibrant objects, and in particular pointwise homotopy equivalences between functors will be taken to homotopy equivalences.
Most imediately, we have that \(C(F,D,G)\) is weakly equivalent to \(C(F,D,B(D,D,G)) \cong M^D(B(D,D,F),B(D,D,G))\), i.e. the natural transformations between the left deformations of \(F\) and \(G\).
There is a natural map then to \(M(\hocolim F,\hocolim G)\).

We now move from general to review to a focus on showing that \(C(F,D,G)\) is contractible given some condition on the objects of \(G\).
The condition needed will involve first considering a generalization of contractibility.

\begin{definition}
An object \(Y\) in a category \(M\) is \emph{subterminal} if for every object \(X\) in \(M\), \(M(X,Y)\) has at most one morphism.
An object \(Y\) in a simplicially enriched model category \(M\) is \emph{subcontractible} if it is homotopy equivalent to a subterminal object \(X\), which we will call the \emph{associated subterminal to \(Y\)}.
\end{definition}

\begin{lemma}
The subterminal objects in \(\Filt_{[\Q\times\Q\colon 2\Z\times 2\Z]}\) are precisely the doubly filtered points.
\end{lemma}
\begin{proof}
Let \(Y_{\ast\ast}\) be a subterminal object in \(\Filt_{[\Q\times\Q\colon 2\Z\times 2\Z]}\).
Given a topological space \(X\), the doubly-filtered space \(G_{n,m}(X)\) which as height \((n,m)\) everywhere will have either a single doubly-filtered map to \(Y_{\ast\ast}\) or no maps to \(Y_{\ast\ast}\).
However, such maps are in bijection with maps in \(\Top(X,Y_{n,m})\).
If \(Y_{n,m}\) is nonempty, there will always exist at least one map that factors through a point in \(Y_{n,m}\).
However, in that case there will always be exactly one map, so \(Y_{n,m}\) is a point.
For any \((n,m)\in [q_1,q_2]\), \(Y_{n,m}\) is either a point or empty, then \(Y_{\ast\ast}\) must be a doubly-filtered point.
Conversely, given any doubly-filtered point there will be at most one map into it since all filtered continuous maps are also continuous maps.
\end{proof}

\begin{proposition}\label{prop:CFDGcontract}
Given a small category \(D\) and a simplicially enriched model category \(M\), a functor \(F\colon D\to M\), and a functor \(G\colon D \to M\) which is objectwise subcontractible, \(C(F,D,G)\) will be empty if there is any \(d\in D\) so that \(M(F(d),G(d))\) is empty, and otherwise \(C(F,D,G)\) will be contractible.
\end{proposition}
\begin{proof}
If there exists any \(d\in D\) so that \(M(F(d),G(d))\) is empty, then both products in the coequalizer defining \(C(F,D,G)\) will have an empty term and thus be empty themselves.
We will now assume that \(M(F(d),G(d))\) is non-empty for every \(d\in D\).
For \(d\in D\) let \(p_d\) be the associated subtermimal to \(G(d)\).
There is a unique functor \(H\colon D \to M\) so that \(H(d)=p_d\) and a unique natural transformation \(\epsilon \colon G\to H\).
Because \(\epsilon\) is a pointwise homotopy equivalence we have that it induces a homotopy equivalence \(C(F,D,G) \to C(F,D,H)\).
However, note that because \(H\) is objectwise subterminal and for every \(d\), \(M(F(d),G(d))\) and thus \(M(F(d),H(d))\) is nonempty then there must exist a unique natural transformation \(\eta\colon F\to H\), which is in fact a unique element of \(C(F,D,H)\) by the homotopy coherent mapping space's construction.
Hence, \(C(F,D,G)\) is contractible.
\end{proof}

\subsection{Decomposition and reduction}\label{subsec:reducVertex}

For this section, \(G_{v_0}\) will always be a negative definite forest with some weighted vertex \(v\) and a single unweighed vertex \(v_0\), and \(\tf \in \SpincX{Y_{G}}\).
Let \(I\) be a subset of the weighted vertices of \(G_{v_0}\).
We will now construct a category \(D_{G,I}\) whose objects are orbits \(\OrbI{K,E}\) of cubes \([K,E]\in \Qc_{G,\tf}\) not supported in the directions \(I\) and whose morphisms recognize when the cubes in one orbit are in the boundary of the cubes of another orbit.
In particular, given \(E_1\subseteq E_2\subseteq E\) we have a morphism
\[\OrbI{K,E}\to\OrbI{K+\sum_{u \in E_1}u,E\backslash E_2}.\]
For example if \(V\) is the entire, vertex set of \(G\) then \(D_{G,V}\) is a single object with the identity morphism.

Now associated to each orbit \(\OrbI{K,E}\) there is a cube complex \(F_1(\Orb{K,E})\) that picks out all the cubes \([K,E]\) in that orbit along with cubes \([K,E\cup I_1]\) for \(I_1\subseteq I\) that project onto that orbit when forgetting any support in directions \(I\).
Essentially, choosing a point \(x\) in the middle of the cube \([K,E]\), we can look at the filtered space with points \(x+\sum_{v\in I}t_vv\) for \(t\in \R\) rather than restricting to orbits under the integers.
These arrange together to form a diagram \(F_{1}\) of doubly filtered spaces of shape \(D_{G,I}\), as a morphism in \(D_{G,I}\) represents one orbit being in the boundary of another orbit, and moving to the boundary necessarily induces a doubly filtered map.

\begin{proposition}
As a doubly filtered space \(\hocolim F_{1} \cong \CFKb^{\nat}(G_{v_0},\tf)\), where \(F_1\) is the diagram on \(D_{G,I}\) discussed above.
\end{proposition}
\begin{proof}
Let \(\tilde{I}\) represent the simplicial complex with three vertices, \(0, \frac{1}{2},\) and \(1\), an edge starting at \(\frac{1}{2}\) and going to 0, and another edge starting at \(\frac{1}{2}\) and going to 1.
This is the subdivision of the standard simplicial complex for the interval.
We can thus model a cube \([K,E]\in \Qc_{G}\) as \(\tilde{I}^{|E\backslash I|}\times [0,1]^{|I\cap E|}\), which while homeomorphic to the standard cube has a different cell decomposition as \(\tilde{I}^{|E\backslash I|}\) inherits its cell structure as simplicial complex
Doing this for all cubes in \(\CFKb^{\nat}(G_{v_0},\tf)\) with the filtration information tracked makes \(\CFKb^{\nat}(G_{v_0},\tf)\) more recognizable to the homotopy colimit definition.
In particular for every nondegenerate map from \([n]\) into \(D_{G,I}\) starting at \(\OrbI{K,E}\) and the \(F_1\left(\OrbI{K,E}\right)\) associated to \([K,E\cup I_1]\) with \(I_1\subseteq I\) contributes a cell of shape \([K,I_1]\times \Delta^n\) at the height of \([K,E\cup I_1]\) to the homotopy colimit.
This cell can directly be identified with a cell in the decomposition of the cube \([K,E\cup I_1]\) described above.
See Figure \ref{fig:DGHocolimMatch} for a visual depiction of the subdivisions for different choices of \(I\).
\end{proof}

\begin{figure}
\centering
\begin{subfigure}[b]{.3\textwidth}
\includegraphics[width=\textwidth]{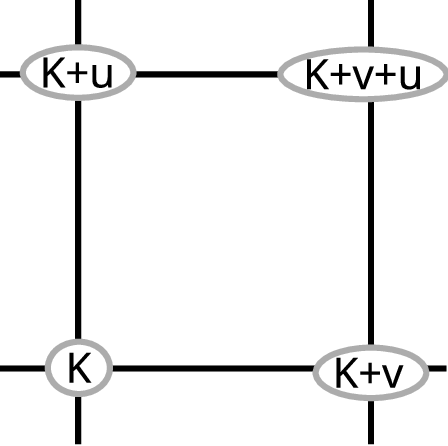}
\caption{The homotopy colimit associated to the functor \(F_1\) on \(D_{G,\{v_1,v_2\}}\).}\label{subfig:DGv1v2Hocolim}
\end{subfigure}
\hfill
\begin{subfigure}[b]{.3\textwidth}
\includegraphics[width=\textwidth]{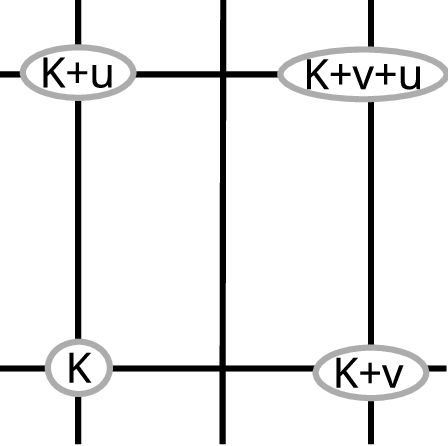}
\caption{The homotopy colimit associated to the functor \(F_1\) on \(D_{G,\{v_2\}}\).}\label{subfig:DGv2Hocolim}
\end{subfigure}
\hfill
\begin{subfigure}[b]{.3\textwidth}
\includegraphics[width=\textwidth]{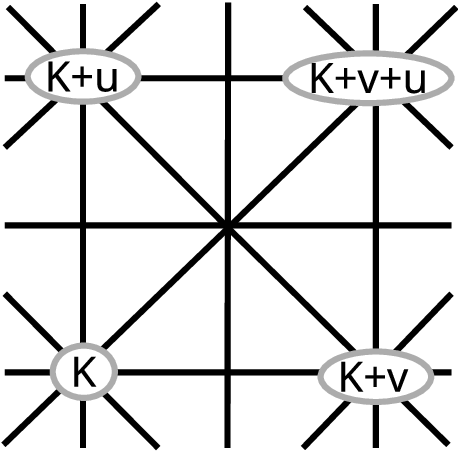}
\caption{The homotopy colimit associated to the functor \(F_1\) on \(D_{G,\{v_2\}}\).}\label{subfig:DGemptyHocolim}
\end{subfigure}
\caption{Figures depicting how the cell decomposition changes depending on whether you take the homotopy colimit of a functor sitting over \(D_{G, \{u,v\}}\), \(D_{G,\{v\}}\) or \(D_{G,\emptyset}\).
Here \(K\) is some charac
Here \(u\) represents the horizontal direction and \(v\) the vertical direction. If \(u\) and \(v\) are the only vertices of \(G\) then Figure \ref{subfig:DGv1v2Hocolim} shows that this recovers \(\CFKb^{\nat}(G_{v_0},\tf)\).
Meanwhile \(D_{G,\emptyset}\) meanwhile depicts a full subdivision of each cube, while \(D_{G,\{v\}}\) only subdivides the direction of \(u\), but when looking in the direction \(v\) the cubes remain the same.}\label{fig:DGHocolimMatch}
\end{figure}

We will now show that setting \(I=\{v\}\) for a single weighted vertex \(v\), the filtered space \(F_1(\Orb{K,E})\) is subcontractible for every \(\Orb{K,E}\).
To do so, we will show in Lemma \ref{concaveHeight} that along \(\Orb{K,E}\) the height function \(h_U\) (and thus \(h_V\)) is concave.
As such both \(h_U\) and \(h_V\) will achieve maximums on \(\Orb{K,E}\).
Next, we will show  that there must be a cube \([K',E]\) in \(\Orb{K,E}\) which maximizes both \(h_U\) and \(h_V\), before confirming that \(F_1(\Orb{K,E})\) can contract down to this point.

\begin{lemma}\label{concaveHeight}
Given a negative-definite forest \(G\), a vertex \(v\) of \(G\) and a cube \([K,E] \in \Qc\), then 
\[h_U([K,E]) >\frac{h_U([K+v,E]) +h_U([K-v,E])}{2}.\]
\end{lemma}

\begin{proof}
We start with the case of vertices, i.e. \(E=\emptyset\).
For any \(K \in \Char(G)\), the function \(f_{K,v}\colon \R \to \R\) given by \(f_{K,v}(t) := \frac{K^2 +s}{4} + tK(v) + t^2v^2\) is a concave function, since it is a parabola with negative leading coefficient.
Observe that \(h_U(K+iv)=f_{K,v}(i)\).
For an arbitrary choice of \(E\), by definition,
\[h_U([K+iv,E])=\min \{ f_{K',v}(i)\, :\, K' \mbox{ is a vertex of }[K,E] \},\]
and because minimizing over strictly concave functions gives a strictly concave function, the result follows in this case.
\end{proof}

\begin{lemma}\label{v0filteredBigBehavior}
Let \(G_{v_0}\) be a negative definite graph with unweighted vertex \(v_0\) and let be $v$ be some weighted vertx of \(G_{v_0}\). If $b_v[K,E]>0$ then $c_v[K,E]=0$. Similarly, if $a_v[K,E]>0$ then $d_v[K,E]=0$
\end{lemma}
\begin{proof}
First, we can use Proposition \ref{ConnectSumsKnot} and the definition of the boundary map on a tensor product to reduce to the case where \(v_0\) is a leaf.
We will without loss of generality consider the case where $b_v[K,E]>0$, thus implying that $B_v[K,E]>A_v[K,E]$.
Note that \(c_v[K,E]= a_v[K+v_0,E]\), so we wish to show that \(B_V[K+v_0,E]\geq A_v[K+v_0,E]\).
For every $I \subset E \cup \{v\}$, let $x_I = \sum_{w\in I}w$.
If $u \in I$, where \(u\) is the vertex to which \(v_0\) is adjacent, then 
\begin{align*}
    f([K +v_0,I])&=\frac{(K+v_0)(x_I)+ x_I^2}{2}\\
    &= \frac{K(x_{I\setminus u}) + K(u)+2 +x_I^2}{2} \\
    &= f([K,I])+1,
\end{align*}
and if $u\notin I$ then
\begin{align*}
    f([K +v_0,I])&=\frac{(K+v_0)(x_I)+ x_I^2}{2}\\
    &= \frac{K(x_{I}) +x_I^2}{2} \\
    &= f([K,I]),
\end{align*}
Therefore for any vertex \(K+x_I \) in \([K,E]\) we have that
\[h_U(K+x_I) - h_U(K) = h_V(K+x_I) - h_V(K) + \varepsilon,\]
where \(\varepsilon \in \{0,1\}\).

As such 
\begin{align*}
    h_V[K,E - \{v\}]-h_V(K)&=h_U([K,E-\{v\}])-h_U(K)+\alpha \\
    h_V([K+v,E-v]) - h_V(K) &= h_U([K +v,E-\{v\}]-h_U(K) +\beta,
\end{align*}
where $\alpha,\beta \in \{0,1\}$ depending on if the minimizing values of $f([K,I])$ necessarily had $u \in I$ or not.
Hence, $\Delta_v[K+v_0,E]\geq \Delta_v[K,E]-1\geq 0$, and thus $c_v[K,E]=0$.
\end{proof}

\begin{proposition}\label{prop:subconDecomp}
For any cube \([K,E]\) with \(v\notin E\), the doubly filtered space \(F_{1}(\Orb{K,E})\) is subcontracible.
\end{proposition}
\begin{proof}
By Lemma \ref{concaveHeight} \(h_U\) has at least one maximum.
Furthermore concavity ensures that even if there are  two maximum, they are right next to each other and thus the front and back faces of a cube supported in direction \(v\) in \(F_{1}(\Orb{K,E})\) which has the same maximal height.
Additionally, any local maxima must be a global maxima.
Furthermore \(v_0\) commutes with the action of \(v\) and thus provides a bijection between \(D_1(\Orb{K,E})\) and \(D_1(\Orb{K+v_0,E})\) and thus we also get similar statements about \(h_V\).

Lemma \ref{v0filteredBigBehavior} also guarantees that there is a cube \([K',E]\) that maximizes both \(h_U\) and \(h_V\).
Otherwise, suppose that the \(h_U\) maximizing \([K_1,E]\) does not overlap with a \(h_V\) maximizing cube \([K_2,E]\) and with out loss of generality assume that with respect to the action of \(v\) the largest \([K_1,E]\) is smaller than the smallest \([K_2,E]\).
So, \(a_v[K_1,E\cup \{v\}]>0\) reflecting \(h_U([K_2+v,E])<h_U([K_2,E])\).
However, by the concavity of \([K_2,E]\) we have 
\[h_V([K_1+v,E])-h_V([K_1,E])\geq h_V([K_2,E])-h_V([K_2-v,E])>0,\]
and thus \(d_v[K_1,E\cup\{v\}]>0\), providing our contradiction.

Now let \([K',E]\) be one such maximizing cube.
The continuous homotopy which retracts \(F_1(\Orb{K,E})\) down to a point on \([K',E ]\) will indeed be doubly filtered, since concavity prevents \(h_U\) and \(h_V\) from having local maxima not connected to \([K',E]\).
\end{proof}

\subsection{Comparisons to other work} \label{subsec:Compare}
Here we have enough information to compare with the previous work that uses quasi-fibrations and contractions.
An example of an argument given using quasi-fibrations appears in the proofs of invariance for analytic lattice homotopy \cite[Theorem 4.3.2]{anLattice} and the topological lattice homotopy type \cite[Proposition 7.3.5]{normalSurfSingBook}.
It also appears in the reduction theorems, with the case for topological lattice homotopy type given in \cite{reductionTheorem} and the analytic lattice homotopy type given in \cite[Theorem 4.6.7]{anLattice}.
The arguments in each case follow a similar pattern. 
There is an open covering of the reduced space that comes from a cube decomposition compabitble with the reduction, where for each\([K',E]\) of the reduced space one considers the interior of all cubes containing \([K',E]\).
The the preimage of each of these open sets is then subcontractible in the non-reduced space.
These not only allow for the application of a \v{C}ech covering argument on the level of homology but they satisfy Theorem 6.1.5 of \cite{ShapeThry} to show the reduction is an quasi-fibration with contractible homotopy fiber and thus a weak equivalence.

Up to the subdivision mentioned at the start of this section the preimage of each of those open sets will correspond to a potentially thickened \(B(D,D,F_2)(\Orb{K,E})\).
In particular, asking for the interiors of all cubes containing a given cube, precisely captures the overcategory of that cube, and a slice of this preimage perpendicular to the cube will be isomorphic to \(B(D,D,F_2)(\Orb{K,E})\).
The the need to be working with open sets means that the open set also has some thickness from the directions in which the cube \([K,E]\) is supported, but this does not affect subcontractibility.

Keeping track of these open covers one could potentially, reconstruct the space given by \(C(F_1,D,F_2)\) but without the overarching framework of homotopy colimits this would be more difficult and painstaking.
Without specifying the open cover, specifying contractible space of maps would be difficult if not impossible, especially since unlike with more traditional fibrations the desired maps are not necessarily sections.

For an example where these distinctions would matter consider the category with 4 objects \(a,b,c,d\) and both \(c\) and \(d\) have maps to both \(a\) and \(b\).
Consider the diagram \(F_1\) of topological spaces which as a point over \(c\) and \(d\) and the interval \([0,1]\) over \(a\) and \(b\) with the points form \(c\) mapping to \(0\) and the points form \(d\) mapping to \(1\).
Clearly there is a natural transformation \(\eta\) from \(F_1\) to the diagram \(F_2\) which as a single point as each object and corresponding maps between them.
This is shown in Figure \ref{fig:hocolimNeededDiag}.
There are no natural transformations in the reverse direction; however, as all the objects of \(F_1\) and \(F_2\) are contractible, there is a contractible space of homotopy coherent natural transformations from \(F_i\) to \(F_j\) where \(i,j\in \{1,2\}\), in particualar guarunteeing a contractible space of homotpy invereses.
Note that \(\hocolim F_1\) and \(\hocolim F_2\) both are homeomorphic to \(S^1\) with \(\hocolim \eta\) a standard homotopy equivalence between them and furthermore is a quasi-fibration.
However, \(\hocolim\eta\) has no sections and the space of all maps, even all maps that provide a homotopy inverse to \(\eta\), are not contractible due to both \(\hocolim F_1\) and \(\hocolim F_2\) being circles.

\begin{figure}
\centering
\scalebox{.7}{\includegraphics[width=.7\textwidth]{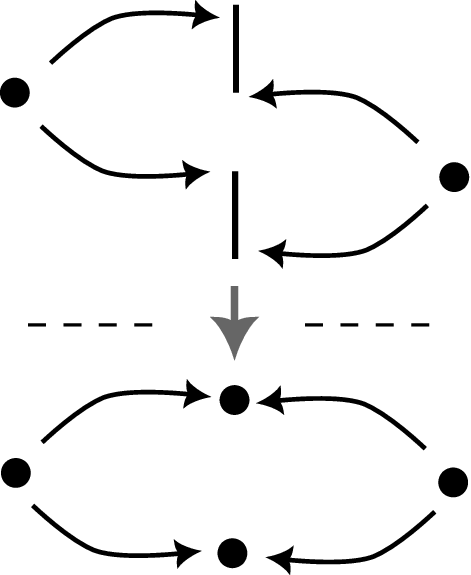}}
\caption{Figure depicting a case where the structure of the homotopy colimit is needed. Depicted are two diagrams, seperated by the dashed line, that yield circles as their homotopy colimits and a natural transformation between them. There are no natural transformations in the reverse direction.
However, the space of homotopy inverses between the homotopy colimits is not contractible.
The notion of the homotopy colimit and homotopy coherent natural transformation, however, can identify a contrible space of homotopy inverses.}\label{fig:hocolimNeededDiag}
\end{figure}

The contractions of \OSS{} in the appendix of  \cite{knotsNLatticeUnpub} (the arxiv version of \cite{knotsAndLattice}) meanwhile provide a concrete description of the map we have described for the homotopy colimit mapping each \(F_1(\Orb{K,E})\) to a specific cube in each orbit that maximizes \(h_U\) and \(h_V\), though they focus on the case when \(v\) is a good vertex, i.e. \(-v^2\) is bigger than the degree of \(v\).
The homotopy \(H_v\) from \(\id\) to \(C_v\) helps modulates the contraction of each orbit \(\Orb{K,E}\) to a point and the terms of \(\bou H_v+H_v\bou\) which are supported in direction \(v\) determine the homotopy coherence relations for a homotopy coherent natural transformation.
Reproducing the contractible space of maps without explicitly using the terminology of homotopy colimits may be technically possible with this framework, but would put more onus on tracking lots of details thus decreasing the clarity, especially once one wants to compare diagrams of subcontractible spaces where neither diagram is simply a diagram of filtered points as will be needed for more general naturality statements.

\subsection{Generalizations and choices}

We will now explore how \OSS{} choose the maximizing cubes in each \(v\)-orbit is made, as this will help highlight how said choice needs to be modified to also maximize \(h_V\) for different choices of \(v_0\).
In particular, Remark \ref{knotNeededRem} highlights how these techniques help in pinpointing the error in \cite{otherInvariance}.
Additionally, the tools to find cubes \([K,E]\) maximizing \(h_U\) and \(h_V\) in their \(v\) orbits may be useful in extending the invariance results to link lattice homology where even finer control may be needed.
Proposition \ref{parityMatters} provides an algorithm for finding a cube maximizing \(h_U\) and \(h_V\) only requiring computation of \(h_U\) and the graph theoretic distance between \(v_0\) and \(v\), while Proposition \ref{regionOfMaxs} limits the region in a \(v\)-orbit of a cube \([K,E]\) where a simultaneous maximum for \(h_V\) and \(h_U\) can occur based on the value of \(K(v)\). 
The author also used Lemma \ref{recursiveDCalc}, which is used to prove both of these propositions, in searching for the example highlighed in Remark \ref{knotNeededRem} and to handle the case of blow ups of edges before discovering the updated proof of invariance in \cite{normalSurfSingBook}.

The map \(C_v\) is defined by first defining the homotopy \(H_0\), which we will refer to as \(H_0^a\) and the map
\[C_0:=\id +\bou H_0^a+H_0^a\bou,\]
which on specific inputs will stabilize under iteration to give the contraction  \(C_v\).
Essentially \(H_0^a\) moves a cube \([K,E]\) one step closer to the given choice of maximizing \([K+iv,E]\), and iterating this proceedure will stabilize on each initial cube to give this maximizing cube plus a portion of connecting cubes need to make this a homotopy coherent natural transformation.
We reproduce this homotopy below as $\tilde{H}_0^a$.
To define \(H^a_0\) \OSS{} define a value \(T\) for each orbit \(\Orb{K,E}\), based on the behavior of $[K+2i_0v^{\ast},E]$ for the value of \(i_0\) where $(K+i_0v)(v)\in [v^2,-v^2)$.
If $a_v[K+i_0v,E\cup v]=0$, Ozsv\'ath, Stipsicz, and Szab\'o define $[K,E]$ to be of type-a and set $T=-1$, and if $a_v[K+i_0v,E\cup v]>0$ they define $[K,E]$ to be of type-b and set $T=1$.

\begin{equation*}
H_0^a([K,E]) =
\begin{cases}
0 & \mbox{if } v\in E \mbox{ or } (T-2)(-v^2)\leq K(v)<T(-v^2)\\
[K,E\cup v] & \mbox{if }v\notin E \mbox{ or } K(v)\geq T(-v^2)\\
[K-v,E\cup v] &\mbox{if }v \notin E \mbox{ or } K(v)<(T-2)(-v^2)
\end{cases}
\end{equation*}

In exploring how the maximizing cube for \(h_U\) used in the corresponding \(C_v\) is chosen, notice that Lemma 7.3 of \cite{knotsNLatticeUnpub} guarantees that for $v\notin E$ a good vertex at least one of $[K+i_0v,E]$ and $[K+(i_0+1)v,E]$ is a maximum of \(h_U\) for the \(v\) orbit of \([K,E]\). 
A closer inspection of the proof Lemma 7.3 reveals that \(a_v[K+(i_0+1)v,E]>0\), since the inequality they use to show \(B_v[K+(i_0+1)v,E]\leq A_v[K+(i_0+1)v,E]\) would need \(K(v)\leq v^2\) for equality to be achieved.
As such, we can rule out the case where both \([K + (i_0+1)v,E]\) and \([K+(i_0+2)v,E]\) are maxima, but we cannot rule out both $[K + (i_0-1)v,E]$ and $[K+i_0v,E]$ being maxima.
In that case, we can simply choose the maximum that is in $\{i_0,i_0+1\}$ as the guiding maximum for the contraction.
Then, the only problem is if both $[K+i_0v_0,E]$ and $[K+(i_0+1)v,E]$ are \(h_U\) maximizing cubes, at which point the decision to define type-a in terms of $a_v[K+i_0v,E\cup v]=0$ (instead of $b_v[K+i_0v,E\cup v]>0 $) means the preference is to treat $[K + (i_0-1)v,E]$ as the preferred maximizing cube.

We will generalize the decisions made abovbe to cases where \(v\) is not good.
In particlar, both \(C_v^a\) and \(C_v^b\) will be the result of homotopy cohrent natural transformations on the decomposition \(F_1\) on \(D_{G,\{v\}}\), which object wise are deformation retractions down to a point that is a cube maximizing \(h_U\).
The difference is that when given a choice of cubes to deformation retract down onto \(C_v^a\) will always choose a cube \([K+iv,E]\) with a higher value of \(i\), while \(C_v^b\) will always choose the a cube \([K+iv,E]\) with lower \(i\).
Note that algebraically there is only one choice of homotopy coherence relation, and thus the choice of \(h_U\)-maximizing cube for each \(\Orb{[K,E]}\) determines the map.
In terms of Figure \ref{fig:exampleContraction}, \(C_v^a\) is the one with all the question marks become up arrows, and \(C_v^b\) has all the question marks become down arrows.

To explore when \(C_v^a\) and \(C_v^b\) are filtered with respect to \(h_V\), we will first need to better track how \(h_U([K,E])\) and \(h_U([K+v,E])\) differ for a given \([K,E]\), i.e. \(2\Delta_v[K,E\cup\{v\}]\) in a way that can more easily track how the action of \(v_0\) affects them.
While the argument of  Lemma \ref{v0filteredBigBehavior} suffices to show that \(v_0\) cannot affect \(\Delta_v[K,E\cup\{v\}]\) too much, it does not give much finer control over that change. 

Recall that
\[\Delta_v[K,E]=\frac{h_U([K+v,E-\{v\}]) - h_U([K,E-\{v\}])}{2}.\]
As such in order to check when \(C_v^a\) or \(C_v^b\) is tuned to \(v_0\), we will need to check that if \(\Delta_v[K,E]=0\), then \(\Delta_v[K+v_0,E]\geq 0\) in the case of \(C_v^a\) or \(\Delta_v[K+v_0,E] \leq 0\) in the case of \(C_v^b\).

We will want to be able to compute $\Delta_v[K,E]$ recursively, and to do so will require us to define some auxiliary functions.
To start, the recursion will work graph theoretically, where the calculation at the vertex \(v\) combines results from caclulations on the neighbors of \(v\), i.e. the vertices that are adjacent to \(v\) graph theoretically, which in turn will appeal to cacluations done on vertices at distance two from \(v\), ect.
Because we assume that our graphs give rational homology three-spheres, our graphs are always forests, this process never loops back in on itself and completes in finite time.
To aid in this recursion for a vertices $v$ and $w$ of $G$, we will define $E^v_w$ to be all vertices $u \in E$ such that the unique path from $u$ to $v$ includes $w$.
The path from \(u\) to \(v\), if it exists is unique by the assumption that \(G\) is a forest.
Vertices \(u\) not on the same component of \(G\) as \(v\) will not have an affect on \(\Delta_v[K,E]\), due to the connect sum forming a tensor product.
If $u \in E$ then $u\in E_u^v$, and $E_v^v=E$.
Now define $\Delta_u^v[K,E]$ to be $\Delta_u[K,E_u^v]$ (and similarly $A_u^v[K,E]=A_u[K,E_u^v]$ and $B_u^v[K,E]=B_u[K,E_u^v]$).

\begin{lemma}\label{recursiveDCalc}
Let $u,v \in G$, and let $u_1,u_2,\ldots, u_m$ be the neighbors of $u$ in \(E\) which are farther from $v$ than $u$.
Let $\gamma^v_u[K,E]$ be the number of $u_i$ for which $\Delta_{u_i}[K,E_{u_i}^v]<0$.
Then $\Delta_u[K,E_u^v]=f[K,\{u\}]+\gamma^v_u[K,E]$.
\end{lemma}
\begin{proof}
First, we will find an expression for $B_u[K,E_u^v]$, and thus let $I \subseteq E_u^v$ with $u \in I$.
We can rewrite $I$ as $\{u\} \sqcup \bigsqcup_{u_i}I_{u_i}^v$, where the \(I_{u_i}^v\) are defined analogously to the \(E_u^v\) above.
Using this, have that 
\begin{equation*}
 f[K,I] = f[K,\{u\}]+ \sum_{u_i \in I_{u_i}^v}(f([K,I_{u_i}^v]) +1) + \sum_{u_i\notin I_{u_i}^v}f([K,I_{u_i}^v]).
\end{equation*}
For, finding the set $I$ that minimizes $f[K,I]$, we will then work on minimizing the contribution from each $I_{u_i}^v$.
If $u_i \in I_{u_i}^v$, the minimum contribution is $B_{u_i}[K,E_{u_i}^v]+1$, whereas if $u_i \notin I_{u_i}^v$ then the minimum is achieved by $A_{u_i}[K,E_{u_i}^v]$.
Let $\mathcal{Y}= \{u_i \, |\, \Delta_{u_i}[K,E_{u_i}^v]<0\}$, so that \(|\mathcal{Y}|= \gamma_u^v[K,E]\).
For \(u_i\in \mathcal{Y}\), \(B_{u_i}[K,E_{u_i}^v]-A_{u_i}[K,E_{u_i}^v]<0\) and thus the minimizing choice of \(I\) for \(f[K,I_{u_i}^v]\) comes from the calculation of \(B_{u_i}[K,E_{u_i}^v]\) and not \(A_{u_i}[K,E_{u_i}^v]\) and is decisive enough that the addition of 1 brought by the inclusion of \(u_i\) in terms contributing to \(B_{u_i}[K,E_{u_i}^v]\) does not affect the end result.
hus, bytracking which of these is the true minimum, we can thus compute that
\[ B_u[K,E_u^v]=f[K,\{u\}]+ \sum_{u_i \in \mathcal{Y}}(B_{u_i}[K,E_{u_i}^v]+1) + \sum_{u_i \notin \mathcal{Y}}A_{u_i}[K,E_{u_i}^v].\]
A similar calculation allows us to work out that
\[ A_u^v[K,E] = \sum_{u_i\in \mathcal{Y}}B_{u_i}[K,E_{u_i}^v] + \sum_{u_i\notin\mathcal{Y}} A_{u_i}[K,E_{u_i}^v]. \]
Comparing terms gets us our end result.
\end{proof}

Finally the following propositions help us determine when \added{ \(C_v^a\) and \(C_v^b\) are actually filtered with respect to \(h_V\).}
Proposition \ref{parityMatters} gives us that at least one of \(C_v^a\) and \(C_v^b\) will always work, while Proposition \ref{regionOfMaxs} provides a range in which appropriate choices of maximizing \([K,E]\) can be made in terms of \(K(v)\).
This generalizes Lemma 7.3 of \cite{knotsAndLattice} beyond the setting of good vertices, which is relevant to this paper because blowing up an edge produces a non-good vertex, while still remaining relatively controlled.

\begin{proposition}\label{parityMatters}
Suppose the graph \(G\) represents a plumbing of a rational homology sphere \(Y\).
If the parity of the distance from \(v\) to \(v_0\) is odd then \(C_v^a\) is tuned with respect to \(v_0\).
If the parity of the distance from \(v\) to \(v_0\) is even then \(C_v^b\) is tuned with respect to \(v_0\).
If \(v_0\) and \(v\) are on separate components of the graph \(G\) then both \(C_v^a\) and \(C_v^b\) are tuned with respect to \(v_0\).
\end{proposition}
\begin{proof}
To show that \(C_v^a\) is doubly filtered, we have to show that if \(\Delta_v[K,E]=0\) then \(\Delta_v[K+v_0,E]\geq 0\).
To show that \(C_v^b\) is doubly-filtered, we have to show that if \(\Delta_v[K,E]=0\) then \(\Delta_v[K+v_0,E]\leq 0\).
First to handle the case where \(v_0\) and \(v\) are on separate components, use Proposition \ref{ConnectSumsKnot} to see that \(\Delta_v [K,E] = \Delta_v[K_1,E_1]\).
Because \([K_1,E_1]\) is not affected by the action of \(v_0\), \(\Delta_v[K,E]= \Delta_v[K+v_0,E]=0\) as needed.

Now suppose that \(v_0\) and \(v\) are on the same component of \(G\).
Since \(G\) represents a rational homology sphere, we know that it is a forest and all vertices represent spheres.
In particular, there exists a unique path from \(v_0\) to \(v\), which we will write \(v_0,u_1,u_2,\ldots, u_m=v\).
We will prove more generally that for \(i\) odd that \(\Delta_{u_i}^v[K+v_0,E]\geq \Delta_{u_i}^v[K,E]\) and for \(i\) even that \(\Delta_{u_i}^v[K+v_0,E]\leq \Delta_{u_i}^v[K,E]\).
In the case where \(i=1\) this reduces to the calculation in \cite{knotsAndLattice} or equivalently using Lemma \ref{recursiveDCalc} we have that the number \(\gamma\) remains unchanged and \(f[K,\{u_1\}]\) has increased by 1.

Now assume we have shown it to be true for some \(i\).
Note that \(f[K,\{u_{i+1}\}]\) is unchanged between considering \([K,E]\) and \([K+v_0,E]\).
So, by Lemma \ref{recursiveDCalc}, we only need to consider the relation between \(\gamma_{u_{i+1}}^v[K,E]\) with \(\gamma_{u_{i+1}}^v[K+v_0,E]\), which leads to the corresponding relation between \(\Delta_{u_{i+1}}^v[K,E]\) and \(\Delta_{u_{i+1}}^v[K+v_0,E]\).
Furthermore, all such neighbors \(w\) of \(u_{i+1}\)  other than \(u_i\) will have \(\Delta_w^v[K+v_0,E]=\Delta_w^v[K,E]\), since from the perspective of \(E_w^v\), \(v_0\) is on a separate component or otherwise has no means of influencing the outcome. 
Finally, if \(i\) is even we would have by our inductive hypothesis \(\Delta_{u_i}^v[K+v_0,E]\geq \Delta_{u_i}^v[K,E]\), so \(\gamma_{u_{i+1}}^v[K,E]\geq \gamma_{u_{i+1}}^v[K+v_0,E]\).
Similarly if \(i\) is odd, we would have by our inductive hypothesis that \(\Delta_{u_i}^v[K+v_0,E]\leq \Delta_{u_i}^v[K,E]\), so \(\gamma_{u_{i+1}}^v[K+v_0,E]\geq \gamma_{u_{i+1}}^v[K,E]\).
\end{proof}

\begin{figure}
\centering
\includegraphics[width=.8\textwidth]{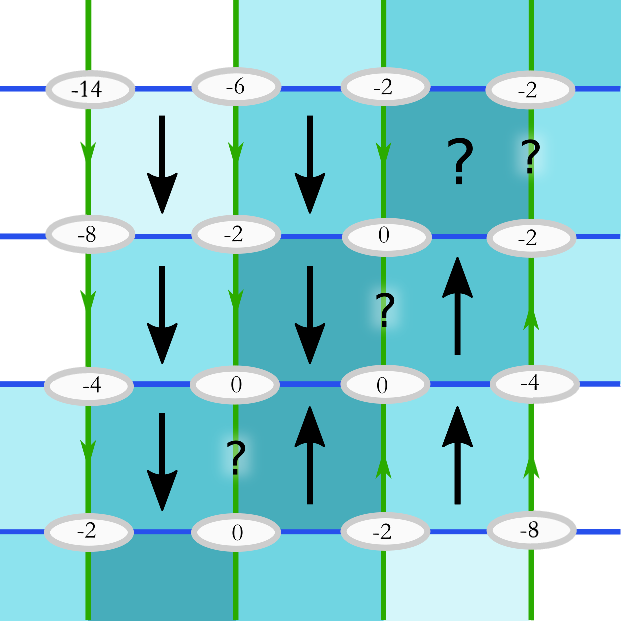}
\caption{An visualization of a contraction with respect to the \(-1\) weighted vertex in Subfigure \ref{subfig:ExamplePlumbing}, which we will call \(e\). Here a downward arrow on a cube signals that \(H_0\) sends the back face of that cube to that cube, while an upward arrow signals that \(H_0\) sends the front face of that cube to that cube or in other words for the final homotopy coherent natural trasformation of the decompositon \(F_1\) on \(D_{G,\{e\}}\), arrows represent where the deformation retraction will be moving in that direction over that cube. \added{Question marks signal \(e\)-orbits, where there are multiple .}}
\label{fig:exampleContraction}
\end{figure}

\begin{remarkn}\label{knotNeededRem}
Note that input from the knot is needed, even in the case such as in Figure \ref{fig:exampleContraction} where the vertex being contracted comes from a blow-up.
Figure \ref{fig:exampleContraction} provides a description of how Lemma \ref{v0filteredBigBehavior} already restricts the choice of maximizing elements in each orbit with arrows indicating when a homotopy would need to be moving the back face forward or the front face backward.
Question marks come from when there are two maximizing cubes and Lemma \ref{v0filteredBigBehavior} does not provide a restriction.
Recall that Figures \ref{fig:knotLatticeExample} and \ref{fig:knotLatticeExample2} both provide two different ways to add a knot to this graph.
In particular \([K_0,\{v_1,v_2\}]\) has a quesition mark, where \(K_0\) is the characteristic cohomology class described in Example \ref{ex:latticeEx}.
Checking Figure \ref{fig:knotLatticeExample}, we would have that the upper right question marks need to become up arrows, but checking Figure \ref{fig:knotLatticeExample2} we would have that the upper right question marks need to become down arrows.
The chain homotopy provided by Jackson \cite[Lemma 2.5]{otherInvariance} does not depend on the choice of unweighted vertex and would not be doubly filtered for the example in Figure \ref{fig:knotLatticeExample2}.

Lemma \ref{recursiveDCalc} was used in locating the particular problem cube by realizing that we needed \(\Delta_{v_1}[K_0,E]\)=0, and if a knot adjacent to \(v_2\) were to affect this result we would need, by Lemma \ref{recursiveDCalc},
\[\Delta_{v_2}[K_0,E\backslash\{v_1\}]=f[K,\{v_2\}]= \frac{K(v_2)-2}{2}=-1.\]
This in turn forces in the calculation of \(\Delta_{v_1}[K,E]\) so  that
\[-1=f[K,\{v_1\}]=\frac{K(v_1)-1}{2}\]
and hence \(K\) must be \(K_0\).
\end{remarkn}

\begin{proposition}\label{regionOfMaxs}
Let \(v\) be a vertex with degree \(d\) and let \(\eta = \max\{v^2+d,0\}\).
Furthermore, let \(v_0\) be an unweighted vertex.
One can always select choice \added{of cubes \([K,E]\) for each \(v\)-orbit that maximize both \(h_U\) and \(h_V\)} so that 
\begin{equation}
2v^2-2\eta \leq K(v)< -v^2. \label{Kbounds}
\end{equation}
With \(v_0\) not adjacent to \(v\), we can further assume
\[2v^2-2\eta<K(v)<-v^2,\] provided such a \(K\) exists in the \(v\)-orbit of \([K',E]\).
\end{proposition}
\begin{proof}
Observe that in each \(v\) orbit there will be a \(K\) (not necessarily maximizing either \(h_U\) or \(h_V\) that satisfies the conditions of Equation (\ref{Kbounds}).
So, we show separately that in searching for a maximizing cube, that cubes below the target range always have heights \(h_U\) and \(h_V\) less than or equal to that of the cube above them.
Similarly, a cube chosen above the target range will always have heights less than or equal to that of the cube below it.
As such, due to the convexity of \(h_U\) and \(h_V\) when searching for maximizing cubes, we always have the option of moving into the desired target range. 
To show that we can assume that \(K(v)< -v^2\), note that Lemma 7.3 of \cite{knotsAndLattice} gives that if \(K(v)\geq -v^2\) and \(v\in E\) then \(a_v[K,E]=0\)
Furthermore, if \(K(v)\geq -v^2\) then \((K+v_0)(v)\geq -v^2 \), so \(a_v[K+v_0,E]=0\) as well.
As such, we have both
\begin{align*}
h_U([K+v,E])&\geq h_U([K,E])\\
h_V([K+v,E])&\geq h_V([K,E]),
\end{align*}
so so may assume that a choice of cube \([K,E]\) maximizing \(h_U\) and \(h_V\) on their orbit has \(K(v)<-v^2\).

Now suppose that \(v_0\) is adjacent to \(v\), \(K(v)<2v^2 -2\eta\) and that \(v\in E\). 
Then, \(f[K-v,\{v\}] < v^2 - \eta\) and \(f[K-v+v_0,\{v\}] \leq v^2 -\eta\).
Furthermore,
\[\gamma_v[K-v,E] \leq |E|\leq d\leq -v^2+\eta,\]
and the same argument gives \(\gamma_v[K+v_0,E]\leq -v^2+\eta\)
Therefore, by Lemma \ref{recursiveDCalc}, we have that \(\Delta_v[K-v,E]<0\) and \(\Delta_v[K-v+v_0,E]\leq 0\).
As such,
\begin{align*}
h_U([K-v,E])&\geq h_U([K,E])\\
h_V([K-v,E])&\geq h_V([K,E]),
\end{align*}
and we may assume that the choice of maximizing cube in the orbit has \(K(v)\geq 2v^2-2\eta\).
If \(v_0\) is not adjacent to \(v\) then \(f[K,\{v\}]\) does not change between \([K,E]\) and \([K+v_0,E]\), and as such, we merely need to assume that \(K(v)\leq 2v^2 -2\eta\) to get the same result.
\end{proof}

\section{Invariance and Naturality}\label{sec:InvarNat}

\subsection{Invariance under Blow Ups}\label{subsec:BlowUps}

In this section we will specifically be showing that knot lattice homology and homotopy are invariant under blow-ups, which will complete the proofs of Theorems \ref{invarianceResult} and \ref{topInvarianceResult}.
Recall from Theorem \ref{BlowUpsSuffice} that the topological knot type determines the nested singularity type, so we only need to consider the five moves shown in Figure \ref{fig:trefblowup}:
\begin{enumerate}
\item generic blow-ups,
\item blow-ups of a weighted vertex,
\item blow-ups of the unweighted vertex,
\item blow-ups of edges connecting two weighted vertices, and
\item blow-ups of edges incident to the unweighted vertex.
\end{enumerate}

We will first cover some constructions and lemmas that hold in all of the cases above.
These  lemmas \added{remove} the dependence on knowledge about \(v_0\), and thus we will conclude the section highlighting how the proofs of invariance for the lattice homotopy of the underlying three-manifold to lift/specialize to the case of knots. 
\added{Because these use the approach of Section \ref{ContractionSec}, by the comments in subsection \ref{subsec:Compare} these will also specifically recover previous proofs of invariance such as \OSS's proof for the blow up of edges incident to the unweighted vertex or the proofs of N\'emethi for the underlying three-manifold invariant.
Furthermore, by laying out such a general strategy, we hope to potentially provide tools useful in other proofs, for example, in extending these results to links.}
For a concrete example, we will then discuss the case of blowing up the unweighted vertex.


Let \(G_{v_0}\) be a negative-definite forest with unweighted vertex, and let \(G_{v_0}'\) be some blow-up of \(G_{v_0}\) at a point  \(p\).
Note that \(p\) can be perturbed slightly to give a generic blow-up \(G_{v_0}''\). 
Let \(X\) be the plumbing for \(G_{v_0}\), \(X'\) the plumbing for \(G_{v_0}'\), and \(X''\) the plumbing for \(G_{v_0}''\).
Both \(X'\) and \(X''\) can be achieved by adding a two-handle to \(X\) representing the exceptional sphere \(e\).
Additionally, \(X'\) and \(X''\) are diffeomorphic, and in fact, if the blow-up was already generic or a blow-up of the unweighted vertex, their handle decompositions are the same.
Otherwise \(X'\) and \(X''\) are related by a handle slide or two of existing handles over the new handle, which become the new edges connecting \(e\) to the rest of the graph.

Now because the intersection form for \(X''\) is the intersection form for \(X\) direct sum \([-1]\), there is a canonical projection of characteristic cohomology classes from \(\Char(G'')\) to \(\Char(G)\) which forgets evaluation on \(e\).
Furthermore, the handle slide from \(X'\) to \(X''\) provides a map identifying \(\Char(G')\) and \(\Char(G'')\).
Let \(\tilde{P}:\Char(G')\to \Char(G)\) be the composition of these maps.
The map \(\tilde{P}\) preserves the orbits given by the action of the second homology and thus provides a bijection between \(\SpincX{Y_G}\) and \(\SpincX{Y_{G'}}\) where \(Y_G\) and \(Y_{G'}\) are the associated three-manifolds.
As such, we will be using \(\tf\) to represent a \Spinc structure on \(Y_G\) or \(Y_{G'}\) under this identification.
So, given a particular \(\tf\), we get that \(\tilde{P}\) restricts to a map from \(\Char(G',\tf)\) to \(\Char(G,\tf)\).

In more singularity theoretic language, this is equivalent to defining \(\tilde{P}\) based off of how the canonical map from a blow-up to the original resolution affects homology.
To get a map on the second cohomology, the map sending a curve to its total transform preserves the intersection form and goes in the appropriate direction to induce the needed cohomological map.
The change of basis induced by the handle slides play the same role as changing between the total and proper transforms.

The map \(\tilde{P}\) can be extended to a continuous \added{(but as of yet potentially not filtered)} map
\[P^{\nat}:\CFKb^{\nat}(G_{v_0}',\tf)\to \CFKb^{\nat}(G_{v_0},\tf),\]
by reproducing the above identifications using coefficients in \(\R\) rather than \(\Z\).
The handle slides provide a change of basis for the second homology and thus a change in cellular structure, so the map from \(\Char(G')\to \Char(G'')\) is not necessarily cellular.
However, the change in basis only involved handle slides over the new handle for \(e\), which results in a sheering action on the previous cube decomposition in the direction of \(e\).
During the projection to \(\Char(G)\) the action of this sheering disappears and cubes in \(\Char(G')\) are sent to cubes in \(\Char(G)\).
As such \(P^{\nat}\) is a cellular map and gives a bigraded chain map
\[P: \CFKb(G_{v_0}',\tf)\otimes \Z[U,U^{-1},V,V^{-1}] \to \CFKb(G_{v_0},\tf)\otimes \Z[U,U^{-1},V,V^{-1}].\]
Note \(\CFKb^{\infty}\) represents taking coefficients localized at \(U\), which will have a morphism for every not-necessarily filtered cellular map.
The name has been chosen deliberately to match the map \(P\) given in the appendix of \cite{knotsAndLattice}.

\added{
This construction can in fact be generalized to give a functor from the category of resolutions mentioned in Section \ref{SingBackSec} to the cell complexes and cellular maps, which takes a resolution to the lattice space and an analytic map \(\phi\) between resolutions to a map \(P_{\phi}^{\nat}\).
The maps \(P_{\phi}^{\nat}\) can be defined directly using the total transforms defined above.
However, because each \(\phi\) can be decomposed into a sequence of blow ups, to check that \(P_{\phi}\) is filtered or doubly filtered only requires checking that it is filtered when \(\phi\) represents a blow up.
}
\begin{lemma}\label{PsingleFilt}
The map \(P^{\nat}\) commutes with the actions of \(\Ii^{\nat}\) and \(\Jj^{\nat}\), and in particular is doubly-filtered if it is filtered with respect to \(h_U\).
\end{lemma}
\begin{proof}
\added{
To show \(P^{\nat}\) commutes with the desired maps it will suffice to show that it commutes for the characteristic cohomology classes \(K \in \Char(G)\).
In particular, both of of \(\Ii^{\nat}\) and \(\Jj^{\nat}\) can be viewed as extending the action from the vertices linearly, with cubes as the convex hulls of their vertices, and \(P\) also respects this linear action.
}

\added{
Given a weighted vertex \(v_i\) in \(G_{v_0}\), let \(\bar{v}_i\) represent the total transform of \(v_i\) in \(G'_{v_0}\).
Then \(\lrangle{v_0,v_i}=\lrangle{v_0,\bar{v}_i}\) regardless of what point \(p\) was blown up.
In particular, if \(p\) is not at an intersection of \(v_i\) and \(v_0\) any previous intersections will remain and no new intersections will be created.
Otherwise, the proper transforms of \(v_0\) and \(v_i\) will not intersect, but the total transform of \(v_i\) will contain a copy of \(e\) which now intersects \(v_0\).
Since \(\lrangle{v_0,v_i}=\lrangle{v_0,\bar{v}_i}\), the action of \(v_0\)  must commute with \(P\).
Furthermore, \(-P(K)(v_i) = -K(\bar{v}_i)=(-K)(\bar{v}_i)\), which verifies \(P\) commuting with \(\Ii\).
Because \(\Jj\) is the compositoin of \(\Ii\) and the action of \(v_0\), \(P\) also commutes with \(\Jj\). 
}

If \(P^{\nat}\) is singly filtered, then it must be that for every \([K,E] \in \Qc_{G'}\), \(h_U(P([K,E]))\geq h_U([K,E])\).

However, in that case,
\begin{align*}
h_V(P([K,E]))&= h_U(P([K,E])+v_0) \\
&=h_U(P([K+v_0,E])) \\
&\geq h_U([K+v_0,E])\\
&=h_V([K,E]),
\end{align*}
and thus \(P^{\nat}\) actually induces a map with coefficients in \(\Z[U,V]\).
Furthermore, the image under \(P^{\nat}\) of a cube \([K,E]\) with \(e\notin E\) is a single cube in \(\Qc_{G}\), not just algebraically but on the nose.
As such, the calculation of the height function for cubes where \(P([K,E])\) is algebrically non-zero translates to the check for whether \(P^{\nat}\) is doubly-filtered.
If \(e\in E\), then \(P^{\nat}([K,E])\) is the same as \(P^{\nat}\) of \([K,E]\)'s front and back faces in direction \(e\).
Because \(P^{\nat}\) is filtered on both the front and back faces of \([K,E]\) and \(h_U([K,E]\) is the minimum of \(h_U\) across the front and back faces, then \(P^{\nat}\) must be filtered on the points in \([K,E]\) as well.
\end{proof}

\added{The key observation needed to apply the framework of Section \ref{ContractionSec} is that} the map \(P^{\nat}\) is onto and the preimage of any given point is an \(e\)-orbit in \(\CFKb^{\nat}(G_{v_0}',\tf)\).
\added{For a cube \([K,E]\) in \(\CFKb^{\nat}(G_{v_0}',\tf))\) let \(F_2(\Orbe{K,E})\) be a point in the center of \(P([K,E])\).
This will define a functor \(F_2\colon D_{G,e}\to \Filt_{[\Q\times \Q:2\Z\times 2\Z]}\), so that \(\hocolim F_2 \cong \CFKb^{\nat}(G_{v_0},\tf)\).
The map \(P^{\nat}\) is induced by a natural transformation \(\eta_{P}\colon F_1\to  F_2\), where \(F_1\) is the functor discussed in Section \ref{ContractionSec} for \(G_{v_0}\).
The next key step is then to show that there exists a homotopy coherent natural transformation from \(F_2\) to \(F_1\) to provide the homotopy inverse.}

\begin{lemma}\label{lemma:singleFtoDoubleF}
Let \(G_{v_0}\) be a negative definite forest with an unweighted vertex \(v_0\) and let \(G_{v_0}'\) be a blow up of \(G_{v_0}\).
Additionally, let \(\tf \in \SpincX{Y_G}\).
If \(P\) provides a singly filtered map
\[P^{\nat}\colon \CFb^{\nat}(G', \tf)\to \CFb^{\nat}(G,\tf)\]
and there exists a (not necessarily continuous) \(\phi\colon \Qc_{G,\tf}\to \Qc_{G',\tf}\) providing a section for \(P\) which is filtered with respect to \(h_U\), then
\[P^{\nat}\colon \CFKb^{\nat}(G_{v_0}',\tf) \to \CFKb^{\nat}(G_{v_0},\tf) \] 
has a homotopy inverse \(R^{\nat}\).
\end{lemma}
\begin{proof}
By Lemma \ref{PsingleFilt} \(P\) inducing a map on \(\CFb(G,\tf)\) implies that \(P^{\nat}\) is doubly filtered, and the only thing left to check is that it has a homotopy inverse.
To do so, we will need to upgrade \(\phi\) to a doubly filtered section \(\tilde{\phi}\).
Because we have already shown that each orbit of the exceptional sphere \(e\) in \(\CFKb^{\nat}(G_{v_0},\tf)\), and thus every preimage under \(P\) of a cube \([K,E]\in \Qc_{G,\tf}\) will have a cube \([K',E]\) that maximizes both \(h_U\) and \(h_V\) and define \(\tilde{\phi}([K,E])=[K',E]\)
 \[h_U([K',E])\geq h_U(\phi(P([K',E])))\geq h_U(P([K',E]))=h_U([K,E]).\]
To carry out a similar argument, it will suffice to provide a section \(\phi'\colon \Qc_{G,\tf}\to \Qc_{G',\tf}\) which is filtered with respect to \(h_V\) but not necessarily \(h_U\).
By defining \(\phi'([K,E])= \phi([K+v_0,E])-v_0\) for \([K,E]\in \Qc_{G,\tf}\) we then have
\begin{align*}
h_V(\tilde{\phi}([K,E]) &\geq h_V(\phi'([K,E]))\\
&=h_V(\phi([K+v_0,E])-v_0) \\
&=h_U(\phi([K+v_0,E])) \\
&\geq h_U([K+v_0,E]) \\
&= h_V([K,E]),
\end{align*}
as needed.

The section \(\tilde{\phi}\) guarantees that for every \(e\) orbit \(\Orbe{[K,E]}\), there is a filtered map from \(F_2\left(\Orbe{K,E}\right)\) to \(F_1\left(\Orbe{K,E}\right)\).
So by Proposition \ref{prop:CFDGcontract}, the space \(C(F_2,D,F_1)\) and thus
\[\Filt_{[\Q\times\Q:2\Z\times2\Z]}^{D_{G',e}}(B(D,D,F_2),B(D,D,F_1))\]
is contractible and in particular nonempty.
Let \(\eta_R\) be such a natural transformation.
Furthermore, Proposition \ref{prop:CFDGcontract} can also be applied to show that for each \(F_i\) the space of self natural transformations of \(B(D,D,F_i)\), which contains the identity on \(F_i\) and the composition of \(\eta_P\) with \(\eta_R\) in the appropriate order, is contractible.
Hence, \(\eta_P\) and \(\eta_R\) are doubly filtered homotopy inverses.

\end{proof}

Finally we can provide the proof for Theorem \ref{topInvarianceResult} and thus also Theorem \ref{invarianceResult}.

\begin{proof}
By Lemma \ref{lemma:singleFtoDoubleF} it suffices to show for each type of blow up that \(P^{\nat}\) is singly filtered and has a singly filtered (not necessarily continuous) section.
Note that these are both statements about the lattice spaces of the ambient three-manifold and in fact for blow ups of weighted vertices and edges N\'emethi proves those facts in \cite{normalSurfSingBook}, where there \(P^{\nat}\) is referred to as \(\pi_{\R,\ast}\) with \(\ast\) signifying a particular layer of the filtration.
Showing that \(\pi_{\R,\ast}\) is a well defined map is the proof that it is filtered, and the proof that it is onto provides the proof that it has a section.
Additionally Lemma \ref{PsingleFilt} guarantees that \(P^{\nat}\) commutes with \(\Ii\), \(\Jj\) and \(\Gamma\) and thus any homotopy inverse will commute with these up to homotopy.
N\'emethi assumes the graph is connected and thus generic blow ups and blow ups of the unweighted vertex still remain.

Both of these cases can be expressed as taking the connect sum with a different version of the unknot and thus can be reduced to showing that copy of the unknot produces the same result as the graph with a single unweighted vertex.
Ozsv\'ath, Stipsicz, and Szab\'o already showed  these to be chain homotopy equivalent to \(\Z[U,V]_{0,0}\), where the subscript represents the element 1 being in bigrading \((0,0)\) \cite{knotsAndLattice}.
However, we will review the case of blowing up the unweighted vertex to show that this can be done in the topological setting and to provide a specific concrete example of how Lemma \ref{lemma:singleFtoDoubleF} applies.
The proof for generic blow-ups follows similarly to this one.

Let \(G_{v_0}\) be the graph for the unknot given by a single unweighted vertex, and let \(G'_{v_0}\) represent this graph blown up at the unweighted vertex, with the new vertex having label \(e\).
The knot lattice space associated to \(G_{v_0}\) is a point \(K\) with \(h_U(K)=0\) and \(h_V(K)=0\), while the knot lattice space associated to \(G'_{v_0}\) is shown in Figure \ref{fig:blownUpUnknot}.
The map \(P^{\nat}\) is the unique map from \(\CFKb^{\nat}(G_{v_0}',\tf)\) to \(\CFKb^{\nat}(G_{v_0},\tf)\).
This map is doubly-filtered by Lemma \ref{PsingleFilt} and one can observe directly that \(h_U([K,E])\leq 0\) and \(h_V([K,E])\leq 0\) for \([K,E]\) a cube of \(\CFKb^{\nat}(G_{v_0}',\tf)\).
Letting \(K_0\) represent the characteristic cohomology class on \(G'\) with \(K_0(e)=-1\), observe that \(h_U(K_0)=0\) and \(h_V(K_0)=0\), i.e. \(K_0\) maximizes \(h_U\) and \(h_V\).
We can define a doubly filtered section by \(\phi(K)=K_0\).

Note that \(D_{G',e}\) consists of only a single object with no nontrivial morphisms.
As such in this case the homotopy inverse is \(\phi\) since no higher homotopy coherent relations are needed.
In particular the homotopy from Proposition \ref{prop:subconDecomp} suffices as a homotopy between  \(\eta_P\) and \(\eta_R\) and the identity.
\begin{figure}
\centering
\begin{subfigure}[b]{.9\textwidth}
\includegraphics[width=\textwidth]{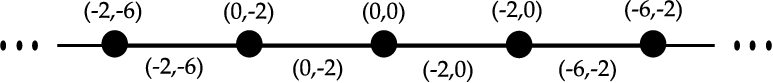}
\caption{The knot lattice space for a blown up unknot. The heights for the cells are labeled with \((h_U,h_V)\) with the vertices labeled above and the edges labeled below.} \label{subfig:blownUpUnknotklhomotopy}
\end{subfigure}
\vspace{.3in}

\begin{subfigure}[b]{.6\textwidth}
\begin{center}
\includegraphics{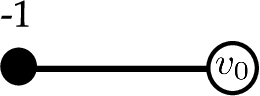}
\caption{the graph of the blown up unknot}\label{subfig:blownUpUnknotGraph}
\end{center}
\end{subfigure}
\caption{Subfigure \ref{subfig:blownUpUnknotklhomotopy} gives the knot lattice space \(\CFKb^{\nat}(G_{v_0}',\tf)\)f or the case of the blown up unknot used in the proof that knot lattice homotopy is invariant under blowing up the unweighted vertex. The graph used to create this example is given in Subfigure \ref{subfig:blownUpUnknotGraph}.}\label{fig:blownUpUnknot}
\end{figure}

Furthermore, \(\Ii^{\nat}\) and \(\Jj^{\nat}\) are given by the identity on \(\CFKb^{\nat}(G_{v_0},\tf)\), while on \(\CFKb^{\nat}(G_{v_0}',\tf)\), \(\Ii^{\nat}\) is given by reflection across the middle of the edge to the left of \(K_0\), and \(\Jj^{\nat}\) is given by reflection across the vertex \(K_0\).
Lemma \ref{PsingleFilt} has that \(P^{\nat}\) commutes with both.
Meanwhile, \(\phi\) commutes with \(\Jj^{\nat}\), and commutes with \(\Ii^{\nat}\) up to the homotopy relating \(K_0\) with \(K_0-e\).
\end{proof}

\subsection{Naturality for graphs}\label{subsec:Naturality}

The following builds on ideas from \cite{naturalityHF} to discuss the naturality of this construction, i.e. the level to which this guarantees a unique identification between the lattice homologies of any two graphs.
At the moment there is a risk that, for example, a collection of blow ups and blow downs that starts and ends at the same graph will result in a string of homotopy equivalences that together to become a self homotopy equivalence not homotopic to the identity.
One way to express this would be to produce a functor on a category equivalent to \(C_{\GenAlgKnot}^{\Spin^c}\) but with objects corresponding to resolutions, as the well definedness of the functor would ensure that compositions do not produce unexpected monodromy.
To capture that we have only produced homotopy equivalences and higher order homotopies between our maps, our notion of equivalent will allow a morphism in \(C_{\GenAlgKnot}^{\Spin^c}\) to be associated to a contractible space in our new category.

\begin{lemma}
There is a functor defined up to unique natural isomorphism 
\[\Phi_{\min}\colon C_{\GenAlgKnot}^{\Spin^c}\to \Filt_{[\Q\times\Q:2\Z\times 2\Z]},\]
which sends \((Y,K,\tf)\) to \(\CFKb^{\nat}(G_{v_0},\tf)\), where \(G_{v_0}\) is a dual graph for the minimal resolution of the weak nested singularity associated to \((Y,K)\).
All morphisms \(\Phi_{\min}(f)\) commute with \(\Ii\) and \(\Jj\).
\end{lemma}
\begin{proof}
This functor will be defined up to unique isomorphism due to minimal resolutions being defined up to unique isomorphism.
Let \((Y_1,L_1,\tf_1)\) and \((Y_2,L_2,\tf_2)\) be weak generalized algebraic knots with \Spinc structures on the ambient three-manifolds and \(f\colon (Y_1,L_1,\tf_1)\to (Y_2,L_2,\tf_2)\) be a graph induced diffeomorphism which pulls back \(\tf_2\) to \(\tf_1\).
Let \(\tilde{X}_i\) be a minimal resolution for the weak nested singularity with associated dual graph \(G_{i,v_0}\) associated to \((Y_i,L_i)\) and let \(\tilde{f}\colon  \tilde{X}_1\to \tilde{X}_2\) be an extension of \(f\) across the minimal resolutions.
From \(\left(\tilde{f}^{-1}\right)^{\ast}\), we get an isomorphism between the \(H^2(\tilde{X_i};\R)\), and for each base oriented surface \([S_i]\) in \(G_{1,v_0}\) we have that there exists a base surface \([S_j']\) in \(G_{2,v_0}\) so that \(\tilde{f}_{\ast}([S_i])= \pm[S_j]\).
The map \(\tilde{f}_{\ast}\) respects the action of \(H_2(\tilde{X_i};\Z)\) so cubes \([K,E]\in \Qc_{G_1,\tf}\) are sent to cubes in \(\Qc_{G_2,\tf}\) in a continuous fashion.
The squaring of characteristic cohomology classes is preserved leads to \(h_U\) being preserved, and since the total transforms of the underlying curves \(C_i\) are sent to each other, the action of \(v_0\) is preserved and thus \(h_V\) is preserved.
This ensures that the identifications of the \(H^2(\tilde{X_i};\R)\) induces a doubly-filtered continuous map
\[\tilde{f}_{lat}\colon \CFKb(G_{1,v_0},\tf)\to \CFKb(G_{2,v_0},\tf)\]
Because \(\left(\tilde{f}^{-1}\right)^{\ast}\) is functorial with respect to \(\tilde{f}\), the result is functorial.
\end{proof}

\begin{theorem}
There exists a topologically enriched category \(\tilde{C}_{\GenAlgKnot}^{\Spin^c}\) which 
\begin{enumerate}
\item has objects \((\pi,\tf)\) where \(\pi\colon \tilde{S}\to S\) is a resolution of a weak nested singularity \((S,C,\{p_1,\ldots, p_n\})\) with link a generalized algebraic knot in a rational homology three-sphere \((Y,K)\), and \(\tf \in \SpincX{Y}\).
\item has morphism spaces \(\tilde{C}_{\GenAlgKnot}^{\Spin^c}((\pi_1,\tf_1),(\pi_2,\tf_2))\) that have a contractible component for each morphism in \(C_{\GenAlgKnot}^{\Spin^c}((Y_1,K_1,\tf_1),(Y_2,K_2,\tf_2))\) between the weak generalized algebraic knots.
\end{enumerate}
and there is a functor
\[\CFKb^{\nat}\colon \tilde{C}_{\GenAlgKnot}^{\Spin^c}\to \Filt_{[\Q\times\Q\colon 2\Z\times2\Z]}\]
which takes a resolution \((\pi,\tf)\) to \(\CFKb^{\nat}(G_{v_0},\tf)\) where \(G_{v_0}\) is the dual graph for \(\pi\).
On minimal resolutions \(\CFKb^{\nat}\) agrees with the functor \(F_{min}\).
\end{theorem}
\begin{proof}
Let \(\pi_1\) and \(\pi_2\) be some resolutions of weak nested singularities with dual graphs \(G_{1,v_0}\) and \(G_{2,v_0}\) and let \(\pi_1'\) and \(\pi_2'\) be minimal resolutions for their respective singularities with dual graphs \(G_{1,v_0}^{\min}\) and \(G_{2,v_0}^{\min}\).
Furthermore let \(f\in C_{\GenAlgKnot}^{\Spin^c}((Y_1,K_1,\tf_1),(Y_2,K_2,\tf_2))\) where \((Y_i,K_i)\) is the link of singularity for the weak nested singularity associated to \(\pi_i\).
We will now construct the contractible space associated to \(f\) in \(\tilde{C}_{\GenAlgKnot}^{\Spin^c}((Y_1,K_1,\tf_1),(Y_2,K_2,\tf_2))\).

Because minimal resolutions are terminal in the category of resolutions associated to a weak nested singularity, there exist unique morphisms \(\phi_1\colon \pi_1\to \pi_1'\) and \(\phi_2\colon \pi_2\to \pi_2'\).
Let \(E_i\) represent the set of vertices in \(G_i\) that are in the exceptional fiber of \(\phi_i\).

Then, as with the case where the \(\phi_i\) are blow ups and blow downs, the objects of \(D_{G_i,E_i}\) can be associated to cubes in \(\Qc_{G_i^{\min}}\) with morphisms indicating boundary relations, as \(P_{\phi_i}\) induces an isomoprhism \(\hat{P}_{\phi_i}\) between \(D_{G_i,E_i}\) and \(D_{G_i^{\min},\emptyset}\).
Now let \(F_{i}\) be the decomposition  of \(\CFKb(G_{i,v_0},\tf)\) associated to \(D_{G_i,E_i}\), which we can view as a functor on \(D_{G_i^{\min},\emptyset}\), and let \(F_{i}^{\min}\) be the decomposition of \(\CFKb(G_{i,v_0}^{\min},\tf)\) associated to \(D_{G_i^{\min},\emptyset}\).
Furthermore the isomorphism given by \(\Phi_{\min}(f)\) will induce an isomoprhism
\[\hat{\Phi}_{\min}(f)\colon D_{G_1^{\min},\emptyset}\to D_{G_2^{\min},\emptyset}.\]

We will define our morphism space associated to \(f\) to be
\[\Filt^{D_{G_1^{\min},\emptyset}}_{[\Q:2\Z]}(B(D_{G_1^{\min},\emptyset},D_{G_1^{\min},\emptyset},F_{1}),B(D_{G_1^{\min},\emptyset},D_{G_1^{\min},\emptyset},\hat{\Phi}_{\min}(f)^{\ast}(F_{2})).\]
This produces a topologically enriched category and is essentially a subcategory of the Grothendieck construction.
Taking homotopy colimits over the \(D_{G_i^{\min},\emptyset}\) functorially recovers \(\CFKb^{\nat}(G_{i,v_0},\tf)\). 
That this will not depend on the choice of minimal resolution since \(\Phi_{\min}\) is defined up to a canonical isomorphism based on said choice of minimal resolution.
What remains to show is that this morphism space associated to \(f\) is contractible. 
By proposition \ref{prop:CFDGcontract} it suffices to show that for all \([K,E]\) that \(F_{1}([K,E])\) and \(F_{2}([K,E])\) are subcontractible and there exists a morphism from \(F_{1}([K,E])\) to \(F_{2}([K,E])\).

We will first show that for every cube \([K,E] \in \Qc_{G_1^{\min}}\), both \(F_{i}([K,E])\) and \(\hat{\Phi}_{\min}(f)^{\ast}\left(F_{i}^{\min}([K,E])\right)\) will have doubly-filtered homotopy equivalent results.
Factoring the \(\phi_i\) into a sequence of blow ups \(\phi_{i,j}\colon \tilde{\pi}_{i,j}\to \tilde{\pi}_{i,j+1}\) provides a sequence of functors \(F_{i,j}\) on \(D_{G_i^{\min},\emptyset}\) so that \(F_{i,j}\) and \(F_{i,j+1}\) are objectise homotopy equivalent.
In particular when showing that \(\phi_{i,j}\) provides a homotopy equivalence between \(\CFKb^{\nat}(G_{i,j+1,v_0},\tf)\) and \(\CFKb(G_{i,j,v_0},\tf)\), one shows that the functors \(F'_{i,j}\) and \(F'_{i,j+1}\) on the larger category \(D_{G_{i,j+1},\emptyset}\) evaluate on objects to give homotopy equivalent subcontractible objects.
Using the unique map \(\tilde{\phi}_{i,j}\) from \(\tilde{\pi}_{i,j+1}\) to our chosen minimal resolutions \(\pi_i'\), we can use the action of \(P_{\tilde{\phi}_{i,j}}\) to construct functors from \(D_{G_{i,j+1},\emptyset}\) to \(D_{G_i^{\min},\emptyset}\).
We can transfer the objectwise homotopy equivalences between \(F'_{i,j}\) and \(F'_{i,j+1}\) to functors over \(D_{G_i^{\min},\emptyset}\) by taking homotopy colimits across the fibers of this functor.
These will remain objectwise homotopy equivalences in the process, and chaining these together gives the doubly fitlered homotopy equivalences between \(F_{i}([K,E])\) and \(F_{i}^{\min}([K,E])\).

Note now that for all \([K,E]\in \Qc_{G_1^{\min}}\), \(F_{i}^{\min}([K,E])\) is a doubly-filtered point, and by the map \(\Phi_{\min}(f)\) being an isomorphism, \(F_{1}^{\min}([K,E])\) and \(F_{i}^{\min}([K,E])\) have the same double filtration.
As such, \(F_{1}([K,E])\) and \(\hat{\Phi}_{\min}(f)^{\ast}\left(F_{2}([K,E])\right)\) are both subcontractible and there exist morphisms going both directions between them.
\end{proof}

\bibliographystyle{abbrv}

\bibliography{references}

\end{document}